\documentclass[11pt, generic, amscd,amssymb,verbatim]{amsart}

\newtheorem{thm}{Theorem}
\newtheorem{Cor}{Corollary}

\newtheorem{Prop}{Proposition}

\numberwithin{equation}{section}
\numberwithin{Cor}{Prop}
\numberwithin{Prop}{section}
\newtheorem{Rem}{Remark}
\numberwithin{Rem}{section}
\newtheorem{Def}{Definition}
\numberwithin{Def}{section}

\usepackage{ifpdf}
\usepackage{amssymb}
\usepackage{amsmath}
\begin{document}
\title[Schwarzians as moments and the adapted complexification]{Higher order Schwarzians for geodesic flows, moment sequences, and  the radius of adapted complexifications}
\author{Raul M. Aguilar}
\maketitle
\centerline{Science and Mathematics Department}
\centerline{Massachusetts Maritime Academy}
\centerline{Buzzards Bay, MA, 02532}
\centerline{e-mail: raguilar@maritime.edu}
\centerline{March 24, 2011}
\section{Abstract}
In the first part of the paper, comprising section 1 through 6,
we introduce a sequence of
functions in the tangent bundle $TM$
 of any smooth
two-dimensional manifold $M$ with smooth Riemannian metric
$g$
that correspond to the higher order Schwarzians
of the linearized geodesic flow. With these functions and a classical theorem of Loewner on analytic continuation we are able to characterize  the existence of the adapted
complex structure induced by $g$ on the set $T^RM$ of vectors in $TM$ of length up to $R$, equivalently for $M$ compact, to the existence of a Grauert tube of radius $R$ in terms of infinite Hankel matrices involving these Schwarzian functions.
The basic characterization so obtained can be expressed
as a sequence of differential inequalities of increasing order
polynomial in the covariant
derivatives of the Gauss curvature  on $M$ and in $\pi\,R^{-1}$ that
should be regarded as  the higher order versions of
a curvature  inequality by L. Lempert and  R. Sz\H{o}ke.

The second part of the paper, sections 7 through 11, includes a discussion of the rank of the infinite Hankel matrix of the Schwarzians from part 1
and of new Schwarzians defined now for purely imaginary radius,  as well as some computations and examples. A characterization of the existence of the adapted structure on $T^RM$
in terms of moment sequences with parameters $R$ and $\mathrm{v}$ in $TM$ is also noted.
\tableofcontents
\section{Introduction and description of main results}
Let $M$ be a two-dimensional real analytic complete  Riemannian manifold with
metric $g$ and Gauss curvature $\sigma$, and let its tangent bundle be denoted by  $\pi\colon TM\to M$\,.

L. Lempert and R. Sz\H{o}ke \cite{LESZ} defined
for any real analytic metric $g$
the \emph{adapted complex structure}, which  always exists in some neighborhood of the zero section $M\subset TM$.
An equivalent construction is defined in the cotangent bundle by V. Guillemin and M. Stenzel \cite{GLLST}
(Please see beginning of Section
\ref{SECTION-GT} for definitions).

A prototype for these structures is found in the study of the geometry of certain foliations in complexifications of symmetric spaces of rank one by G. Patrizio and P.M. Wong \cite{PW}, which in turn was motivated
by the work on the complex homogeneous Monge-Amp\'{e}re equation by D. Burns \cite{BurnsMA} and that of W. Stoll \cite{STOLL}.

There is current interest in the adapted complex structure,
and from various points of view \cite{HALL-KIRWIN} \cite{SU-JEN-KAN}

We have concentrated our attention in how the geometry of $(M,g)$
determines the set where the adapted complex structure can be defined.
 One aspect of this is the determination of
the  possible values of $R>0$ for which
\begin{equation}\label{TRM-0}
T^RM\,=\,\{\mathrm{v} \in TM\,\\, \|\mathrm{v}\|=\sqrt{g(\mathrm{v},\mathrm{v})}\,<\,R\}\,
\end{equation}
admits the adapted complex structure, equivalently, for $M$ compact, the possible radii $R$ the Grauert tube asssociated to  $(M,g)$ can have.

There are two inequalities that  serve as our  motivation
and starting point.

The first one is due to
Lempert and Sz\H{o}ke \cite{LESZ} who showed that  if
the adapted complex structure is defined
on $T^RM$,  for some $R>0$ finite or infinite, the
inequality
\begin{equation}\label{LSZIR}\fbox{$\displaystyle
\sigma\,\geq \,-\frac{1}{4}\frac{\pi^2}{R^2}\,$}
\end{equation}
must hold. The second one applies to the
  case  $R\,=\,\infty$ and is due to Sz\H{o}ke \cite{SZ-THESIS} who proved
  that
  in  addition to (\ref{LSZIR}), if the adapted structure is defined on all $TM$,
 the inequality
 \begin{equation}\label{SZI}
 \Delta\sigma\,\geq\,-16 \,\sigma^2\,,
  \end{equation}
  with $\Delta$ the Laplacian operator for the metric $g$ on $M$, must
hold as well.

Clearly, the curvature condition alone
is far from sufficient to guarantee existence of the structure on $T^RM$. This is illustrated for $R=\infty$ by the case of a torus \cite{LESZ}, by
Sz\H{o}ke result on surfaces of revolution \cite{SZOKE-SR},
and as shown by the author, by the triaxial ellipsoid \cite{RMAEllipsoid} and more generally Liouville metrics on the two-sphere \cite{RMA-QJM-09}.

So, we would like to understand if there are properties of the curvature that when met would guarantee the existence of the adapted structure. This question of course applies to $M$ of any dimension, but as a first step we limit ourselves to dimension of $M=2$; so ``curvature'' means Gauss curvature.

  The two necessary conditions above are a consequence
of the non-negativity of all the derivatives of
odd order of functions
in the Pick class, also known as Nevanlinna or Herglotz functions, that is,
functions analytic and with positive imaginary part on the
upper half-plane,
\begin{equation}\label{H+0}
H^+\,:=\,\{z\in \mathbb{C}\,|\, 0\,<\,\Im z\},
\end{equation}
 which are
defined across some interval of the boundary of $H^+$. Higher order necessary conditions are of course obtainable, as
the computations of Lempert and Sz\H{o}ke make clear.

In this paper we propose a framework to view all those higher order conditions, and provide \emph{necessary and sufficient}
conditions for the adapted complex structure to be defined on $T^RM$
that depend on the Gauss curvature $\sigma$
of $M$\,. This is possible because our set up is based on
properties of derivatives which
give a characterization of functions in the Pick class according to  Loewner's Theorem \cite{Bendat-Sherman}, \cite{DONOGHUE}.

We give a basic set of conditions in the form of an infinite
sequence of differential inequalities of
increasingly higher order in $\sigma$, viewed as function in $TM$.

Some of our results will apply to sectional curvature in higher dimensions for some situations, giving necessary conditions. However obtaining necessary and sufficient conditions requires the curvature operator.

In (\ref{INTRO-S-R-TM-vs}) we display our basic
inequalities in terms
of certain
functions in $TM$, the ``higher order Schwarzians in TM",  that we
introduce in Section \ref{SECTION-SCHWAR-CURV}. These
inequalities are preceded by
those in
(\ref{INTRO-DET-R-INEQ})  that we derive
in Section \ref{SECT-HOSCHWARZ} in terms of the
higher order Schwarzians on the real line $\mathbb{R}$.

For convenience we introduce the  following.

\begin{Def} Let $R\in (0,\,\infty]$ be fixed.
We denote by $\mathfrak{L}_R$ the class of smooth functions
$\vartheta\colon \mathbb{R}\to \mathbb{R}$ for which a quotient, and hence \emph{all} quotients (see comment at the beginning of the Proof of  Theorem \ref{TH-EXTEND-FINITE-R})
of independent solutions of
\begin{equation}\label{EQ-INTRO}
y^{\prime\prime}(x)+\vartheta(x)\,y(x)\,=\,0
\end{equation}
have meromorphic extensions to the $R$-strip
$$\{z\,\in \mathbb{C}\,,|\Im z|\,<\,R\,\}$$
that are holomorphic and  with non-vanishing imaginary part for $0\,<\,\Im z\,<\,R$.
\end{Def}

Thus, we begin by seeking conditions on $\vartheta$ to be in $\mathfrak{L}_R$.
It is well-known, and can be checked by a computation,  that
$\vartheta$ is related to
the quotient  $h=y_2/y_1$ of any pair of
independent solutions of (\ref{EQ-INTRO}),
$y_1$ and $y_2$,
   by the equality valid off the zeros of $y_1$,
\begin{equation}\label{classical-SCHW}
2\, \vartheta\,=\,
 \frac{{h}^{\prime\prime\prime}}
 {{h}^{\prime}}-
\frac{3}{2}
\left(\frac{{h}^{\prime\prime}}
           {{h}^{\prime}}
           \right)^2 \,,
\end{equation}
the right-hand side of which is the \emph{(classical) Schwarzian} of $h$.
For high order conditions on $\vartheta$ to be in $\mathfrak{L}_R$
we are led to  consider  the \emph{higher order Schwawrzians}
\begin{equation}
\mathcal{S}^{\vartheta}_n\colon\mathbb{R}\to \mathbb{R},
\end{equation} for  $n\,=\,1, \,2,\,\cdots$,
defined by the generating function (see also (\ref{n-Schwarzian}))
\begin{eqnarray} \label{INTRO-S-R}
\frac{  h(x+t)-h(x)}
{h^\prime(x)+ \frac{1}{2}\,\frac{h^{\prime\prime}(x)}{h^\prime(x)}\left(
h(x+t)-h(x) \right)}\,
=\,\sum_{n\,=\,1}^\infty\,
\mathcal{S}^{\,\vartheta}_n(x)\, \frac{t^n}{n!}\,.
\end{eqnarray}
\begin{Rem}
By the Fa\`{a} di Bruno formula,
\begin{eqnarray}\notag
\frac{\mathcal{S}^\vartheta_n(x)}{n!}=
\hskip-1.6cm\sum^n_{\begin{array}{c}k=1\\k_1+\cdots+k_n=k\\
k_1+ 2k_2\cdots+nk_n=n\end{array}}\hskip-1cm
(-1)^{k+1}\frac{k!}{ k_1!\cdots k_n!}
\left(\frac{h^{(1)}(x)}{1!}\right)^{k_1-2k+1}
\left(\frac{h^{(2)}(x)}{2!}\right)^{k_2+k-1}\cdots
\left(\frac{h^{(n)}(x)}{n!}\right)^{k_n}\,.
\end{eqnarray}
always provided that $h^\prime(x)\neq 0$, otherwise use $1/h(x)$\,, in light of Remark \ref{MOEBIUS-INV}.
\end{Rem}
\begin{Rem}
We  get that $\mathcal{S}^{\,\vartheta}_0\,=\,0$,\,
$\mathcal{S}^{\vartheta}_1\,=\,1$, \,${S}^{\,\vartheta}_2\,\equiv\, 0$,
 while $\mathcal{S}^{\,\vartheta}_3$ equals the right-hand side of
(\ref{classical-SCHW}). Moreover (see Section \ref{computations})
\begin{eqnarray}\label{shift}
\mathcal{S}^{\,\vartheta}_4=
\frac{h^{\prime\prime\prime\prime}}{h^{\prime}}
-
4
\frac{h^{\prime\prime}\;
h^{\prime\prime\prime}}{ \left(h^{\prime}\right)^{2}}
+
3
\left(\frac{h^{\prime\prime}}{h^{\prime}}\right)^{3}
\,=\,2\,\vartheta^\prime\,.
\end{eqnarray}

These functions are treated by H. Tamanoi \cite{Tamanoi},
and more recently by S. Kim and T. Sugawa \cite{KS}.
Note that our  indexing for the Schwarzians
is shifted by one unit from theirs.
The following recursion, hinted at in (\ref{shift}),  gives an alternative definition
 (Proposition 2.1 \cite{KS}),
\begin{equation}\label{HOS-REC-FORMULA}
 {\mathcal{S}}^{\,\vartheta}_{n+1}(x)
 \,=\,\left({\mathcal{S}}^{\,\vartheta }_{n}(x)\right)^\prime\;+\;\vartheta(x)\,
  \sum_{k=1}^{n-1}\,
  \binom{n}{k}\;{\mathcal{S}}^{\,\vartheta}_{k}(x)\;
  {\mathcal{S}}^{\,\vartheta}_{n-k}(x)\;.
  \end{equation}
  \end{Rem}
\begin{Rem}\label{FFS}
  In  Proposition
   \ref{Def-FPS}, we provide another very useful expression for these functions
   that can be interpreted  as their "normal coordinate" definition; in fact that is the formula most used
   throughout this paper.
 \end{Rem}
By definition,
for $R=\infty$ the function $\vartheta$ is in $\mathfrak{L}_R$
if and only if any quotient $h$ of two independent solutions (\ref{EQ-INTRO}) is a Pick function defined also across intervals in $\mathbb{R}=\partial H^+$. These functions are classically
characterized
by the non-negativity of an infinite quadratic form
defined by their derivatives arranged in a Hankel matrix of  infinite order.
This leads to  inequalities that are
determinantal of increasing order in
the Schwarzians, and thus, to
canonical higher order conditions for
$\vartheta$.

We treat the case $R<\infty$ simultaneously  with the case $R\,=\,\infty$
by the introduction of
 higher order ``$g_R$-Schwarzians"
$${\mathcal{S}}^{\,\vartheta,\,g_R}_{n}\colon\mathbb{R}\to\mathbb{R}$$
given by Definition \ref{HORS}.
These are linear combinations
of the Schwarzians, as  displayed in (\ref{Lin-Comb}),
with coefficients polynomial in $R$
that depend on a choice of   a Pick function $g_R$ selected
so that in the limit  $R\,=\,\infty$
 the $g_R$-Schwarzians become the higher
 order Schwarzians.
 \begin{Rem}\label{REM-GR}
 To illustrate, with $\displaystyle g_R(z)\,=\,-\frac{R}{\pi}\,\ln \left(1-\frac{\pi\,z}{R}\right)$,
we now list (see Section \ref{computations}),
\begin{equation}\label{INTRO-First-R-Schwarzians}
{\mathcal{S}}^{\,\vartheta, \,g_R}_1
\,=\,1\,,\;\;\;
{{\mathcal{S}}^{\,\vartheta, \,g_R}_2}\,=\,
\frac{\pi}{R}\,,
\;\;\;{\mathcal{S}}^{\,\vartheta, \,g_R}_3\,=
\,{\mathcal{S}}^{\,\vartheta}_3\,+\, \frac{2\,\pi^2}{R^{2}}
\,=\,
2\,\vartheta\,+\, \frac{2\,\pi^2}{R^{2}}\,.
\end{equation}
\end{Rem}
The  main results are the following.

\vskip.2cm
  \textit{Theorem \ref{TH-EXTEND-FINITE-R}} shows that, for $R$ finite or infinite, $\vartheta$ is in $\mathfrak{L}_R$ if and only
if for all integer $n \geq 1$ and all $x\in \mathbb{R}$\,,
\begin{equation}\label{INTRO-DET-R-INEQ-1}
{\mathcal{D}}^{\,\vartheta, \,g_R}_n(x)\,:
=\,\det\left[\frac{{\mathcal{S}}^{\,\vartheta, \,g_R}_{i+j-1}
(x)}{(i+j-1)!}\right]_{\,i,\,j\,=\,1}^{n} \,\geq \,0\,.
\end{equation}
Furthermore, assuming real analyticity of $\vartheta$, it is a
sufficient condition
that the inequalities hold at at least one point.

Since the  functions $\mathcal{S}^{\,\vartheta}_n$
are expressible as polynomials in $\vartheta$ and its derivatives,
  \begin{equation}\label{S-INTRO}
  \mathcal{S}^{\,\vartheta}_n
  \,=\,\pi_n(\,\vartheta,\,
  \vartheta^\prime,\,\ldots,\,
  \vartheta^{(n-3)} ),
  \end{equation}
  where $\pi_n$ are certain polynomials, essentially
  the ones given by Tamanoi \cite{Tamanoi}
  (See Proposition \ref{Schwarzians-in-derivatives}
  and Remark \ref{FFS}), then
  the inequalities
(\ref{INTRO-DET-R-INEQ-1}) yield canonical
conditions on $\vartheta$.

  It  follows that
  for every integer $n\geq 2$  the inequality
  ${\mathcal{D}}^{\,\vartheta, \,g_R}_n\,\geq\,0$ is
   polynomial in
 $$\left\{\vartheta, \,\vartheta^\prime,\,
  \vartheta^{\prime\prime},\,\cdots,\,
   \cdots, \,\vartheta^{(2n-4)}, \frac{\pi}{R}\right\}\,,$$
   of degree $n(n-1)$ in $\frac{\pi}{R}$,
  isobaric with
   weight $\operatorname{w}=n(n-1)$ in $\vartheta$ and its derivatives (See Definition \ref{t-degree}) and with  rational constant coefficients.

The  construction in  Theorem \ref{TH-EXTEND-FINITE-R} is readily
applied to  quotients
of Jacobi fields along the geodesics of a two-dimensional manifold $M$
 with Riemannian metric $g$. One  obtains necessary and sufficient
 conditions for the existence of the adapted complex structure
 on $T^RM$, for $R$ finite
 or infinite. Such characterizing conditions are expressible as inequalities
  in
 covariant derivatives of the Gauss curvature $\sigma$ on $M$.

In more detail, for a given $g_R$ as above, define  a
sequence of functions on the tangent bundle
\begin{equation}\label{INTRO-S-R-TM}
\mathrm{S}^{\sigma, \,g_R}_n  \colon \, TM\,\to \mathbb{R},
\end{equation}  so
that restricted to the unit tangent bundle $UM\subset TM$
correspond to the higher
order $g_R$-Schwarzians  for the second order
differential equation associated by linearization
to the geodesic flow geodesic flow
\begin{equation}\label{GF}
\phi\colon UM\times \mathbb{R}\to UM\,;
\end{equation}
here $\phi_t\mathrm{v}=\dot{\gamma}(t)$ if $\gamma$ is
the geodesic with $\dot{\gamma}(0)\,=\, \mathrm{v}$.
Thus, we  have for all integer $n>1$, all $
\mathrm{v}\in UM$ and all $t\in\mathbb{R}$,
\begin{equation}\label{INTRO-S-R-TM-vs}
\mathrm{S}^{\sigma, \,g_R}_n(\phi_t \mathrm{v})\,
=\,\mathcal{S}^{\vartheta,\,g_R}_n(t)\,,
\end{equation}
where
\begin{equation}\label{Theta-Sigma}
\vartheta(t)\,=\,\sigma(\phi_t \mathrm{v})\,.
\end{equation}
Putting for  $\mathrm{v}\in TM$,
  $$\|\mathrm{v}\|^2\,=\,g(\mathrm{v},\mathrm{v}),$$
  the correspondence (\ref{INTRO-S-R-TM-vs}) formally amounts to
replacing in (\ref{S-INTRO}) the function
$\vartheta=\vartheta(t)$ by the product of $\|\mathrm{v}\|^2$
 times the Gauss curvature $\sigma$ viewed as a function
in $TM$, and
the derivatives
of $\vartheta$ with respect to  $t$
in (\ref{S-INTRO}) by derivatives of $\sigma$ along the geodesic
flow.

We introduce these functions
in Definition \ref{DEF-S-TM} and Definition \ref{DEF-S-R-in-TM}
where we use, equivalently,  covariant derivatives of $\sigma$. Here we note for clarification that
in particular for $R\,=\,\infty$,
\begin{equation}\label{S-INTRO-TM}
  \mathrm{S}^{\,\sigma}_n
  \,=\,\pi_n(\,\|\mathrm{v}\|^2\, \sigma,\,
  \|\mathrm{v}\|^2 \nabla_{\mathrm{v}}\vartheta,\,\ldots,\,
  \|\mathrm{v}\|^2\|\nabla^{n-3}_{\mathrm{v}}\sigma ),
  \end{equation} and that the corresponding functions for finite $R$
  are linear combinations of the functions in (\ref{S-INTRO-TM})  with coefficients polynomial in
  $R\,\|\mathrm{v}\|^{-1}$.

\vskip.2cm
\textit{Theorem \ref{TH-EXTEND-TUBE-FINITE}}
characterizes the  existence of the
adapted complex structure on $T^RM$ where  $R\,>\,0$, is finite or infinite,
by a basic  sequence of inequalities
\begin{equation}\label{INTRO-DET-R-INEQ}
{\mathrm{D}}^{\,\sigma, \,g_R}_n(\mathrm{v})\,:=\,\det\left[\frac{{\mathrm{S}}^{\,\sigma, \,g_R}_{i+j-1}
(\mathrm{v})}{(i+j-1)!}\right]_{\,i,\,j\,=\,1}^{n} \,\geq \,0\,,
\end{equation}
for all integer $n\,\geq\,1$ and
for all $\mathrm{v}$  on a certain set $\mathcal{Z}\subset TM$
which may be taken to be the unit tangent
bundle $UM$,\footnote{Since our definitions
render $\mathrm{D}^{\sigma, \,g_R}_n$ homogeneous
along the fiber of $TM$ of degree
$n(n-1)$, we may take any set that project (along the fibers of TM) to $UM$, or even  a
smaller is there is symmetry}
Such are   inequalities
polynomial in the covariant derivatives of
$\sigma$ illustrated by (\ref{EQ-INEQ-3-TM}).
\begin{Rem}
 From this basic set, by integration along each fiber of the unit tangent bundle $UM$,
one may derive necessary conditions on $M$ which include
(\ref{LSZIR}) and, in second order in $\sigma$,  as shown in Proposition \ref{PROP-SEC-OR-M},
  \begin{eqnarray}\label{INE-3}
  32\,\sigma^3
             +6\,\Delta{(\sigma^2)}
             +\left(3\,\Delta{\sigma} +
             48\,\sigma^2\right)
             \frac{\pi^2}{R^2}
             +18\,\sigma\,
             \frac{\pi^4}{R^4}+\,2\,
             \frac{\pi^6}{R^6}\,\geq\,27\,\|\operatorname{Grad}{\sigma}\|^2\,.
             \end{eqnarray}

             Note that inequality
             (\ref{INE-3})
             results from a determinantal condition in the quadratic form referred to above, and
             in the limit $R=\infty$ reduces to an inequality
             stronger  than (\ref{SZI}) which comes from  a diagonal condition.
\end{Rem}
\begin{Rem}
             One more integration, now of (\ref{INE-3}) on $M$ assumed orientable and closed, yields, with
             $\int_M\,
\mathrm{dA}\,=\,\mathrm{Area}(M)$ the total area of $M$,
        \begin{eqnarray}\label{INTRO-INTEG}
  &32\,\int_M\,\sigma^3\,\mathrm{dA}
             +
             48\,\,
             \left( \frac{\pi}{R}\right)^2\,\int_M\,\sigma^2\,\mathrm{dA}
             +18\,\left( \frac{\pi}{R}\right)^4\,\int_M\,\sigma\,
             \mathrm{dA}
             +\,2\,\left( \frac{\pi}{R}\right)^6\,\mathrm{Area}(M)
             &\\\notag &\,
             \qquad\qquad\geq\,27\,
             \int_M\,\|\operatorname{Grad}{\sigma}\|^2\,
             \mathrm{dA}\,.&
             \end{eqnarray}
\end{Rem}
In Theorem \ref{TH-EXTEND-TUBE-FINITE}
the role of the Schwarzians is to provide an intrinsic
description, via determinantal curvature inequalities,  of the higher order  properties of Jacobi fields
necessary for the existence of the adapted complex structure.
However, the infinite Hankel matrix with the entries involving the $g_R$-Schwarzians can be used to characterize other properties of quotients of Jacobi fields along a geodesic.
We include some results along this direction in Section \ref{SECT-RANK} as well as in Section \ref{SECT-RANK-C}
where we extend the
parameter $R$ in the  $g_R$-Schwarzians
to purely imaginary values,  $R=\sqrt{-1}\lambda$.

In Section \ref{SEC-SCHW-AS-MOMENTS} we briefly point out some equivalent
ways in which properties of the Schwarzians  characterize the existence of the adapted structure. This is derived from the classical theory of moments applied at  each $\mathrm{v}\in TM$, which for example implies that the existence of   a positive Borel measure so that for all integer $n\geq 1$
 $$\fbox{$\displaystyle {\mathrm{S}}^{\,\sigma,\, R}_n(\mathrm{v})\,= \,n!\,
   \,\int_{-\infty}^{\infty}
    t^{n-1} \mathrm{d}\mu^{\,\sigma,\,R}_\mathrm{v}(t)$}\,
 $$
 is equivalent to the adapted complex structure being defined on $T^RM$. Here the measures depend on  $\mathrm{v}$,  but uniform estimates
of the size of their support, all of which are actually bounded, are easily obtainable under certain assumptions on the Gauss curvature in $M$. For completeness we include these in  Proposition
\ref{PROP-SUPP-M}.

We include many illustrative computations in section \ref{SEC-Computations} such as
the explicit calculation of the Schwarzians and related determinants for constant curvature. A list items included in this section is displayed in the table of contents.

\section{Higher Schwarzians in $\mathbb{R}$ and related functions}\label{SECT-HOSCHWARZ}

Let $\vartheta \colon \mathbb{R}\to \mathbb{R}$ be a smooth
function and  $f_1$ and $f_2$  any pair of independent solutions  of
\begin{equation}\label{DEQA}
f^{\prime\prime}(t)+\vartheta(t)f(t)=0\,.\end{equation} Since   $\,f_1\,f_2^\prime-f_1^\prime\,f_2\,\neq\, 0$
\begin{equation}\label{FMATRIX}
 \mathbf{F}(t)\,=\,\left[\begin{array}{cc}
f_1(t)&f_1^\prime(t)\\
f_2(t)&f_2^\prime(t)
\end{array}\right]
\end{equation}
  is invertible for all $t$; thus, for all $(s,t)$ in $\mathbb{R}^2$, we define the  matrix
$$\mathbf{F}(x,\,t)\,:=\,\mathbf{F}^{-1}(x)\,\mathbf{F}(t)\,.
$$
 For fixed $x$,  the components of the first column of $\mathbf{F}(x,\,t)$
viewed as functions of $t$,
 form a pair of independent solutions of  (\ref{DEQA}). Since $\mathbf{F}(x,\,x)\,=\,\mathbf{I}$, the $2\times2$ identity matrix, the function
\begin{equation}\label{Vst}
V(x,\,t)\,=\,\frac{\mathbf{F}(x,\,t)_{2,1}}{\mathbf{F}(x,\,t)_{1,1}},
\end{equation}
is
defined for all values of $t$ in a neighborhood of $x$ (See also Remark \ref{Vst-REMARK}).
 \begin{Def}\label{HOS}
 The\emph{ higher order Schwarzian of order $n$ corresponding to $\vartheta$} is the function  ${\mathcal{S}}^{\,\vartheta}_{n}
\colon \mathbb{R}\to\mathbb{R}$ given by
\begin{equation}\label{n-Schwarzian}
 \mathcal{S}^{\,\vartheta}_n(x)\,:=\, \frac{\partial^{\,n} V(x,x+t)  }{\partial \,t^n}\,|_{\,t\,=\,0}\,.
\end{equation}
\end{Def}

\begin{Rem}\label{Vst-REMARK}
Let $x\in (a, b) \subset \mathbb{R}$ and let $h\,=\,f_2/f_1$ where $f_1$ and $f_2$
independent solutions of (\ref{DEQA}) with $f_1\neq 0$ in $(a, b)$. Then, it
follows,
using  (\ref{Vst}), $h\,=\,\frac{f_2}{f_1}$,
$ h^\prime\,=\,\frac{f_2^\prime}{f_1}-\frac{f_2\,f_1^\prime}{f_1^2}\,
=\,\frac{1}{f^2_1}$ and $h^{\prime\prime}\,=\,-2\,\frac{f^\prime_1}{f_1^3}$ that
\begin{eqnarray}\label{Schwarzian-shift}
V(x,\,x+t)\,
=\,
\frac{  h(x+t)-h(x)}
{h^\prime(x)+ \frac{1}{2}\,\frac{h^{\prime\prime}(x)}{h^\prime(x)}\left(
h(x+t)-h(x) \right)}\,.
\end{eqnarray}
Thus the functions in Definition \ref{HOS} are the same as those in (\ref{INTRO-S-R}).
\end{Rem}

\begin{Rem}\label{MOEBIUS-INV}
It is well-known that the values in (\ref{n-Schwarzian})
  do not depend on the pair of solutions chosen in (\ref{FMATRIX}), since given another pair of independent solutions $\tilde{f}_1$ and $\tilde{f}_2$  there are $a,b,c,d$  real constants with $a\,d\,\neq \,b\,c$
so that $\tilde{\mathbf{F}}(t)\,:=\,
\left[\begin{array}{cc}
\tilde{f}_1(t)&\tilde{f}_1^\prime(t)\\
\tilde{f}_2(t)&\tilde{f}_2^\prime(t)
\end{array}\right]\,=\,\left[\begin{array}{cc}
a&b\\
c&d
\end{array}\right]\left[\begin{array}{cc}
f_1(t)&f_1^\prime(t)\\
f_2(t)&f_2^\prime(t)
\end{array}\right],
$
 and hence
$\tilde{\mathbf{F}}(s,\,t)\,:=\,
 \tilde{\mathbf{F}}^{-1}(x)
 \,\tilde{\mathbf{F}}(t)\,=\,
 \mathbf{F}^{-1}(x)\,\mathbf{F}(t)
 \,=\,\mathbf{F}(x,\,t)$.
\end{Rem}

\begin{Def}[The $g_R$-Schwarzians]\label{HORS}
For  $R\,\in\,(0,\,\infty]$
we denote by $g_R$ a Pick function
\begin{equation}\label{GR-PICK}
g_R\,\colon (-a,\,a)\,{\cup}\,\left(\mathbb{C}
\setminus\mathbb{R}\right)\,\to\, \,(-b,b)
\,{\cup}\,\{z\,\in \,\mathbb{C}\,|\,0\,<\,|\Im{z}|<\,R\,\}
 \end{equation}
  for some  $a,b$  $\in \mathbb{R}$, $a,b>0$,  so that  $g_R(0)=0$, and
 \begin{equation}\label{gR-lim}
 \lim_{R \mapsto \infty}\,g_R(z)\,=\,z
 \end{equation}
  for every $z$ in its domain.
Then, for any such $g_R$,
 we define
 higher order ``$g_R$-Schwarzians" by
 \begin{equation}\label{VH-SR}
 \fbox{$ \displaystyle \,{\mathcal{S}}^{\,\vartheta, \,g_R}_{n}(x)\,=\,
 \frac{\partial^{\,n}   }{\partial \,t^n}V \left( x,\;x\,+\,g_R(t) \right)\,|_{\,t\,=\,0}$}\,.
\end{equation}
 \end{Def}
 \begin{Rem}
   A preferred example is $\displaystyle g_R(z)\,=\,-\frac{R}{\pi}\,\ln \left(1-\frac{\pi\,z}{R}\right)$. We will indicate when such a specific choice is made\,.
\end{Rem}
We
 present preliminary formulas to be used  in Theorem  \ref{TH-EXTEND-FINITE-R}.
\begin{Def}\label{Def-FPS}
For a given $x \in \mathbb{R}$
\emph{the pair of fundamental solutions associated to the point $x$} is
the pair of solutions
$f_{x,1}=f_{x,1}(t)$ and $f_{x,2}=f_{x,2}(t)$ of  (\ref{DEQA}) determined by
\begin{equation}\label{FP}
f_{x,1}(x)\,=\,f^\prime_{x,2}(x)\,
=\,1\,,\;\;\; f_{x,2}(x)\,=\,f^\prime_{x, 1}(x)\,=\,0\,.
\end{equation}
\end{Def}

\begin{Prop}\label{H-SCHWARZ}
Given $x\in \mathbb{R}$ putting
$h_{x}(t)\,:=
\,f_{x, 2}(t)/f_{x, 1}(t)$ then for integer $n\,\geq \,1$, for
any $g_R$ as in Definition \ref{HOS} and
with the derivative notation as usual,
   $h^{(0)}\,=\,h$, $h^\prime\,=\,h^{(1)}$, etc.
\begin{itemize}
\item[(1)] we have
   \begin{eqnarray}\label{FLA-H-SCHWARZ}
   \fbox{$\displaystyle \mathcal{S}^{\,\vartheta}_{n}(x)\,=\,
  h_{{x}}^{(n)}(x)$}\,;
   \end{eqnarray}
\item[(2)] putting
\begin{equation}\label{DEF-Tx}
        T_{x}(t)=x+t
\end{equation}
    we have
     \begin{eqnarray}\label{FLA-H-R-SCHWARZ}
    \fbox{$\displaystyle \mathcal{S}^{\,\vartheta,\,g_R}_{n}(x)\,=\,
     \left({h_{x}{\circ}T_{x}{\circ }\, g_R}\right)^{(n)}(0)$}\,.
       \end{eqnarray}
\end{itemize}
\end{Prop}
\begin{proof}  Construct $\mathbf{F}$ as in (\ref{FMATRIX})
the fundamental pair associated to $x$
with $f_{x,\,1}$ and $f_{x,\,2}$ to get
$\mathbf{F}(x)\,=\,\mathbf{I}$, and thus
\begin{equation}\label{VH}
V(x, x+t)\,=\,h_{{x}}(t)\,.
\end{equation} Now,
for (1), use  Definition \ref{HOS}, and for  Part (2), Definition \ref{HORS}.
\end{proof}

\begin{Rem}\label{3Schwarz} From (\ref{FLA-H-SCHWARZ}) and
(\ref{FLA-H-R-SCHWARZ}) note that
since
$h_{{x}}(x)=0$, $h_{{x}}^\prime(x)=1$,  $h_{{x}}^{\prime\prime}(x)=0$,
and,  by (\ref{classical-SCHW}),
$2\,\vartheta(x)\,=\,h_{x}^{\prime\prime\prime}(x)$, we have
${\mathcal{S}}_0^{\,\vartheta}\,= \,0$, ${\mathcal{S}}_1^{\,\vartheta}\,=\, 1$, ${\mathcal{S}}_2^{\,\vartheta} \,=\, 0$ and $
{\mathcal{S}}^{\,\vartheta}_3\,=\,2\,\vartheta$\,.
\end{Rem}

\begin{Rem}\label{R-Schwarz}
From the definition,
each function $\mathcal{S}^{\,\vartheta,\,g_R}_{n}$ is a linear
combination with constant coefficients
of the Schwarzians; we have
\begin{equation}\label{Lin-Comb}
    \fbox
    {$ \displaystyle
        \mathcal{S}^{\,\vartheta,\,g_R}_{n}(x)\,=\,
     \sum_{k=1}^n\,  g_{R, \,n,\,k}\,{\mathcal{S}}^{\,\vartheta}_k(x)$}
\end{equation}
where, by Fa\`{a} di Bruno formula,
 \begin{eqnarray}\label{GNK}
 {\mathrm{g}}_{R,\,n,k}\,=\,\overline{\sum}_{n,k}\prod_{i=1}^n\frac{1}{k_i!}
\left(\frac{({\mathrm{g}}_R)^{(i)}(0)}{i!}\right)^{k_i}\,
\end{eqnarray}
with $
\overline{\Sigma}_{\,n,\,k}$ indicating the sum according to
  $\sum_{i\,=\,1}^n\,k_i\,=\,k$ and  $\sum_{i\,=\,1}^n i\; k_i\,=\,n$.

An explicit computation of these coefficients for the choice  $\displaystyle g_R(z)\,=\,-\frac{R}{\pi}\,\ln \left(1-\frac{\pi\,z}{R}\right)$ is provided in  Proposition \ref{PROP-GR-LN-in-R-}\,.
\end{Rem}
\begin{Rem}\label{HORS-C} From (\ref{Lin-Comb}) it follows that the  real parameter $R>0$ enters in  any $g_R$-Schwarzian
polynomially in $R^{-1}$, and thus $R$  can be extended to mean a non-zero
complex number. For example, for $0\neq \lambda \in \mathbb{R}$ put
 \begin{equation}\label{HORS-GR-C}
 \fbox{$\displaystyle
 \mathcal{S}_n^{\,\sigma, \,g_{\sqrt{-1}\lambda}}
 :=\mathcal{S}_n^{\,\sigma, \,g_{R}}\big{|}_{R=\sqrt{-1}\lambda}$}\,.
 \end{equation}
 We will use these in connection to closed geodesic in Section \ref{SEC-RANK-C}
\end{Rem}
\begin{Def} \label{DEF-DET-R}
Given $R\,\in\,(0,\,\infty]$, $g_R$ as above,
and integer $n\,\geq 1$ put
\begin{equation}\label{DEF-D-R-in-R}
\fbox{$\displaystyle
{\mathcal{D}}^{\,\vartheta, \,g_R}_{n}\,(x)\,:=\,
\det \left[\frac{{\mathcal{S}}^{\,\vartheta,
\,g_R}_{i+j-1}(x)}{(i+j-1)!}\right]_{i,\,j\,=1}^{n}
$}\,.
\end{equation}
\end{Def}

\begin{Rem}\label{LAMBDA-FOR-R-INFTY}
For each $x \in \mathbb{R}$ is defined and each integer $n\,\geq\,1$,
 \begin{eqnarray}\label{LIM-S}
 &&\lim_{R \rightarrow \infty} {\mathcal{S}}^{\,\vartheta, \,g_R}_{n}(x)\,=\,{\mathcal{S}}^{\,\vartheta}_{n}(x)\,,
 \\\label{LIM-D}
 &&
 \lim_{R \rightarrow \infty} {\mathcal{D}}^{\,\vartheta, \,g_R}_{n}(x)\,
 =\,{\mathcal{D}}^{\,\vartheta}_{n}(x)\,:=\,
\det
\left[\frac{{\mathcal{S}}^{\,\vartheta}_{i+j-1}(x)}{(i+j-1)!}
\right]_{i,\,j \,=\,1}^{n}\,.
 \end{eqnarray}
\end{Rem}

\section{Theorem \ref{TH-EXTEND-FINITE-R}. Continuation to a strip, finite or infinite}\label{SECT-TH-1}

In this section we obtain a set of necessary and sufficient
conditions on
$\vartheta \colon \mathbb{R}\to \mathbb{R}$
so that it belongs to $\mathfrak{L}_R$.

The underlying result here is \emph{Loewner's Theorem}.
\emph{A $C^\infty$ function $F$ on $(a, b)
\subset \mathbb{R}\subset \mathbb{C}$ has
the property that for all $n=1,2,\ldots$
and all $x\in (a,b)$
\begin{equation}\label{DET-INEQ-F}
\det\left[\frac{F^{(i+j-1)}\,(x)}{(i+j-1)!}
 \right]_{\,i,\,j\,=\,1}^{n} \,\geq \,0
\end{equation}
if and only if it extends
to a function on $(\mathbb{C}\setminus\mathbb{R})\,\cup (a,b)$
with  positive imaginary part when $\Im z>0$, with equality true for
a value of  $n$ if and only if $F$ is rational.}

In the proof of the next result,
we will show  the necessary condition using Fatou
representation for Pick functions. As for the sufficient condition stated above,
 with the hypothesis at just one point and
real analyticity to it, we adopt an  argument from
Bendat and Sherman \cite{Bendat-Sherman}.

\begin{thm}\label{TH-EXTEND-FINITE-R}
Let
$\vartheta\colon \mathbb{R}\to \mathbb{R}$ be smooth.
  Let $f_1=f_1(t)$ and $f_2=f_2(t)$ be any pair of  independent solutions  of
 $f^{\prime \prime}(t)+\vartheta(t)\,f(t)\,=\,0$\,.
 Put $W(f_1,\,f_2)\,:
=\,f_1\,f^\prime_2-f^\prime_1\,f_2$ (a constant $\neq 0$)
and set  $h\,=\,f_2/f_1$. Let
$R\,>\,0$ finite or infinite be given  and take some map $g_R$ as in (\ref{GR-PICK}).
\begin{itemize}\item[(1)]
If $h$
 extends  meromorphically  on $
 \{\,z\,\in\,\mathbb{C}\,\,,|{\Im}(z)|\,<\,R\,\}$,
  analytically on  \footnote{For $R\,=\,\infty$
 (\ref{HR}), the upper-half plane, is denoted by
  $H_+$.}
\begin{equation}\label{HR}
H^R_+\,:=\,\{ \,z\in \mathbb{C}\,\,,0\,<\,\Im(z)<\,R\,\}\,,
 \end{equation} with $W(f_1,\,f_2)\,\Im{h(z)}\,\Im(z)>0$,
then,
 $\vartheta$ is  real analytic  on $\mathbb{R}$, and for
 all integer $n\,\geq 1$ $$\mathcal{D}^{\,\vartheta, \,g_R}_{n}\,\geq\,0\,.$$
  \item[(2)]
 If $\vartheta$ is real analytic on  $\mathbb{R}$ and there is $x_0\,\in\,\mathbb{R}$  such that
 for all integer $n\,\geq \,1$
 $$\mathcal{D}^{\,\vartheta, \,g_R}_{n}(x_0)\,\geq\,0\,$$
 then $h$ extends as in (1).
 \end{itemize}
\end{thm}

\begin{proof}
\underline{For Part (1)}.
On account that
\begin{itemize}
         \item[(i)] quotients of independent solutions are related by
real Mo\"{e}bius transformations,
         \item[(ii)] a Mo\"{e}bius transformation
$\displaystyle M(z)=\frac{a\,z+b}{c\,z+d}$ with $a,b,c, d \in \mathbb{R}$ sends $H^R_+ $
to itself if and only if $\det{M}\,=\,a\,d-c\,d\,>\,0$, and
         \item[(iii)] $W(af_1+bf_2, cf_1+df_2)\,
         =\,W(f_1, f_2)\,\det{M}$,
       \end{itemize}
 it follows that a map
 $h$ satisfies the hypothesis in 1) if and only any real Mo\"{e}bius transformation of $h$ does.

  Consequently, if  $x$ is  any point of $ \mathbb{R}$,  the map
 $\displaystyle h_{x}=\frac{f_{x,\,2}}{f_{x,\,1}}$, which is
the quotient of the pair of fundamental solutions associated to
$x$ (Definition \ref{Def-FPS}), satisfies hypothesis 1).
Since $h^\prime_x(x)\,=\,1\,\neq \,0$,
 the classical Schwarzian of $h_x$ computes $2\,\vartheta$ and shows
its real analyticity, in some
neighborhood
of $x$. It follows the analyticity in all $\mathbb{R}$.

Now, let  $R \in (0, \infty]$ be given  and again let
  $x$ be  any point of $ \mathbb{R}$.
Consider the composition
\begin{equation}
\label{MAP-H-T-GR}
h_{x} \,{\circ} \,T_{x}\,{\circ}\,  {g_R}
\end{equation}
with $T_x$ as in (\ref{DEF-Tx})
and   $g_R$ as in (\ref{GR-PICK}). Since $g_R$ is a Pick function defined
in $\left(\mathbb{C}\,\setminus\,\mathbb{R}\right)\,\cup\,
(a,b)$, with $a\,<\,0\,<\,b$,  and  maps into $|\Im (z)|<R$,  there is
an $\epsilon_x=\epsilon_x^{\vartheta,\,g_R}\,>\,0$ such that
(\ref{MAP-H-T-GR})
is defined and holomorphic
on
\begin{equation}\label{a}
H_+ \,\cup\,H_-\cup\,(-\epsilon_x,\,\epsilon_x)\;.
\end{equation}
Moreover, the map (\ref{MAP-H-T-GR})
has positive imaginary part
on $H_+$
since $g_R$ and
$h_x$ both do  ($g_R$ by definition, and $h_x$ since it satisfies (1) and  $W(f_{x,\,1}, f_{x,\,2})\,=\,1\,>\,0$).
So,
the map (\ref{MAP-H-T-GR}) is a Pick function with domain (\ref{a}).

According to the Fatou  representation for a Pick function, for $z\in\mathbb{C}$,
 \begin{equation}\label{FATOU-F}
\left( h_{x}\,{\circ}\,
T_{x}\,{\circ}\, {g_R}\right)(z)\,=\,\alpha_x\,z\,+ \,\beta_x
 \,+\int_{-\infty}^\infty \,\left(\, \frac{1}{t-z}-\frac{t}{1+t^2}
 \right) \,d\nu_x(t),
 \end{equation}
  with constants $$\alpha_x\,=\,\alpha_x^{\vartheta,\,g_R}\,\geq\, 0,\,\;\;\;
  \beta_x \,=\,\beta_x^{\vartheta,\,g_R}\in \mathbb{R},$$
  and with
 $d\nu_x\,=\,d\nu_x^{\vartheta,\,g_R}$ a non-negative Borel measure
 with $\displaystyle \int_{-\infty}^\infty\frac{d\nu_x(t)}{1+t^2}\,<\,\infty$.

With $z\,=\,x+\sqrt{-1}y$, $x,y\in \mathbb{R}$, the  measure $d\mu_x$ is
 given  by  the weak limit of
 $$\Im \left( h_{x}\,{\circ}\,
T_{x}\,{\circ}\, {g_R}\right)
 (x+\sqrt{-1}y)\,\;\textrm{ as }\;\;y\mapsto  0^+\,.$$
Since the  map  $h_{x}\,{\circ}\,
T_{x}\,{\circ}\, {g_R}$
is real valued in
$(-\epsilon_x,\,\epsilon_x)\subset \mathbb{R}$ we have
$$\int_{-\epsilon_x}^{\epsilon_x} d\mu_x(t)\,=\,
\lim_{y\mapsto 0^+}\int_{-\epsilon_x}^{\epsilon_x} \,\Im \left( h_{x}\,{\circ}\,
T_{x}\,{\circ}\, {g_R}\right)
 (t+\sqrt{-1}y)\, dt\,=\,0\,.
$$
Thus
\begin{equation}\label{supportFATOU}
\mathrm{Support}\left(d\nu_x\right)
\cap(-\epsilon_x, \epsilon_x)=\emptyset\,,
\end{equation}
and we can take derivatives at the origin under the integral sign to get
 \begin{eqnarray}
 1\,\equiv\,{\mathcal{S}}^{\,\vartheta,\,g_R}_{1}(x)\,=\,
 \left( h_{x}\,{\circ}\,
T_{x}\,{\circ}\, {g_R}\right)^\prime(0)
 &=&\alpha_x+\int_{-\infty}^\infty
   \frac{d\nu_x(t)}{t^{2}}\,,
   \end{eqnarray}
   and for all $n\geq 2$
   \begin{eqnarray}
 {\mathcal{S}}^{\,\vartheta,\,g_R}_{n}(x)
 &=&\left( h_{x}\,{\circ}\,
T_{x}\,{\circ}\, {g_R}\right)^{(n)}(0)=\,n!\,\int_{-\infty}^\infty
   \frac{d\nu_x(t)}{t^{n+1}}\,.
\end{eqnarray}
It follows that
 \begin{equation}\label{DET-THM-1}
\det \left[\frac{{\mathcal{S}}^{\,\vartheta,
\,g_R}_{i+j-1}(x)}{(i+j-1)!}\right]_{\,i,\,j\,=\,1}^{n}
\end{equation}
is represented by
$$
\det \left[
\begin{array}{cccc}
   \alpha_x+\int_{-\infty}^\infty t_1^{-2}{d\nu_x(t_1)} & \int_{-\infty}^\infty t_2^{-3}{d\nu_x(t_2)}
     &\,\cdots\,
     &\int_{-\infty}^\infty t_n^{-n-1}{d\nu_x(t_n)}
     \\\\
   \int_{-\infty}^\infty t_1^{-3}{d\nu_x(t_1)}
   &
   \int_{-\infty}^\infty t_2^{-4}{d\nu_x(t_2)}
     &\,\cdots\,
     &\int_{-\infty}^\infty t_n^{-n-2}{d\nu_x(t_n)}
     \\\\
     \cdots & \cdots & \cdots & \cdots
     \\\\
     \int_{-\infty}^\infty t_1^{-n-1}{d\nu_x(t_1)}
      & \int_{-\infty}^\infty t_2^{-n-2}{d\nu_x(t_2)}
     &\,\cdots\,
     &\int_{-\infty}^\infty t_n^{-2n}{d\nu_x(t_n)}
   \end{array}
   \right]\,.
$$
Due to the symmetry of the Hankel matrix,
the multi-linearity of the determinant,
and the linear nature of  integration,  the expression above  is  written as
\begin{eqnarray}\label{Integrals}
\alpha_x \int_{-\infty}^{\infty}\cdots\int_{-\infty}^{\infty}
\det\left[t_{j}^{-i-j}\right]_{\,i,\,j\,=\,2}^{n}
\,d\nu_x(t_2)\cdots \,d\nu_x(t_n) +\\\notag
\qquad
+
\int_{-\infty}^{\infty}\cdots\int_{-\infty}^{\infty}
\det\left[t_j^{-i-j}\right]_{\,i,\,j\,=\,1}^{n}
\,d\nu_x(t_1)\cdots \,d\nu_x(t_n) \,.
\end{eqnarray}
Consider the  integrand corresponding to the second term in (\ref{Integrals}); it can be written as
\begin{eqnarray}\label{INTEGRAND-2}
\det\left[t_{j}^{-i-j}\right]_{\,i,\,j\,=\,1}^{n}&=&
\,
\det\left[t_{j}^{-i}\right]_{\,i,\,j\,=\,1}^{n}\,\prod_{j=2}^{n}t_{j}^{-j}.
\end{eqnarray}
Let $\Sigma$ be the permutation group
of the set  $\{1,\,\cdots,\,n\}$.
For any $\delta \in \Sigma$, from (\ref{INTEGRAND-2}),
\begin{eqnarray}\notag
\det\left[t_{\delta(j)}^{-i-j}\right]_{\,i,\,j\,=\,1}^{n}
\,&=&\,
\det\left[t_{\delta(j)}^{-i}\right]_{\,i,\,j\,=\,1}^{n}
\;
\prod_{j\,=\,1}^{n}t_{\delta(j)}^{-j}
\\\notag
\,&=&\,\mathrm{sign}(\delta)
\;\det\left[t_{j}^{-i}\right]_{\,i,\,j\,=\,1}^{n}\;
 \prod_{j\,=\,1}^{n}
t_{\delta(j)}^{-j}\,,
\end{eqnarray}
from which it follows that
\begin{eqnarray}\label{AV-2}
\sum_{\delta \in \Sigma}
\det\left[t_{\delta(j)}^{-i-j}\right]_{\,i,\,j\,=\,1}^{n}\,=\,
\left(\det\left[t_{j}^{-i}\right]_{\,i,\,j\,=\,1}^{n}\right)^2\,.
\end{eqnarray}
Now, since no $\delta\in \Sigma$ influences  the value of
$$\int_{-\infty}^{\infty}\cdots\int_{-\infty}^{\infty}
\det\left[t_{\delta(j)}^{-i-j}\right]_{\,i,\,j\,=\,1}^{n}
\,d\nu_x(t_1)\cdots \,d\nu_x(t_n)\,,
$$
 by integrating (\ref{AV-2}),
the second integral in (\ref{Integrals}) equals
$$
\frac{1}{n!}\int_{-\infty}^{\infty}\cdots\int_{-\infty}^{\infty}
\left(\det\left[t_{j}^{-i}\right]_{\,i,\,j\,=\,1}^{n}\right)^2
\,d\nu_x(t_1)\cdots \,d\nu_x(t_n)\,.
$$

The argument for the integral in (\ref{Integrals}) corresponding to $\alpha_x$ is similar.

In conclusion,
\begin{eqnarray}\label{DET-THM-1-NICE-FORM}
\mathcal{D}^{\,\vartheta, \,g_R}_{n}(x)\;=&&
\displaystyle{\det \left[\frac{{\mathcal{S}}^{\,\vartheta,
\,g_R}_{i+j-1}(x)}{(i+j-1)!}\right]_{\,i,\,j\,=\,1}^{n}}\\\notag
=&&\displaystyle{\frac{\alpha_x}{(n-1)!}\,
\int_{-\infty}^{\infty}\cdots\int_{-\infty}^{\infty}
\left(\det\left[t_{j}^{-i}\right]_{\,i,\,j\,=\,2}^{n}\right)^2
\,d\nu_x(t_2)\cdots \,d\nu_x(t_n) \;+}\\\notag
&&\;\;\;+\;\displaystyle{
\frac{1}{n!}\int_{-\infty}^{\infty}\cdots\int_{-\infty}^{\infty}
\left(\det\left[t_{j}^{-i}\right]_{\,i,\,j\,=\,1}^{n}\right)^2
\,d\nu_x(t_1)\cdots \,d\nu_x(t_n) }\,,
\end{eqnarray}
which is non-negative, due to the non-negativity of $\alpha_x$ and of the measure.

 Since $x\in \mathbb{R}$ was arbitrary, Part (1) is proved.

\underline{\textbf{For Part (2).}}

By the opening paragraph of the proof of Part (1)
 it suffices to consider the case $h_{{x_0}}\,=\,f_{x_0,\,2}/f_{x_0,\,1}$,
 the quotient of
 the fundamental solutions associated to $x_0$ as in (\ref{FP})\,.

 From the hypothesis on $\vartheta$, $h_{{x_0}}$
 is real analytic on
 some neighborhood of $x_0$ and thus so is
 $\displaystyle h_{{x_0}}{\circ\,}T_{x_0}{\circ} {g_R}$
  in  some  neighborhood of the origin\,.
  Thus for all integer $n\,\geq \,1$,
  \begin{equation}\label{DET}
  \det\left[\frac{\left(h_{{x_0}}{\circ}T_{x_0}{\circ\,} {g_R}
  \right)^{(i+j-1)}}{(i+j-1)!}(0)
  \right]_{\,i,\,j\,=\,1}^{n}
  \,\,\,=\,\,\,
{\mathcal{D}}^{\,\vartheta, \,R}_{n}(x_0)
\,\,\,\geq \,\,\,0\,.\end{equation}
 Here the equality is valid by Proposition \ref{H-SCHWARZ} together with
   $(h_{{x_0}}{\circ}T_{x_0}{\circ\,} {g_R} )^{(n)}(0)\,=\,
   \mathcal{S}^{\,\vartheta, \,g_R}_{n}(x_0)$, while the inequality holds
by hypothesis.

\leftline{\textbf{Claim}: \emph{From  (\ref{DET}),
it follows that $h_{{x_0}}{\circ}T_{x_0}{\circ\,} {g_R}$ extends as a Pick function.}}

\begin{proof}[Proof of Claim]
The claim is a  result by
Bendat and Sherman \cite{Bendat-Sherman} plus compositions
of functions that we want to keep track of.
We
rely on the solution of the classical Hamburguer moment problem \cite{WIDDER-LT} which
given the non-negativity of the Hankel determinants
guarantees the existence
of a non-decreasing function
\begin{equation}\label{MU-depend}
\mu_{x_0}=\mu^{\,\vartheta,\, g_R}_{ \,x_0}
\colon\mathbb{R}\to\mathbb{R}
\end{equation}
such that
 in terms of Stieltjes integrals,
 \begin{equation}\label{HMP}
 \,(h_{{x_0}}{\circ}\,T_{x_0}{\circ\,} {g_R} )^{(n)}(0)\,=\, n!
 \int_{-\infty}^{\infty} t^{n-1} d\mu_{x_0 } (t)
 \,.
 \end{equation}
 Then, by the real analyticity
 $h_{{x_0}}{\circ }\, T_{x_0} {\circ\,} {g_R} $
 in a neighborhood of $x\,=\,0$,
 \begin{eqnarray}
 F^{\,\vartheta,\,g_R}_{x_0}(z)
 \,&:=&\,
\sum_{n\,=\,1}^{\infty}
(h_{x_0} {\circ}\, T_{x_0} {\circ\,} {g_R} )^{(n)}(0)\;\frac{z^n}{n!}\\\notag
\,&=&\,
\sum_{n\,=\,1}^{\infty}\left(\int_{-\infty}^{\infty} t^{n-1}
d\mu_{x_0}(t)\right) {z^n}\,,
\end{eqnarray}
is
absolutely convergent
 for $\displaystyle z\,\in\, \mathcal{O}_{x_0}
 =\mathcal{O}^{\,\vartheta,\,g_R}_{x_0}
 =\{z\in\mathbb{C},\,|z|\,<\,\rho_{x_0}\}$,  where
 \begin{equation}\label{RHO-depend}
 \rho_{\,x_0}=\rho^{\vartheta,\,g_R}_{\,x_0}\,>\,0
 \end{equation}
  is the radius
 of convergence of $h_{{x_0}}{\circ }\, T_{x_0} {\circ\,} {g_R}$\,.

 The measure
  $d\mu_{x_0}$
 is  supported in the
 interval
 \begin{equation}\notag
 [\,-1/\rho_{x_0},
 \,1/\rho_{x_0}\,]
 \subset \mathbb{R}\,,
 \end{equation}
since otherwise, there is a constant $C\,>\,1/\rho_{x_0}$
such that for all
  integer $m\,\geq\,1$
  $$\int_{-\infty}^\infty\,t^{2\,m}\, d\mu_{x_0}(t)
  \,\geq\,C^{2\,m}\,\left(
 \int_{-\infty}^{-C}\, d\mu_{x_0}(t)+\int_{C}^\infty\,
 d\mu_{x_0}(t)\right)\,>\,0\;,$$ and thus
 $\lim_{m\mapsto \infty}\sqrt[2\,k+1]{\int_{-\infty}^\infty\,t^{2\,m}\,
 d\mu_{g_R,\,x_0}}\,\geq\, C \,>\,1/\rho_{x_0}$, which would
 imply a radius of convergence
 of $h_{x_0}\,{\circ}\, T_{x_0}\, {\circ}{g_R}$
 strictly smaller that $\rho_{x_0}$, a contradiction.

 So we now interchange
 summations in
 \begin{eqnarray}\label{SUM2}
 F^{\,\vartheta,\,g_R}_{x_0}(z)\,&=&\,
\sum_{n\,=\,1}^{\infty}\left(
\int_{-1/\rho_{x_0}}^{1/\rho_{x_0}} t^{n-1}
d\mu_{x_0}(t)\right) {z^n}\,,
\end{eqnarray}
to obtain for $z\in \mathcal{O}_{x_0}$
 \begin{equation}\label{F-REPRESENT}
 F^{\,\vartheta,\,g_R}_{x_0}(z)\,
 = \, \int_{-1/\rho_{x_0}}^{1/\rho_{x_0}}
 \frac{z}{1-t\,z} \; d\mu_{x_0}(t)\,.
 \end{equation}
Since, by construction,
on the interval  $(-\rho_{x_0}, \,\rho_{x_0})\,=\,\mathcal{O}_0\,\cap\, \mathbb{R}$
we have the agreement
   $$F^{\,\vartheta,\,g_R}_{x_0}\,\equiv\,
 h_{{x_0}}{\circ}\, T_{x_0} {\circ\,} {g_R}\,,$$
 the function
  (\ref{F-REPRESENT}) represents a holomorphic  extension of
  $h_{{x_0}}{\circ} \,T_{x_0} {\circ\,} {g_R} $ which is now defined on
 $$\left(\mathbb{C}\,\setminus\,\mathbb{R}\right)\,\cup\,
 \mathcal{O}_0\,,$$ and
 that maps $H_+$ to itself.  This is the extension as a Pick function
 derived from the conditions (\ref{DET}) whose existence was
 claimed.
\end{proof}
\vskip.2cm
 Now, from this, the function
\begin{equation}\label{h0-extended-1}
 \widehat{h}_{{x_0}}\,:
 =\,F^{\,\vartheta, \,g_R}_{\,x_0}\,{\circ} \,g^{-1}_R {\circ\,}T_{-x_0}
\end{equation}
is
 holomorphic on
 \begin{equation}\label{Dom-h}
 H^R_+\,\cup \,H^R_-\,\cup\,
 \mathcal{U}_{x_0},
 \end{equation}
where
$\displaystyle \mathcal{U}_{x_0}=T_{x_0}\left(g_R \left(\mathcal{O}_0\right)\right)$
 is an open neighborhood of $x_0$ in $\mathbb{C}$,
and  maps  $H^R_+$ to $H_+$\,.

This means that
$\widehat{h}_{x_0}$ provides an extension of $h_{x_0}$ as a Pick function
with domain (\ref{Dom-h})\,.
We now  check that  it extends
across $\mathbb{R}
\setminus f_{x_0,\,1}^{-1}(0)$ as well, as a Pick function,
which will complete the proof.

To see this last point, note that, since the original $h_{x_0}$
is real analytic on $\mathbb{R}
\setminus f_{x_0,\,1}^{-1}(0)$,
there is
an open set $\mathcal{U}$
of $\mathbb{C}$,  with
 \begin{equation}\label{INC-1}
 x_0\,\in\,
\mathbb{R}\,\setminus
f_{x_0,\,1}^{-1}(0)\,
\subset\,\mathcal{U}\,,
\end{equation}
and
a holomorphic function
$$\widetilde{h}_{x_0}\colon \mathcal{U}\to \mathbb{C}$$
so that on
$\mathbb{R}
\setminus f_{x_0,\,1}^{-1}(0)$,
$$\widetilde{h}_{x_0}\,\equiv \,h_{x_0}\,.$$

In particular,
$\widetilde{h}_{x_0}$ is determined near $x_0$
 by the Taylor series of $h_{x_0}$
centered at $x_0$, which has a non-zero radius of convergence, say $r_0$.
So, on the neighborhood of $x_0$  in $\mathbb{C}$ given by
$$\{z\,\in\, \mathbb{C},\,|z-x_0|<r_0\}
\,\cap\, \mathcal{O}_0\,,$$
we have the coincidence
\begin{equation}\label{agree1}
\widetilde{h}_{x_0}\,\equiv\,\widehat{h}_{x_0}\,.
\end{equation}
Thus, by uniqueness of analytic continuation,
the identity (\ref{agree1}) also holds
in $$\mathcal{U}\cap\,H^R_+\,,$$
and, similarly, in $$\mathcal{U}\cap\,H^R_-\,.$$
This means, in light of (\ref{INC-1}),  that $\widetilde{h}_{x_0}$ provides an analytic extension
 of $\widehat{h}_{x_0}$ across the set $\mathbb{R}\setminus
 f_{x_0,\,1}^{-1}(0)$, where the function
 $\widetilde{h}_{x_0}$ assumes only real values. It follows that
 $\widehat{h}_{x_0}$ is now analytically extended
 as a Pick function with domain
   $$H^R_+\,\cup \,H^R_-\,\cup\,
   \mathbb{R}\setminus f_{x_0,\,1}^{-1}(0)\,
   =\,\mathbb{C}\setminus f_{x_0,\,1}^{-1}(0)\,.$$
   It extends meromorphically on $\mathbb{C}$ with  (simple) poles at
 $f_{x_0,\,1}^{-1}(0)\subset \mathbb{R}$\,.

This proves part (2).
 \end{proof}

\begin{Rem}\label{D-at-x0}
Theorem \ref{TH-EXTEND-FINITE-R} shows that, for $\vartheta$  real analytic
in $\mathbb{R}$,
if there is an $x_0\in \mathbb{R}$ so that for all $n\,\geq\,1$
$\mathcal{D}^{\,\vartheta, \,g_R}_{n}(x_0)\,\geq \,0$
 then
  $\mathcal{D}^{\,\vartheta, \,g_R}_{n}(x)\,\geq\,0$
for all $n\,\geq\,1$ and  for all  $x\in \mathbb{R}$. This will be used in Corollary
\ref{COR-REAL-ANALYTIC}.
\end{Rem}

\subsection{A simple curvature estimate for the measure}\
\vskip.2cm
We make a few comments on the relation exploited in the proof of Theorem \ref{TH-EXTEND-FINITE-R}, of the Schwarzians on $\mathbb{R}$ and  moments.

Recall from the proof of part (2) of Theorem \ref{TH-EXTEND-FINITE-R} its equivalence to
the possibility of the Schwarzians being
point-wise expressible as moments
\begin{equation}\label{SCHW-MOM}\fbox{$\displaystyle
{\mathcal{S}}^{\,\vartheta,\,g_R}_n(x)\,
 = \,n!\,  \int_{-1/\rho_{x}}^{1/\rho_{x}}
 \,t^{n-1}\; d\mu_{x}(t)$}\,,
 \end{equation}
 where, for each $x$, $d\mu_{x}(t)=d\mu_x^{\,\vartheta,\,g_R}(t)$ is a non-negative  Borel measure in $\mathrm{R}$ with
\begin{equation}\label{SUPPORT-MU-depend-on-x}
       \mathrm{Support}\,d\mu_{x}(t)\,\subset\,[-1/\rho_{x},\;1/\rho_{x}]
\end{equation}
for some
$$
\rho_{x}\,=\,\rho^{\,\vartheta,\, g_R}_{ \,x}\,\in\,(0,\,\infty]\,.
$$
There is a very simple estimate for the  support of that measure in terms of curvature as follows.

\begin{Prop}\label{PROP-SUPP-EST}
Assume $\vartheta$ is as in part (1) of Theorem \ref{TH-EXTEND-FINITE-R}.
Take
$\displaystyle g_R(z)\,=\,-\frac{R}{\pi}\,\ln \left(1-\frac{\pi\,z}{R}\right)$\,.
Set $\sup \vartheta\,=\,\sup_{x\in \mathbb{R}}{\vartheta}$.
Then for any $x\in \mathbb{R}$,
\begin{itemize}
\item[(1)] if $\sup \vartheta\,\leq 0$,
 $\displaystyle \mathrm{ Support }\left({d\mu^{\,\vartheta,\,g_R}_x}\right)\,\subset\,
  \left[\,-\frac{\pi}{R}\;,\;\;\frac{\pi}{R}\,\right]$\,,
interpreted as the set $\{0\} \textrm{ if } R\,=\,\infty$\,;\\
\item[(2)] if $\sup{\vartheta}\,>\,0$,
$\displaystyle
\mathrm{ Support }\left({d\mu^{\,\vartheta,\,g_R}_x}\right)\,\subset\,
  \left[{\displaystyle -\frac{\pi}{R-\,R\,e^{-\frac{\pi^2}{2\,R\sqrt{\sup \vartheta}}}}\;,\;
  \;\frac{\pi}{R-\,R\,e^{-\frac{\pi^2}{2\,R\sqrt{\sup \vartheta}}}}\,}\right]\,,
  $

interpreted as
the interval
 $\left[{\displaystyle  -\frac{2\,\sqrt{\sup{\vartheta}}}{\pi}\;,\;
  \;\frac{2\,{\sqrt{\sup{\vartheta}}}}{\pi}} \right]$ if $R\,=\,\infty$.
 \end{itemize}
\end{Prop}

\begin{proof}
Let $x\,\in\,\mathbb{R}$ be arbitrary. With notation as in the proof of Theorem
\ref{TH-EXTEND-FINITE-R},
for some $\rho_{x}>0$ the series in of $z$ of
$
h_{{x}}{\circ\,}T_{x}\,{\circ} {\mathrm{g}}_R$ is
absolutely convergent
 for $|z|\,<\,\rho_{x}\,$, and necessarily, given
 our choice of ${\mathrm{g}}^R$,
 $$(-\rho_{x},\,\rho_{x})\,\subset\, \left(
 \frac{R}{\pi}\left(1-e^{\frac{\pi}{R}\epsilon_1(x)}\right),\,
  \frac{R}{\pi}\left(1-e^{-\frac{\pi}{R}\,\epsilon_2(x)}\right)\right)\,
  \subset\left(-\infty,\,\frac{R}{\pi}\right)$$ with
  $\epsilon_1(x)\,:=\,\sup_{t\,>\,0}\{\,t\,|\,(x-t,\,x)\subset
  \mathbb{R}\setminus{f_1^{-1}(0)}\}$,
  $
  \epsilon_2(x)\,:=\,\sup_{t\,>\,0}\{\,t\,|\,(x,\,x+t)\subset
  \mathbb{R}\setminus{f_1^{-1}(0)}\}.$
 Thus, we may take
 \begin{eqnarray}\notag
 \rho_{x}\,&=&\,
 \min
  \left\{\frac{R}{\pi}\left(e^{\frac{\pi}{R}\epsilon_1(x)}-1\right),\,
  \frac{R}{\pi}\left(1-e^{-\frac{\pi}{R}\,\epsilon_2(x)}\right)
  \right\}\\\notag
  &\overset{(i)}{\geq}&\,
 \min \left\{\frac{R}{\pi}\left(1-e^{-\frac{\pi}{R}\epsilon_1(x)}\right),\,
  \frac{R}{\pi}\left(1-e^{-\frac{\pi}{R}\,\epsilon_2(x)}\right)\right\}
  \\\notag
  &\overset{(ii)}{=}&\,\frac{R}{\pi}\left(1-e^{-\frac{\pi}{R}\min\{\epsilon_1(x),
      \,\epsilon_2(x)\}}\right)
      \\\notag
  &\overset{(iii)}{\geq}&\,\begin{cases}
  \frac{R}{\pi}\,\textrm{ if } \vartheta \,\leq\,0\,,
  \\\\ {\displaystyle
  \frac{R}{\pi}
  \left(1-e^{-\frac{\pi^2}{2\,R\sqrt{\sup{\vartheta}}}}
  \right)}\,
  \textrm{ if } \sup{\vartheta}\,>\,0\,,
  \end{cases} \end{eqnarray}
  where (i) is valid since $e^t-1\,>\,1-e^{-t}$ for $t\,>\,0$;
   (ii) holds because $1-e^{-t}$ is an increasing function;
  inequality (iii) follows from a standard Sturm comparison
  of the location of zeros of  solutions of
  $f^{\prime\prime}(t)+\vartheta(t)\,f(t)\,=\,0$, with
  those of $\cos(\sqrt{K}(t-x))$,
  with $K\,=\,\sup{\vartheta}$.

Now, use \ref{SUPPORT-MU-depend-on-x} and vary $x_0$.
  \end{proof}

\section{Higher order Schwarzians and related functions in the tangent bundle}
\label{SECTION-SCHWAR-CURV}

The higher order Schwarzians ${\mathrm{S}}^{\,\vartheta,\,R}_n$ defined on $\mathbb{R}$ have canonical expressions
 in $\vartheta$ and its derivatives which Tamanoi proved \cite{Tamanoi} (See Remark \ref{TPPI}).
 This will be the basis for our definition of the  functions ${\mathrm{S}}^{\,\sigma,\,R}_n$ in
the tangent bundle .

 We will
 re-derive Tamanoi's  expressions in the spirit of the last section in
 reference
 \cite{LESZ}, using Proposition \ref{H-SCHWARZ} and  the
 equation (\ref{DEQA})
 directly, in order to keep track of the  weight $\operatorname{w}$ introduced  below used  to describe  homogeneity properties
of the  functions ${\mathrm{S}}^{\,\sigma,\,R}_n$.

\subsection{Schwarzians in terms of curvature via Tamanoi polynomials}

\begin{Def}\label{t-degree}
A monomial in  $\{\,\vartheta, \vartheta^{\prime}, \ldots, \vartheta^{(r)}, \ldots\}$
is said to have weight $\operatorname{w}$  according to
$$ \operatorname{w} \left\{\prod_{i=1}^n
{\left(\vartheta^{(r_i)}(t)\right)}^{l_i}\right\}
=\sum_{i=1}^n (2+r_i)\,l_i.
$$

A polynomial all of whose terms have the same $\operatorname{w}$  will
be called \emph{isobaric}.
\end{Def}
\begin{Prop}\label{Prop-degree}
Let $m=\prod_{i=1}^n {\left(\vartheta^{(r_i)}\right)}^{l_i}(t)$ be a monomial
in $\vartheta$ and its derivatives.
Then \begin{itemize}
\item[(i)]$\operatorname{w}\{\vartheta(t)\, m(t)\}=\operatorname{td}\{m\}+2$\,;
\item[(ii)] $\operatorname{w}\{m^\prime(t)\}=\operatorname{w}\{m\}+1$\,.
\end{itemize}
\end{Prop}
\begin{proof}
It is enough to show this for $m\,=\,{\left(\vartheta^{(r)}\right)}^l$.
Now, (i) is
clear since $\operatorname{w}\{\vartheta\}\,=\,2$. For (ii), note
$m^\prime\,=\,l\, {\left(\vartheta^{(r)}\right)}^{l-1}\,\vartheta^{(r+1)}$ and hence
$$\operatorname{w}\{m^\prime\}=(2+r)\,(l-1)+(2+r+1)=(2+r)\,l+1
=\operatorname{w}\{m\}+1\,.$$
\end{proof}

We note, $\operatorname{w}\{\vartheta\}=2$,
$\operatorname{w}\{\vartheta^\prime\}=3$,\,
$\operatorname{w}\{\vartheta^{\prime\prime}\}=4$,\,
$\operatorname{w}\{\left(\vartheta^\prime\right)^2\}=
\operatorname{w}\{\vartheta^{(4)}\}=6$,
hence the Schwarzian
${\mathcal{S}}_7^{\,\vartheta}=\vartheta^{(4)}
         +76\,\vartheta^{\prime\prime} \vartheta
         +52\,\left(\vartheta^\prime\right)^2
       +272\,\vartheta^3$ is isobaric with weight $\operatorname{w}$ equal to $6$.

In general, we have the following

\begin{Prop}\label{Schwarzians-in-derivatives}
The Schwarzian ${\mathcal{S}}_n^{\,\vartheta}$ can
be expressed as a polynomial in $\vartheta$ and its derivatives up to order $n-3$,
\begin{equation}\label{PI-N}
\pi_n(\vartheta, \,\vartheta^{\prime},\, \ldots,\, \vartheta^{(n-3)}),
\end{equation}
 which is isobaric of weight $\operatorname{w}$ equal to $n-1$.
\end{Prop}
\begin{proof} For convenience of reference  let us recall here that
by Proposition \ref{H-SCHWARZ}, at a point $x$
\begin{equation}\label{S=H}
{\mathcal{S}}^{\,\vartheta}_n(x)=h_{x}^{(n)}(x)
\end{equation}
where
$h_{x}(t)\,=\,\frac{f_{x,\,2}(t)}{f_{x,\,1} (t)},
$
 with $f_{x,\,1}$ and $f_{x,\,2}$ as in Definition \ref{Def-FPS}, that is,
they are the  solutions of
\begin{equation}\label{EQ}
f^{\prime\prime}(t)+\vartheta(t) f(t)\,=\,0
 \end{equation}
  with the initial conditions so that
\begin{equation}\label{IC}
f_{x,\,1}(x)\,=\,f_{x,\,2}^\prime(x)\,=\, h_{x}^\prime(x)\,=\,1,
\quad f_{x,\,2}(x)\,=\,f^\prime_{x,\,1}(x)\,=\, h_{x}(x)\,
=\,h_x^{\prime\prime}(x)\,=\,0.
\end{equation}

Now, (\ref{EQ}) and two derivatives with respect to $t$
of
$
h_{x}(t)\,f_{x,\,1}(t)\,=\, f_{x,\,2}(t)$  gives
\begin{equation}\label{BF}
 2\,h_{x}^\prime(t)\,f_{x,\,1}^\prime(t)
 + h_{x}^{\prime\prime}(t)\,f_{x,\,1}(t)\,=\,0.
 \end{equation}  By use  of (\ref{EQ}) the $n$-th
 derivative with respect to  $t$ of the left-hand side of (\ref{BF}) is written
as
\begin{equation}\label{REC}
\beta_{x,\,n}(t)\, f_{x,\,1}^\prime(t)+ \alpha_{x,\,n}(t)
\, f_{x,\,1}(t)\,=\,0  \, ,
\end{equation}
 where for $n\,=\,1,\,2,\, \cdots$ the coefficients of
 $f_{x,\,1}(t)$ and $f_{x,\,1}^\prime(t)$
 are   computed by
$$\left(\begin{array}{c}
\alpha_{x,\,n}(t)\\
\beta_{x,n}(t)
\end{array}\right)=
\left(\begin{array}{cc}
\partial_t&-\vartheta(t)\\
1&\partial_t
\end{array}\right)^n
   \left(\begin{array}{c}
   h_{x}^{\prime\prime}(t)\\
      2\,h_{x}^\prime(t)
\end{array}\right)
$$
with $(\partial_t)^k y$ meant to indicate $y^{(k)}$.

Consider as inductive hypothesis that for $n\in \mathbb{N}$ it is true that
\begin{eqnarray}\label{IH1}
\alpha_{x,\,n}(t)\,&=&\,h_x^{(n+2)}(t)+p_{x,\,n}(t),\\\label{IH2}
\quad \beta_{x, \,n}(t)\,&=&\,h_{x}^{(n+1)}(t)+q_{x, n}(t),
\end{eqnarray}
where both $p_{x,\,n}(t)$ and $q_{x,\,n}(t)$ are
integer linear combinations of terms of the form
\begin{equation}\label{TYPICAL}
h_{x}^{(k)}(t)\,\prod_{i=1} \left(\vartheta^{(r_i)}(t)\right)^{l_i}
\end{equation}
with their indices $k$, $r_i$ and $l_i$ restricted for
the terms in $p_{x,\,n}(t)$ by
\begin{equation}\label{IH3}
k\,+\,\sum_{i} (2+r_i)\,l_i \,=\,n+2\,,\;\;\;\;\; k\leq n+1\,,
\end{equation}
while for the terms in $q_{x,\,n}(t)$ those indices are restricted by
\begin{equation}\label{IH4}
k\,+\,\sum_{i} (2+r_i)\,l_i \,=\, n+1\,,\;\;\;\; k\leq n\,.
\end{equation}
 Since
$\alpha_{x,\,1}(t)\,=\,\,
h_x^{\prime\prime\prime}(t)-2\vartheta(t)\,h^\prime_{x}(t)$
and $\beta_{x, \,1}(t)\,=\,-h_{x}^{\prime\prime}(t)$
then (\ref{IH1}) through
(\ref{IH4}) hold for $n=1$.
Moreover, by Proposition \ref{Prop-degree} applied to
$$
\alpha_{x,\,n+1}(t)=\alpha^\prime_{x,\,n}(t)-
\vartheta(t) \;\beta_{x,\,n}(t)\;,\;\;\;
\beta_{x, n+1}(t)=\beta^\prime_{x,\,n}(t)
+\alpha_{x,\,n}(t)
$$
it follows that  (\ref{IH1}) through
(\ref{IH4})
are valid for $n+1$ and hence for all $n\in \mathbb{N}$.

Now make the inductive hypothesis ``${\mathcal{S}}^{\,\vartheta}_{n+1}$
 is isobaric of weight $\operatorname{w}$
equal to $n$", which certainly holds for $n=0$, since
${\mathcal{S}}_1^{\,\vartheta}=1$.
We have shown that for all $n$
\begin{eqnarray}\label{alpha-n}
\alpha_{x, n}(t)\,&=&\,h_{x}^{(n+2)}(t)+\sum a_{k, r_i, l_i}\,
h_{x}^{(k)}(t)\,\prod_{i=1} \left(\vartheta^{(r_i)}(t)\right)^{l_i}
\end{eqnarray}
with $a_{k, r_i, l_i}$ integers
and with  the indices $k$, $r_i$ and $l_i$ that satisfy (\ref{IH3})\,. Thus
since by (\ref{REC}) we have $\alpha_{x, n}(x)=0$
then putting $t=x$ in (\ref{alpha-n}) gives
\begin{eqnarray}
\mathcal{S}^{\,\vartheta}_{n+2}(x) \,=
\,-\sum a_{k, r_i, l_i}\,\mathcal{S}^{\,\vartheta}_{k}(x)\,
\prod_{i=1} \left(\vartheta^{(r_i)}(x)\right)^{l_i}\,.
\end{eqnarray}
But since  according to (\ref{IH3}) $k\leq n+1$ using the inductive hypothesis
for the $\mathcal{S}^{\,\vartheta}_{k}(x)$ we get
that $$\operatorname{w}\left\{
\mathcal{S}^{\,\vartheta}_{k}(x)\,
\prod_{i=1} \left(\vartheta^{(r_i)}(x)\right)^{l_i}\right\}=
\operatorname{w}\left\{\mathcal{S}^{\,\vartheta}_{k}(x)\right\}+
\operatorname{w}\left\{
\prod_{i=1} \left(\vartheta^{(r_i)}(x)\right)^{l_i}\right\}=
k-1\,+\,\sum_{i} (2+r_i)\,l_i \,=\,n+1,
$$
the last equality of course in light of (\ref{IH3}).
Thus ${\mathcal{S}}^{\,\vartheta}_{n+2}$ is isobaric with weight $\operatorname{w}$ equal to
$n+1$.
\end{proof}

\begin{Rem}\label{TPPI}
Let $T_n$ be the Tamanoi polynomials \cite{Tamanoi}
. Then
$$\pi_n(x_1,\,\cdots,\, x_{n-3})\,=\,T_{n-1}(2\,x_1,\,\cdots,\,2\,x_{n-3})$$
Accordingly our weight  equals Tamanoi's
``virtual order" plus one.
\end{Rem}

\subsection{Schwarzians in the tangent bundle}
\label{SECTION-SCHWAR-CURV}\
\vskip.2cm
We transplant the previous construction to the context of Riemannian manifolds.
Consider $M$ to be a smooth manifold and $\pi\colon TM \to M$ its tangent bundle. Let $g$ be a smooth Riemannian metric on $M$
 with $\nabla$ the corresponding Levi-Civita connection and
  $\sigma\colon M\to \mathbb{R}$ the Gauss curvature; the
  pull-back  $$\pi^*\sigma\,=\,\sigma{\circ}\pi\colon TM\to \mathbb{R}$$
  is also denoted by $\sigma$, to simplify notation.
  For any $\mathrm{v}\in TM$ we put $\|\mathrm{v}\|^2\,=\,g(\mathrm{v},\mathrm{v})$.

\begin{Def}\label{DEF-S-TM}
The \textit{higher order Schwarzians in $TM$},
$$  {\mathrm{S}}^{\,\sigma}_n\colon TM \to \mathbb{R}$$
 for $n\,\geq\,1$,  are defined by performing in  the Schwarzians
 on $\mathbb{R}$
 from Proposition \ref{Schwarzians-in-derivatives},
 $${\mathcal{S}}_n^{\,\sigma}\,=\,\pi_n\,( \sigma,\,\sigma^{\prime},\,
 \ldots,   \,\sigma^{(n-3)}\,)$$ the substitutions
 \begin{equation}\label{SUBSTITUTE}
\sigma^{(k)}\mapsto \|\mathrm{v}\|^2 \,\nabla^k_\mathrm{v} \sigma \,.
 \end{equation}
That is, at $\mathrm{v}\,\in TM$,
  \begin{equation}\label{DEF-HOS-TM-PIN}
  {\mathrm{S}}^{\,\sigma}_n(\mathrm{v})\,:
  =\,\pi_n\,(\,\|\mathrm{v}\|^2 \,\, \sigma,\,
  \|\mathrm{v}\|^2 \,\,\nabla_\mathrm{v} \sigma,\,
  \ldots ,\, \|\mathrm{v}\|^2\,\,\nabla_\mathrm{v}^{n-3} \sigma\,)\,.
  \end{equation}
\end{Def}
\begin{Rem}
The justification of the factor $\|\mathrm{V}\|^2$ is given in Remark
 \ref{CURV-INTEGR}.
\end{Rem}
\begin{Def}\label{DEF-S-R-in-TM}
Given $R\in (0, \,\infty\,]$  and $g_R$ as in (\ref{GR-PICK})
 for $n\,\geq\,1$ the functions
 $${\mathrm{S}}^{\,\sigma,\,g_R}_{n}\colon TM\to \mathbb{R}$$
 are defined by
\begin{equation}\label{Lin-Comb-TM}\fbox{$\displaystyle
\mathrm{S}^{\,\sigma,\,g_R}_{n}(\mathrm{v})\,=\,
\sum_{k=1}^n\,
g_{R, \,n,\,k}\;\|\mathrm{v}\|^{n-k}\;
{\mathrm{S}}^{\,\sigma}_k(\mathrm{v})$},\,
\end{equation}
with constants $g_{R, \,n,\,k}$ according to (\ref{Lin-Comb-TM})
and with
  ${\mathrm{S}}^{\,\sigma}_n$ given by (\ref{DEF-HOS-TM-PIN}).
\end{Def}
\begin{Rem}\label{REM-SF-LIMITS}
It follows by  Remark \ref{LAMBDA-FOR-R-INFTY} that,
 $$\fbox{${\displaystyle \mathrm{S}^{\,\sigma,\,g_R}_{n}  \textrm { reduce in the limit $R\,=\,\infty$ to }  \mathrm{S}^{\,\sigma}_n}$}.$$
\end{Rem}
\begin{Rem}\label{onM}
If $M$ is  the zero section of $ TM$, then from (\ref{Lin-Comb-TM}),
$$\fbox{$\displaystyle
{\mathrm{S}}^{\,\sigma,\,g_R}_1\,|_{M}\,\equiv 1,\,\, \;{\mathrm{S}}^{\,\sigma,\,g_R}_n\,|_{M}\,\equiv 0,\, n\neq 1$}\,$$.

For higher $n$ and $\displaystyle g_R(z)\,=\,-\frac{R}{\pi}\,\ln \left(1-\frac{\pi\,z}{R}\right)$ see Proposition \ref{PROP-GR-LN-in-R-}, Corollary \ref{COR-GR-LN-COMP}\,.
\end{Rem}
\begin{Def}\label{DEF-DET-R-TM}
Given $R\,\in\,(0, \,\infty]$,  for $\mathrm{v}\in TM$ and integer
$n\,\geq\,1$ put
\begin{eqnarray}\label{DNMR}
{\mathrm{D}}^{\,\sigma, \,g_R}_{n}(\mathrm{v})\,&:=&\,
\det\left[
\frac{{\mathrm{S}}^{\,\sigma, \,g_R}_{i+j-1}(\mathrm{v})}{
(i+j-1)!}\right]_{\,i,\,j\,=\,1}^{n}\,.
\end{eqnarray}
\end{Def}

The following is used in the first part of the proof of Theorem \ref{TH-EXTEND-TUBE-FINITE}, in computations and in Section \ref{SECT-RANK}.

\begin{Prop}\label{HAT-SIGMA-HOMOGENEOUS}
The functions $ {\mathrm{S}}^{\,\sigma,\,g_R}_n $ and \,$\mathrm{D}^{\,\sigma, \,g_R}_n$ are
homogeneous   along the fibers of $TM$ of
degree $n-1$ and  $\,n(n-1)$ respectively.
They are all real analytic if $g$ is.
\end{Prop}

\begin{proof}
Given the integer $k\,\geq\,0$, by the linearity of $\nabla$ in its lower argument, the function
$\|\mathrm{v}\|^2 \;\nabla_\mathrm{v}^k \sigma$ is homogeneous
of degree
$2+k$ along the fibers of $TM$; thus
$$
\textrm{degree of } \prod_{i=1}^n (\|\mathrm{v}\|^2\;\nabla_\mathrm{v}^{k_i}\sigma)^{l_i}\,=\,
\sum_{i=1}^n (2+k_i)\,l_i \,=\,\operatorname{w}
\left\{\prod_{i=1}^n {\left(\sigma^{(k_i)}\right)}^{l_i}\right\}
,$$
where for the second equality we have used
 $2+k\,=\,\operatorname{w}\{\sigma^{(k)}\}$
 by Definition \ref{t-degree}.

  Pair this with  (\ref{DEF-HOS-TM-PIN})
  and the fact the polynomials $\pi_n$ are  isobaric
  with  $\operatorname{w}\{\pi_n\}\,=\,n-1$ by Proposition
  \ref{Schwarzians-in-derivatives}, and
we have that for $ \lambda\neq \,0$,
$${\mathrm{S}}^{\,\sigma}_n(\lambda\,\mathrm{v})\;=\;\lambda^{n-1}
\,{\mathrm{S}}^{\,\sigma}_n(\mathrm{v})\,,$$
which  used in (\ref{Lin-Comb-TM}) shows the homogeneity property claimed,
\begin{equation}\label{TD=HD}{\mathrm{S}}^{\,\sigma,\,g_R}_{n}(\lambda\,\mathrm{v})\,=\,\lambda^{n-1}\, {\mathrm{S}}^{\,\sigma,\,g_R}_{n}(\mathrm{v})\,.
\end{equation}
In light of this, it follows that for every non-negative integer $n\,\geq \,2$ and any $\mathrm{v}\,\neq \,0$
\begin{equation}\label{DET-RESC}
\det\left[\frac{ {\mathrm{S}}^{\,\sigma,\,g_R}_{(i+j-1)}
( \lambda\,\mathrm{v})}{(i+j-1)!}
\right]_{\,i,\,j\,=\,1}^{n}\;
=\;\lambda^{N} \,
\det\left[
 \frac{ {\mathrm{S}}^{\,\sigma,\,g_R}_{(i+j-1)}( \mathrm{v} )}{(i+j-1)!}
      \right]_{\,i,\,j\,=\,1}^{n},
\end{equation}
where $N\,=\,n(n-1)$; indeed, for a given $n$,
the determinant that gives  $\mathrm{D}^{\,\sigma, \,g_R}_{n}$ consists of a sum of terms
of the form $\displaystyle \prod_{i\,=\,1}^n {\mathrm{S}}^{\,\sigma,\,g_R}_{(i+ \pi(\,i\,)-1)}(\mathrm{v})$,
for some permutation $\operatorname{p}$ of $\{\,1,\,\ldots,\,n\,\}$, and
hence, by (\ref{TD=HD}),  of common degree $$N\,=\,\sum_{i\,=\,1}^n\, i
+\operatorname{p}(\,i\,)-2\,=\,\sum_{i\,=\,1}^n\,2\,(i-1)\,= \,2\,\sum_{k\,=\,1}^{n-1}\,k\,=\,2\,\binom{n}{2}\,=\,n(n-1)\,,$$
the homogeneity claimed for $\mathrm{D}^{\,\sigma, \,g_R}_{n}$.

Finally,
 if $g$ is real analytic so is $\sigma$ and its covariant derivatives $\nabla_\mathrm{v}^{k} {\sigma}$ as a function of $\mathrm{v}$, and hence so is $\pi_n\,(\,\|\mathrm{v}\|^2 \,\sigma, \,\ldots , \,\|\mathrm{v}\|^2\,\,\nabla_\mathrm{v}^{n-3} {\sigma}\,)$.
\end{proof}

\section{ Theorem \ref{TH-EXTEND-TUBE-FINITE}. Gauss curvature and the adapted structure on $T^RM$
}\label{SECTION-GT}
We show that the functions introduced in the previous sections give a
characterization of the existence  of the adapted structure in $T^RM$ for a given real analytic and complete Riemannian metric $g$ on $M$.

 For $M$ compact, $T^RM$ with the adapted complex structure is a model for a Grauert tube $X^R$, which by definition is a
complex manifold of dimension $n$ with
an exhaustion function $u\colon X^R\to \mathbb{R}$, $u\,\geq\,0$,  $\sup { u} \,=\,R$,
satisfying,
in local holomorphic coordinates $z_1, \,\ldots\,,z_n$\,:
\begin{itemize}
\item[(i)] $\det \left[{\partial^{\,2} u^2}/{\partial z_i\, \partial \bar{z_j}}\right]\,>\,0$ and \item[(ii)] $\det \left[{\partial^{\,2}{u}}/{\partial z_i \partial \bar{z_j}}\right]\,=\,0\,,$
 valid off  the closed real analytic manifold $M\,:=\,{u}^{-1}(0)$, the fixed-point set of a anti-holomorphic involution of $X^R$.
\end{itemize}
   By (ii) there is a K\"{a}hler metric in $X^R$ with potential ${u}^2$
   which by restriction gives  $M$ a Riemannian metric $g$;
   $(M,\,g)$ is called \emph{the center} of the Grauert tube.

Besides the model for  $X^R$ in the tangent bundle
by Lempert and Sz\H{o}ke where the geometry of the center
$M$ is emphasized, there is the
    one in the cotangent bundle due
 to V. Guillemin and M. Stenzel \cite{GLLST}) where preeminence is given to the Monge-Amp\`{e}re function.

It is on that model of Lempert and Sz\H{o}ke where we base our constructions. There
 $X_R$ is represented as the real manifold
\begin{equation}\label{DEF-TRM}
T^RM=\{\mathrm{v}\in TM\,|\,g(\mathrm{v},\,\mathrm{v})\,<\,R^2\}
\end{equation}
 endowed with the  \emph{adapted complex structure}
 induced by $g$ \cite{LESZ}. Under this
 identification, for all $\mathrm{v}\in TM$,
 $u^2(\mathrm{v})\,=\,g(\mathrm{v},\mathrm{v})\,
 =\,\|\mathrm{v}\|^2\,=\,2\,E$, while the
 map $\mathrm{v}\mapsto -\mathrm{v}$ is the antihololomophic
 involution. The metric $g$ must be real analytic,
 by a result of Lempert  \cite{LEMPREG}.

As long as  $g$ is complete and real analytic,
the adapted complex structure is always defined in some neighborhood
$\mathcal {U}$ of $M$ in $TM$.
However, when $M$ is not compact such $R>0$ might not exist,
for example if along a geodesic the Gauss curvature is not bounded below thus violating Lempert-Sz\H{o}ke inequality for any $R>0$. If $M$ is compact such radius of course exist, but its value is severely limited by the geometry of $M$.

\begin{thm} \label{TH-EXTEND-TUBE-FINITE}  Consider a  complete  real analytic Riemannian metric $g$ on the two-dimensional $M$.
Let $R\in (0, \, \infty]$ be given and  $\mathcal{V}\subset \,TM\,\setminus{M}$
 any set that maps onto the unit tangent bundle $UM$
by the assignment $\mathrm{v}\,\mapsto \|\mathrm{v}\|^{-1}\,\mathrm{v}$.
Then the adapted complex structure exists up to radius $R$ if and only if for all integer $n\,\geq\,1$ and
 for all $\mathrm{v}\in \,\mathcal{V}$,
\begin{equation}\label{sequence}
\mathrm{D}^{\,\sigma,\,g_R}_{n}(\mathrm{v})\,\geq \,0\,.
 \end{equation}
\end{thm}

\begin{proof} Since
by Proposition \ref{HAT-SIGMA-HOMOGENEOUS} the functions ${\mathrm{D}}^{\,\sigma, \,g_R}_n $
are homogeneous   along the fibers of $TM$, their signs are constant along each ray $\{ e^t\,\mathrm{v}\,, t\in \mathbb{R}\}\subset TM\setminus M$ and so, without loss of generality we take $\mathcal{V}\,=\,UM$, the unit tangent bundle.

 Consequently, we consider unit
 speed geodesics in our recollection of the relevant parts of the construction of the adapted complex structure \cite{LESZ}.

For a given unit-speed geodesic $\gamma$ in $M$
 consider two Jacobi fields along $\gamma\colon\mathbb{R}\to M$, $J^{\,\gamma}_1$ and $J^{\,\gamma}_2$,  linearly independent and point-wise orthogonal to $\dot{\gamma}$, and define $h^{\,\gamma}\colon \mathbb{R}\setminus\mathbb{S^{\,\gamma}}\to \mathbb{R}$ by
$$J^{\,\gamma}_2(t)\,=\,h^{\,\gamma}(t)\,J^{\,\gamma}_1(t),$$
with $t$ arc-length and
$\mathbb{S^{\,\gamma}}\,:=\,\{\,t\,\in \,\mathbb{R}\;|\;J^{\,\gamma}_1(t)\,=\,0\,\} \subset \mathbb{R}$.

Consider the vector fields
$\tilde{J}^{\,\gamma}_i$, $i=1,2$
along the map
$\mathbf{P}_\gamma\colon \mathbb{C}\to TM$, where
\begin{equation}\label{MA-LEAF-MAP}
\mathbf{P}_\gamma (\,
x+\sqrt{-1}\,y\,) \,=\, y\,\dot{\gamma}(x)\,, x, y \in \mathbb{R}\,,
\end{equation}
given  at  $w\,=\,x+\sqrt{-1}\,y$,
by
\begin{equation}\label{J-LIFT}
\tilde{J}^{\,\gamma}_i|_w \,=\,(J^{\,\gamma}_i(x))^{\,h}_{ \mathbf{P}_\gamma (w)} +
 (\nabla_{ \mathbf{P}_\gamma (w)} J^{\,\gamma}_i)^{\,v}_{ \mathbf{P}_\gamma (w)},
\end{equation}
where, we recall,  the \emph{horizontal lift} $(u)^{\,h}_\mathrm{v}$ and the
\emph{vertical lift} $(u)^{\,v}_\mathrm{v}$ lift
of a vector $u \in T_{\pi \mathrm{v}} M$ are the vectors in  $T_\mathrm{v}(TM)$  defined by $$\pi_* (u)^{\,h}_\mathrm{v}\,=\,K(u)^{\,v}_\mathrm{v}\,=\,u\,,\;\;\pi_*(u)^{\,v}_\mathrm{v}\,=\,K (u)^{\,h}_\mathrm{v}\,=\,0\,,$$ with $K\colon T(TM)\to TM$ the connection map \cite{Klingenberg}, \cite{LESZ}.

Then, there is $\,R\,>\,0$ such  the adapted complex structure $\mathbf{J}$ exists on
(\ref{DEF-TRM})
if and only if for every unit-speed geodesic,
there is a meromorphic extension of $h^{\gamma}$, $$h^{\gamma}_\mathbb{C}\colon \{ \,x +\sqrt{-1}\,y\,;\; x, y \in \mathbb{R}, \,|y|\,<\,R\,\}\,\subset\,\mathbb{C}\to \mathbb{C}\cup \{\infty\}$$ with poles only on
$\mathbb{S^{\,\gamma}}\,:=\,\{\,t\,\in \,\mathbb{R}\;|\;J^{\,\gamma}_1(t)\,=\,0\,\} \subset \mathbb{R}$ and with positive imaginary part for $y\,>\,0$, such that
$$
\tilde{J}^{\,\gamma}_2-\sqrt{-1}\,\mathbf{J}\,\tilde{J}^{\,\gamma}_2 =
h^{\,\gamma}_\mathbb{C}(x+\sqrt{-1} y)
\;\left(
\tilde{J}^{\,\gamma}_1-\sqrt{-1}\,
\mathbf{J}\,\tilde{J}^{\,\gamma}_1\right)\,.
$$

 On the other hand, along every unit-speed geodesic $\gamma\colon \mathbb{R}\to M$
the map $$t\mapsto f(t)\, \star\dot{\gamma}(t),$$
 where $t\mapsto \{\,\dot{\gamma}(t),\; \star\dot{\gamma}(t)\,\}$ is an orthonormal frame along $\gamma$,  is a Jacobi field if and only if   \begin{equation}\label{DE-GAMMA}
f^{\prime\prime}(t)+\sigma(\gamma(t))\,f(t)\,=\,0\,;
\end{equation}
hence
$h^\gamma\,=\,{f^{\,\gamma}_2}/{f^{\,\gamma}_1}$
where ${f^{\,\gamma}_1}$ and $f^{\,\gamma}_2$ are two independent solutions of (\ref{DE-GAMMA}).

Now apply Theorem \ref{TH-EXTEND-FINITE-R} to every unit-speed $\gamma$
using the relation (\ref{INTRO-S-R-TM}), that establishes the
correspondence (\ref{Theta-Sigma}), which since we identify $\sigma$
with $\pi^*\sigma$,
becomes
$$
 \fbox{$\vartheta\,\leftrightarrow\,\sigma\circ{\gamma}$}\,.$$
 Thus
 for all integer $k\,\geq \,1$ we have the correspondences
 $$\fbox{$\displaystyle
    \vartheta^{(k)}\,\leftrightarrow\,\nabla^k_{\dot{\gamma}}\sigma
    $}
 $$
and
 $$\fbox{$\displaystyle
    \mathcal{S}^{\,\vartheta, \,g_R}_k\,
    \leftrightarrow\,\mathrm{S}_k^{\,\sigma, \,g_R}\circ{\dot{\gamma}}
    $}
 \,.$$
\end{proof}
In some cases it suffices to check the inequalities in a smaller
$\mathcal{V}\subset TM$. For instance we have the following.

\begin{Cor}\label{COR-REAL-ANALYTIC}
 Let the Riemannian metric $g$ on $M$ be complete and  real
 analytic, $G$ a subgroup of the isometry group and
 $\mathcal{Z}\,\subset \,TM\setminus{M}$
  any subset of $TM$ whose orbit corresponding
      to the $G$-action induced by differentials meets every trajectory of the geodesic flow.
The  adapted complex structure exists up to radius $R$ if and only if the inequalities (\ref{sequence}) hold
 for all integer $n\,\geq\,1$ and for all $\mathrm{v}\in \mathcal{Z}$.
 \end{Cor}
 \begin{proof}
 By  part 2 of Theorem \ref{TH-EXTEND-FINITE-R} and
 since the orbits of the geodesic flow on $UM$
 correspond by projection to the geodesics of $M$,
 in light of Remark \ref{D-at-x0} applied to the context of  $TM$, the
real analyticity of $g$ implies that the conditions  (\ref{sequence})
need to be checked at just one point along
each orbit of the geodesic flow. As noted earlier
due to the
re-scaling property of the polynomials $\pi_n$ shown
in Proposition \ref{HAT-SIGMA-HOMOGENEOUS}
the
inequalities in Theorem \ref{TH-EXTEND-TUBE-FINITE}
hold along lines through the origin
in the fibers of $TM$, and thus $\mathcal{Z}$ does
not have to be contained in the unit tangent bundle   $UM$.
 \end{proof}
 \begin{Rem} As an example, consider  a real analytic
 metric  of revolution
  two-sphere $M\,=\,S^2$ isometrically embedded
  in $\mathbb{R}^3$ and with positive Gauss curvature. Then Corollary
 \ref{COR-REAL-ANALYTIC} with
 $G=S^1$ applies, with $\mathcal{Z}\,=\, U_{x}M$ for $x$ any point in the equator of $M$.
\end{Rem}

\begin{Rem}\label{CURV-INTEGR}
On $TM\setminus M$, there is a canonical  frame field
$\displaystyle z\mapsto\{X_1, X_2, V_3, V_4\}$ defined by the Levi-Civita connection and the endomorphism $S\colon TM\to TM$ given by a positive rotation of  $\pi/2$ with respect to $g$.
\begin{itemize}
\item[(1)]
  $X_1$ is the geodesic spray, the geodesic flow given by
  $\displaystyle (\mathrm{v},\,t)\mapsto \phi_{t} (\mathrm{v})\,:=\,\dot{\gamma}_\mathrm{v}(t)$
  with $\gamma_\mathrm{v}$  the constant speed geodesic with $\dot{\gamma}_\mathrm{v}(0)=\mathrm{v}$;
\item[(2)]
$V_1$ is the radial vector field,
$\displaystyle (\mathrm{v},\,t)\,\mapsto\,e^t\,\mathrm{v}$;
\item[(3)]
 $X_2$  has as flow
 $\displaystyle (\mathrm{v},t)\,\mapsto\, \mathrm{j}\,\phi_t\,{\mathrm{j}}^{-1}(\mathrm{v})
 =\mathrm{j}\,\dot{\gamma}_{{\mathrm{j}}^{-1}\mathrm{v}}(t)$,
 where  and ${\gamma}_{{\mathcal{\mathrm{j}}}^{-1}\mathrm{v}}$ is the unit speed geodesic with $\dot{\gamma}_{{\mathrm{j}}^{-1}\mathrm{v}}(0)=-\mathrm{j}\mathrm{v}$;
\item[(4)]
  $V_2$ is the angular field with flow
  $\displaystyle (\mathrm{v},\,t)\mapsto \,\cos(t)\,\mathrm{v}+\sin(t)\,\mathrm{j}\mathrm{v}.$
\end{itemize}
 In terms of the horizontal and vertical lifts  we have
\begin{eqnarray}\notag
 X_1 |_\mathrm{v} \,=\, (\mathrm{v})^{\,h}_\mathrm{v}, \quad
V_1|_\mathrm{v} \,=\, (\mathrm{v})^{\,v}_\mathrm{v}, \quad
X_2|_\mathrm{v} \,=\, (\mathrm{j}\mathrm{v})^{\,h}_\mathrm{v}, \quad
V_2 |_\mathrm{v} \,=\, (\mathrm{j}\mathrm{v})^{\,v}_\mathrm{v},
\end{eqnarray}
and
the relation with the notation as above
by the equation for normal Jacobi fields
$$(\phi_t)_*
\left[
\begin{array}{c}
X_{2}\\
V_{2}
\end{array}
\right]_\mathrm{v}
  \,=\,\left[\begin{array}{cc}
f_1&f_1^\prime\\
f_2&f_2^\prime
\end{array}\right]_t\;
\left[
\begin{array}{c}
X_{2}\\
V_{2}
\end{array}
\right]_{\phi_{t}\mathrm{v}},
$$
interpreted along  unit-speed geodesics $\gamma\colon \mathbb{R}\to M$
where the assignments $$s \mapsto J_i(t)\,=\, f_i(t) \,X_2|_{\phi_t \dot{\gamma}(0)}, i\,=\,1,\,2$$ define linearly independent Jacobi fields along $\gamma$. Now, using these Jacobi fields one proves that
\begin{equation}\label{bracket}
[X_1, X_2]_{\mathrm{v}}\,=\, \sigma \,\|\mathrm{v}\|^2 \; \;V_2,
\end{equation}
which of course defines  the Gauss curvature $\sigma$ as the obstruction for the integrability of the horizontal distribution determined by the connection and spanned by $X_1$ and $X_2$. The equation (\ref{bracket}) is the reason for the introduction of the factor $\|\mathrm{v}\|^2$ in (\ref{SUBSTITUTE}), so that the Schwarzians in $TM$ re-scale correctly along the fibers. It also indicates that the Schwarzians in $TM$ can be defined equivalently in terms to derivatives of $\sigma$ along the geodesic flow.

In fact the version in $TM$ of the formula
(\ref{HOS-REC-FORMULA})
is
\begin{equation}
 \mathrm{S}^{\,\sigma}_{n+1}(\mathrm{v})
 \,=\,\operatorname{d}\mathrm{S}^{\,\sigma}_{n}
 \left(X_1\right)(\mathrm{v})\;+\;\sigma\,
  \sum_{k=1}^{n-1}\,
  \binom{n}{k}\;{\mathrm{S}}^{\,\sigma}_{k}(\mathrm{v})\;
  {\mathrm{S}}^{\,\vartheta}_{n-k}(\mathrm{v})\;.
  \end{equation}
\end{Rem}

\section{On the rank of the infinite Schwarzian matrix}\label{SECT-RANK}
We include some comments on the rank of the infinite
Hankel matrix of $g_R$-Schwarzians and its relation to the adapted complex structure. A companion infinite Hankel matrix is introduced in Definition \ref{DEF-RANK-C} and applied to  closed geodesics.
\begin{Def}\label{DEF-RANK}
Put, for $\mathrm{v}\in TM$,
\begin{equation}\label{DEF-RANK-TM}
\mathrm{R}^{\,\sigma, \,g_R}(\mathrm{v})\,:=\,
\operatorname{Rank}\left[\frac{{\mathrm{S}}^{\,\sigma, \,g_R}_{i+j-1}
(\mathrm{v})}{(i+j-1)!}\right]_{\,i,\,j\,=\,1}^{\infty}\,,
\end{equation}
and consider
 the function $\mathrm{R}^{\,\sigma,\,g_R}\colon \,TM \,\to\,\mathbb{N}\,\cup\,\{\infty\}$ defined by letting $\mathrm{v}\in TM$ vary in (\ref{DEF-RANK-TM})\,.
\end{Def}
 We will show by means of  Proposition \ref{PROP-RANK-RAY}
 and Proposition \ref{PROP-RANK-GF} that this function is constant along the leaves
of the foliation of $TM\setminus{M}$,
$$TM\setminus{M}\,\;\;\;
=\,\bigcup_{\gamma \textrm{ geodesic of M}}
\mathbf{P}_\gamma \left(\mathrm{C}\setminus\mathbb{R}\right)\,,$$
determined by the images of the maps (\ref{MA-LEAF-MAP}).

Note that for all
$\mathrm{v}\in M \subset TM$ we always have   $\mathrm{R}^{\,\sigma,\,g_R}(\mathrm{v})\,=\,1$,
by Remark \ref{onM}\,.

\begin{Prop}\label{PROP-RANK-RAY}
The function $\mathrm{R}^{\,\sigma,\,g_R}$
 is constant along
the lines in the fibers of $TM$, i.e., for all $\mathrm{v}\in TM$ and for all $t\in \mathbb{R}$
\begin{equation}\label{RANK-RAY}
\mathrm{R}^{\,\sigma, \,g_R}
(e^t\,\mathrm{v})\,=\,\mathrm{R}^{\,\sigma, \,g_R}
(\mathrm{v})\,.
\end{equation}
\end{Prop}

\begin{proof}
Put, for all
 integer $n\geq 1$,
 $$\mathrm{C}_n (\mathrm{v})\,=\, \frac{{\mathrm{S}}^{\,\sigma, \,g_R}_{n}(\mathrm{v})}{n!}\,.$$

The result claimed is a consequence of  the homogeneity property
\begin{equation}\label{C-SCALE}
{\mathrm{C}}_n(\mathrm{\lambda\,v})\,=\,\lambda^{n-1}\,
{\mathrm{C}}_n(\mathrm{v})
\end{equation}
proved in Proposition \ref{HAT-SIGMA-HOMOGENEOUS}
and the special symmetry
of an infinite Hankel matrix .

In fact, by Theorem 7 in \cite{GANTMACHER2}, the infinite Hankel matrix
\begin{equation}\label{IHM}
\left[\mathrm{C}_{i+j-1}(\mathrm{v})\right]_{i,j\,=\,1}^{\infty}
\end{equation}
 has finite rank $N$ if and only if $N$ is the smallest natural number such that the $(1+N)^{\textrm{th}}$ infinite column of (\ref{IHM}) is a linear combination of the $1^{\textrm{st}}\,,\cdots N^{\textrm{th}}$  infinite columns of this infinite matrix, i.e.,  if and only   $N$ is the smallest natural number so that there are numbers  $a_1(\mathrm{v}), \ldots a_N(\mathrm{v})$ $\in \mathbb{C}$ such that
\begin{equation}\label{R0-H}
\left[\begin{array}{cccc}
\mathrm{C}_1(\mathrm{v})&C_2(\mathrm{v})&\cdots&\mathrm{C}_N(\mathrm{v})\\
\mathrm{C}_2(\mathrm{v})&\mathrm{C}_3(\mathrm{v})&\cdots&\mathrm{C}_{N+1}(\mathrm{v})\\
\vdots & \vdots & \;&\vdots
\end{array}\right]
\left[\begin{array}{c}a_1(\mathrm{v})\\\vdots\\a_N(\mathrm{v})
\end{array}\right]\,
\,=\,
\left[\begin{array}{c} {\mathrm{C}}_{N+1}(\mathrm{v})\\
                        {\mathrm{C}}_{N+2}(\mathrm{v})\\
                        \vdots
                        \end{array}\right]\,,
\end{equation}
which is equivalent to
the validity of the sequence of equalities
\begin{equation}\label{EQ-RANK-SEQ-EQS}
\mathrm{C}_{N+n+1}(\mathrm{v})=\sum_{i\,=\,1}^N a_i(\mathrm{v})\, \mathrm{C}_{i+n}(\mathrm{v})
\end{equation}
for all integer $n\geq 0$. But from (\ref{C-SCALE}), for any $\lambda\neq 0$, the equalities in the sequence (\ref{EQ-RANK-SEQ-EQS})  hold
for the given $a_1(\mathrm{v}),\,\cdots,\,a_N(\mathrm{v})$ if and only if the equalities in the sequence
$$\mathrm{C}_{N+n+1}(\mathrm{\lambda\,v})=\sum_{i\,=\,1}^N a_i(\lambda\mathrm{v})\, \mathrm{C}_{i+n}(\mathrm{v}),$$
where $n\geq 0$,
hold for the coefficients  $a_i(\lambda\mathrm{v})$  given for $i=1,\, \cdots,\, N$ by $$a_i(\lambda\mathrm{v})\,:=\,\lambda^{N-i+1}\,a_i(\mathrm{v})\,. $$
\end{proof}

We will use the following facts taken from \cite{GANTMACHER2}.
\begin{Prop}\label{PROP-RAT-RANK} Let be given an infinite
Hankel matrix
\begin{equation}\label{IHM-0}
\left[\mathrm{C}_{i+j-1}\right]_{i,j\,=\,1}^{\infty}\,.
\end{equation}
\begin{itemize}
\item[(i)] (Corollary of Theorem 7 in \cite{GANTMACHER2})
If the matrix (\ref{IHM-0}) has finite rank $N$, then
\begin{equation}
\det \left[\mathrm{C}_{i+j-1}\right]_{i,j\,=\,1}^{N}\neq 0\,.\end{equation}
\item[(ii)](Theorem 8 in \cite{GANTMACHER2})
 The matrix (\ref{IHM-0}) has finite rank $N$ if and only if
 $\sum_{n\,=\,1}^\infty \mathrm{C}_n\,u^{-n}$
 is
 rational in $u$ of degree $N$.
 \end{itemize}
 \end{Prop}

\begin{Prop}\label{PROP-RANK-GF}
 The function $\mathrm{R}^{\,\sigma,\,g_R}$
 is constant along the trajectories of the geodesic flow on $TM$, that is, for all $x\in \mathbb{R}$,
\begin{equation}\label{RANK-GF}
\mathrm{R}^{\,\sigma, \,g_R}
(\phi_x\mathrm{v})\,=\,\mathrm{R}^{\,\sigma, \,g_R}
(\mathrm{v})\,.
\end{equation}
\end{Prop}
\begin{proof}
We already noted  that $\mathrm{R}^{\,\sigma, \,g_R}
(\mathrm{v})=1$ for any $\mathrm{v}\in M\subset TM$. Moreover, in light of Proposition
\ref{PROP-RANK-RAY}, it suffices to consider the flow on the unit tangent bundle $UM\subset TM$.
So  let $\mathrm{v}\,\in\, TM$ with $\|\mathrm{v}\|\,=\,1$
and take the unit-speed geodesic $\gamma\colon \mathrm{R}\mapsto{M}$ given by  $\gamma(t)\,=\,\pi\left(\phi_t\mathrm{v}\right)$\,.

 With notation as in the proofs of Theorem \ref{TH-EXTEND-FINITE-R} and Theorem \ref{TH-EXTEND-TUBE-FINITE},
for any $x\,\in\,\mathrm{R}$,
take the  quotient
$h_{{x}}$ of the two fundamental
 solutions at $x$ of equation in (\ref{DE-GAMMA}),
 and consider the function of $s$
 \begin{equation}\label{EQ-hx}
 \,\left(h_x\circ T_x\circ{g_R}\right)(s)\,=\,h_{{x}}\left(x+g_R (s)\right)
  \end{equation}
  which is defined for $s\in (\,-\epsilon,\,\epsilon)\subset \mathrm{R}$ for some $\epsilon >0$.
Since for all integer $n\geq 0$,
$$(h_x\circ{T_x}\circ{g_R})^{(n)}(0)\,
=\,\mathrm{S}^{\,\sigma,\,g_R}_n(\phi_x\mathrm{v}),
$$
then
$$
\mathrm{R}^{\,\sigma, g_R}(\phi_x\mathrm{v})=\operatorname{Rank}
\left[
\frac{{(h_x\circ{T_x}\circ{g_R})}^{(i+j-1)}(0) }{(i+j-1)!}
\right]_{\,i,\,j\,=\,1}^\infty.
   $$

Now, by Proposition
\ref{PROP-RAT-RANK} and the transformation of degree 1, $u\mapsto t^{-1}$, it follows that
the rank of the infinite matrix in the right-hanside is $N<\infty$ if and only if
$$(h_x\circ{T_x}\circ{g_R})(s)$$ is a rational function in $s$ of degree $N$. Such value of $N$ is independent of $x$ since for any $x_1$ and $x_2$ in $\mathbb{R}$ the quotients
$h_{x_1}$ and $h_{x_2}$
are related by a (real) Mo\"{e}bius transformation,
which is rational of degree 1.
\end{proof}
As a consequence
the following aspect of Theorem 7 in \cite{GANTMACHER2} appears in the context of the Schwarzian functions.
 \begin{Cor}\label{PROP-F-R-SR}
 With notation as in (\ref{GF}), for  $\mathrm{v} \in TM $  it holds
 \begin{equation}\label{F-RANK}
 \mathrm{R}^{\,\sigma, \,g_R}(\mathrm{\mathrm{v}})=N<\infty,
  \end{equation}
  if $N$ is the smallest natural number so that for each  $n\geq 1$ the function of $t$
$${\mathrm{S}}^{\,\sigma,\,g_R}_{2N+n}(\phi_t\mathrm{v})$$ can be written, for all $t\in \mathbb{R}$,   in terms of the set of
functions  of $t$ $$\left\{ {\mathrm{S}}^{\,\sigma, \,g_R}_1(\phi_t\mathrm{v}),\,\ldots,\,{\mathrm{S}}^{\,\sigma, \,g_R}_{2\,N}(\phi_t\mathrm{v})\right\}$$ as
\begin{equation}\label{SCHW-LIN-DEP}
{\mathrm{S}}^{\,\sigma, \,g_R}_{2N+n}(\phi_t\mathrm{v})\,=\,
   \frac{
        p_{n,\, N}\left( {\mathrm{S}}^{\,\sigma, \,g_R}_{1}(\phi_t\mathrm{v}),\cdots, {\mathrm{S}}^{\,\sigma, \,g_R}_{2N}(\phi_t\mathrm{v})\right)}
        {\left({\mathrm{D}}^{\,\sigma, \,g_R}_{N}(\phi_t\mathrm{v})\right)^{n+1}},
 \end{equation}
 where $p_{\,n, N}\,=\,p_{n,\, N}(X_1,\,\cdots,\, X_{2N})$ are universal polynomials with coefficients dependent on $t$. (The statement includes ${\mathrm{S}}^{\,\sigma, \,g_R}_1(\phi_t\mathrm{v})\equiv1$, for convenience)\,.
\end{Cor}

\begin{proof}
Set, for all integer $n\,\geq\,1$ and all $t\in \mathbb{R}$, $$\mathrm{C}_n(t)\,=\,
\frac{{\mathrm{S}}^{\,\sigma,g_R}_{n}(\phi_t\mathrm{v})}{n!}
\,,$$
and $$\mathrm{D}_N(t)\,=\,{\mathrm{D}}^{\,\sigma, \,g_R}_{N}(\phi_t\mathrm{v})\,.$$

The fact, proven in Proposition \ref{PROP-RANK-GF}, that the rank of the infinite matrix
$\displaystyle \left[\mathrm{C}_{i+j-1}(t)
\right]_{\,i,\,j\,=\,1}^{\infty}$
is independent of $t$,
together with Proposition \ref{PROP-RAT-RANK} and the sequence of equalities (\ref{EQ-RANK-SEQ-EQS}),  means that the noted rank is  $N\,<\,\infty$ if and only if  such is the smallest natural number with the property that
there are functions $a_1(t),\,\cdots, \,a_N(t)$ defined for all $t\in \mathbb{R}$ so that
for every integer $n\,\geq\,0$,
\begin{equation}\label{R2-H}
\left[{\mathrm{C}}_{i+j+\,n-1}(t)
\right]_{\,i,\,j\,=\,1}^{N}\,
\left[\begin{array}{c}a_1(t)\\\vdots\\a_N(t)\end{array}\right]\,
\,=\,
\left[\begin{array}{c}{\mathrm{C}}_{N+1+\,n}(t)\\\vdots\\
{\mathrm{C}}_{2N+\,n}(t)\end{array}\right]\,.
\end{equation}
From the equality for  $n\,=\,0$, namely
\begin{equation}\label{R1-H-0}
\left[\mathrm{C}_{i+j-1}(t)
\right]_{\,i,\,j\,=\,1}^{N}\,\left[\begin{array}{c}a_1(t)\\\vdots\\a_N(t)\end{array}\right]\,
\,=\,
\left[\begin{array}{c}\mathrm{C}_{N+1}(t)\\\vdots\\
{\mathrm{C}}_{2N}(t)\end{array}\right]\,,
\end{equation}
and induction,
the equalities above are equivalent to
\begin{eqnarray}\label{MATRIX-H-0}
&
\left[\begin{array}{c}{\mathrm{C}}_{N+1+\,n}(t)\\\vdots\\
{\mathrm{C}}_{2\,N+\,n}(t)\end{array}\right]\,=\,
\left[{\mathrm{C}}_{i+j-1}(t)
\right]_{\,i,\,j\,=\,1}^{N}\,
\underset{(*)}{\underbrace{\left[\begin{array}{ccccc}0&0&\ldots&0&a_1(t)
\\1&0&\ldots&0&a_2(t)\\
0&1&\ddots&\vdots&\vdots\\
\vdots&\ddots&\ddots&0&\vdots
\\0&\ldots&0&1&a_N(t)\end{array}\right]^n\,
\left[\begin{array}{c}a_1(t)\\\vdots\\a_N(t)\end{array}\right]}}
\,,
&
\end{eqnarray}
for all integer $n\geq 0$,
 exhibiting the functions ${\mathrm{C}}_{2N+n}(t)$ for every integer $n\,\geq\,0$ as rational in ${\mathrm{C}}_2(t),\,\cdots,\,{\mathrm{C}}_{2N}(t)$. Now, since $\mathrm{D}_N(t)\neq 0$ for all $t\in \mathbb{R}$ by Proposition \ref{PROP-RAT-RANK} and Proposition \ref{PROP-RANK-GF}, that the particular form of such dependence is as
 claimed follows by writing (*) in (\ref{MATRIX-H-0})
 as
 \begin{eqnarray}\notag
  \left(\mathrm{D}_N(t)\right)^{-n-1}\,
 \left[\begin{array}{ccccc}0&0&\ldots&0&p_1
\\{\mathrm{D}}_N(t) &0&\ldots&0&p_2\\
0&{\mathrm{D}}_N(t) &\ddots&\vdots&\vdots\\
\vdots&\ddots&\ddots&0&\vdots
\\0&\ldots&0&{\mathrm{D}}_N(t) &p_N\end{array}\right]^n\,
\left[\begin{array}{c}p_1
\\\vdots\\p_N\end{array}\right]\,,
\end{eqnarray}
where, from Cramer's rule in (\ref{R1-H-0}), for $i\,=\,1,\,\ldots,\,N$,
\begin{equation}
{\mathrm{p}}_{i,N}({\mathrm{C}}_1(t)
\,\cdots,\,{\mathrm{C}}_{2N}(t))\,=\,a_i(t)\, {\mathrm{D}}_N(t)\,,
\end{equation}
 with ${\mathrm{p}}_{i,N}$ polynomials in the  ${\mathrm{C}}_k(t)$ for $k$ as indicated.
\end{proof}
We introduce the following generating
functions.
\begin{Def}\label{DEF-GEN-FUNC}
Let $\mathrm{G}^{\,\sigma, \, g_R}\colon TM\times\mathbb{R}\times \mathbb{R}\to \mathbb{R}$
be defined by
\begin{equation}\label{FLA-GEN-FUNC-GR-SCH}
\mathrm{G}^{\,\sigma, \, g_R}(\mathrm{v},\, x,\,s)\,=\,
\sum_{n\,=\,1}^\infty \mathrm{S}^{\,\sigma,\,g_R}_n(\phi_x\mathrm{v})\,
  \frac{s^n}{n!}\,.
  \end{equation}
  Here $g_R$ is given as before, according to Definition
  \ref{HORS}. In particular
 for $R\,=\,\infty$ the above reduces to
\begin{equation}\label{FLA-GEN-FUNC-SCH}
\,\mathrm{G}^{\,\sigma}(\mathrm{v},\, x,\,s)\,:=\,
\sum_{n\,=\,1}^\infty \mathrm{S}^{\,\sigma}_n(\phi_x\mathrm{v})\,
  \frac{s^n}{n!}\,.
\end{equation}
\end{Def} Then
Proposition \ref{PROP-RAT-RANK} and Proposition
  \ref{PROP-RANK-GF} imply the following.
\begin{Cor}\label{COR-F-R-EXP}
Let $R\in (0,\,\infty]$ and $\mathrm{v}\in UM$.  Then
 $\mathrm{R}^{\,\sigma,\,g_R}(\mathrm{v})
 \,=\,N\,<\,\infty$ if and only if for all $x\in \mathbb{R}$
 the generating function
$\mathrm{G}^{\,\sigma,\,g_R}(\mathrm{v},\,x,\,s)$ is  rational in $s$ of degree $N$\,.
\end{Cor}

\begin{Cor}\label{COR-F-R-EXP-TRM}
Let the Riemannian metric $g$ on $M$ be real analytic, and to be specific,
consider
$\displaystyle g_R(s)\,=\,-\frac{R}{\pi}\ln\left(1-\frac{\pi\,s}{R}\right)$.
Given $R\in (0,\,\infty]$ and $\mathrm{v}\in UM$, $\mathrm{R}^{\,\sigma,\,g_R}(\mathrm{v})
 \,=\,N\,<\,\infty$ if and only if
\begin{equation}\label{GEN-FUNC-RU}
 \mathrm{G}^{\,\sigma}(\mathrm{v},\,x,\,s)\,
 \,=\,
\mathcal{R}(\mathrm{v},\,x,\, u(s)),
\end{equation}
with $\mathcal{R}$ is rational in $u$ of degree $N$ and
\begin{equation}\label{EQ-RAT-U(S)}
u(s)\,=\,\frac{R}{\pi}\,(1-e^{ \frac{\pi}{R}\,s})\,,
\end{equation}
interpreted as $u\,=\,s$ when $ R\,=\,\infty$\,.

Moreover if  the adapted complex structure is defined on the entire $TM$, for every $x\in \mathbb{R}$ the poles and zeros of $\mathcal{R}(\mathrm{v},\,x,\, u)$, as a rational function of $u$, are all simple and real, the poles with negative residues.
  \end{Cor}
  \begin{proof}
  With notation as in
  Proposition
  \ref{PROP-RANK-GF},
  taking the unit-speed geodesic $\gamma\colon \mathrm{R}\mapsto{M}$ with $\dot{\gamma}(0)\,=\,\mathrm{v}$, for each $x\in \mathbb{R}$ we have the equality of power series
  \begin{eqnarray} \mathrm{G}^{\,\sigma, \, g_R}(\mathrm{v},\, x,\,t)\,&=&\,
\sum_{n\,=\,1}^\infty \mathrm{S}^{\,\sigma,\,g_R}_n(\phi_x\mathrm{v})\,
  \frac{t^n}{n!}\\\notag
  &=&
  \,\sum_{n\,=\,1}^{\infty}\left(h_x\circ{T_x}
  \circ{g_R}\right)^{(n)}(0)\frac{t^n}{n!}\,.
  \end{eqnarray}
  By the  real analyticity of $g$ and
   given the choice of
    $g_R$ the power series converges near $t\,=\,0$, thus
    \begin{eqnarray}\label{EQ-HX-COMP-LN}
  \mathrm{G}^{\,\sigma, \, g_R}(\mathrm{v},\, x,\,t)\,=\,h_{x}\left(x-\frac{R}{\pi}
  \ln(1-\frac{\pi\,t}{R})\right)
  \,.
  \end{eqnarray}
  By Corollary \ref{COR-F-R-EXP} the left hand-side is rational in $t$ of degree $N$, say
  \begin{equation}\label{EQ-HX-RAT}
     \mathrm{G}^{\,\sigma, \, g_R}(\mathrm{v},\, x,\,t)\,
     =\,\mathcal{R}^{\,\sigma, \, g_R}(\mathrm{v}(\mathrm{v},\,x,\, t),
     \end{equation}
   if and only if
   $\mathrm{R}^{\,\sigma,\,g_R}(\phi_x\mathrm{v})
 \,=\,N\,<\,\infty $,
    this rank equal to $\mathrm{R}^{\,\sigma,\,g_R}(\mathrm{v})$ by Proposition \ref{PROP-RANK-GF}.

   So, the first part of the Corollary follows from
   (\ref{EQ-HX-COMP-LN}), (\ref{EQ-HX-RAT}) and real analyticity by observing that
  \begin{eqnarray}\label{EQ-HX-COMP-LN}
  \mathcal{R}^{\,\sigma, \, g_R}(\mathrm{v}(x,\,\mathrm{v},\, \frac{R}{\pi}(1-e^{-\frac{\pi}{R}s}))\,&=&\,h_{x}(x+s)
  \\\notag
  &=&\,\sum_{n\,=\,1}^\infty h^{(n)}_{x}(x) \frac{s^n}{n!}
  \\\notag
  &=&\sum_{n\,=\,1}^\infty \mathrm{S}^{\,\sigma}_n(\phi_x\mathrm{v})\,
  \frac{s^n}{n!}\\\notag
  &=&\,\mathrm{G}^{\,\sigma}(\mathrm{v},\, x,\,s)\,.
  \end{eqnarray}

  In addition, if the adapted structure is defined on the whole $TM$, for each $x\in\mathbb{R}$  the function $h_x(x+t)$ on the right-hand side of (\ref{EQ-HX-COMP-LN}) is a Pick function in $t$. Thus, in light of the fact that  $g_R(t)\in \mathbb{R}$ precisely when  $t\in \mathbb{R}$\, the conditions imposed
  on $\mathcal{R}^{\,\sigma, \, g_R}(\mathrm{v},\,x,\,  u)$
  as a rational function in  $u$ follow\,.
  \end{proof}

\section{Schwarzians with purely imaginary  radius and closed geodesics}\label{SECT-RANK-C}
To complete our remarks on higher order Schwarzians
 we relate formally the periodicity
of the geodesic flow along a trajectory with the
 the finiteness of the rank of an infinite Hankel matrix now constructed with complex valued higher order Schwarzians defined by
 considering complex values for the radius $R$ as indicated in Remark \ref{HORS-C} for the Schwarzians in $\mathbb{R}$.
 \begin{Def} For any $0\neq \lambda \in \mathbb{R}$  and integer $n\geq 0$ put
 $$\fbox{$\displaystyle \mathrm{S}_n^{\,\sigma, \,g_{\sqrt{-1}\lambda}}
 :=\mathrm{S}_n^{\,\sigma, \,g_{R}}\big{|}_{R=\sqrt{-1}\lambda}$}\,.
 $$
 \end{Def}
A companion to Definition \ref{DEF-RANK} is the following.
\begin{Def}\label{DEF-RANK-C}
Given $\mathrm{v}\in TM$ put
\begin{equation}\label{DEF-RANK-TM}
\mathrm{R}^{\,\sigma, \,g_{\sqrt{-1}\lambda}}(\mathrm{v})\,:=\,
\operatorname{Rank}
\left[
\frac{{\mathrm{S}}^{\,\sigma, \,g_{\sqrt{-1}\lambda}}_{i+j-1}
(\mathrm{v})}{(i+j-1)!}\right]_{\,i,\,j\,=\,1}^{\infty}\,,
\end{equation}
and consider
 the function $\mathrm{R}_{\mathbb{C}}^{\,\sigma,\,g_R}\colon \,TM \,\to\,\mathbb{N}\,\cup\,\{\infty\}$ defined by letting $\mathrm{v}\in TM$ vary in (\ref{DEF-RANK-TM})\,.
\end{Def}
\begin{Prop}\label{PROP-PJF}
(Periodicity of Jacobi fields and the adapted complex structure). Let $M$ be a two-dimensional real analytic manifold with a real analytic complete Riemannian metric $g$ with the adapted complex structure defined on the entire $TM$.
Let $\gamma\colon\mathbb{R}\to M$  be any unit speed closed geodesic of $M$.

Take $\displaystyle g_R\,=\,\frac{R}{\pi}
 \ln{(1-\frac{\pi}{R}z)}$.

Then if the
 normal Jacobi fields  \footnote{Normal means point-wise orthogonal to $\dot{\gamma}$} along
 $\gamma$ are  periodic with period $T$ then there is an integer  $N\geq 1$ so that  for all $t\in\mathbb{R}$,
\begin{equation}\label{RANK-CLOSED-TM}
\mathrm{R}_{\mathbb{C}}^{\,\sigma, \,g_{\sqrt{-1}\frac{T}{2}}}(\mathrm{\dot{\gamma}(t)})\,=\,N\,.
\end{equation}
\end{Prop}
\begin{proof}
Let $\gamma$ be a unit speed geodesic parameterized of $M$.
Let
$x\in \mathbb{R}$ be fixed, and set $h_{x}(t)\,=\,f_{x,\,2}(t)/f_{x,\,1}(t)$, with $f_{x,\,1}$ and $f_{x,\,2}$ the pair of fundamental solutions
of $f^{\prime\prime}(t)+\sigma(\gamma(t))f(t)\,=\,0$
at $x$ in the sense of (\ref{FP}), each function identified with a normal Jacobi field as in the proof of Theorem \ref{TH-EXTEND-TUBE-FINITE}.

If all the  Jacobi fields normal to $\gamma$ are periodic with period $T$ then
\begin{equation}\label{EQ-hx-PERIOD}
h_{x}(t)\,=\,h_x(t+T)
\end{equation}
 for all $t\in \mathbb{R}\setminus{f^{-1}_{x,\,1}(0)}$.

 Since the adapted complex structure is defined on the entire $TM$, $h_x(z)$ is a Pick function in $z$,
hence representable,
 according to Fatou's formula, for $z\in\mathbb{C}\setminus \mathbb{R}$, by
 \begin{equation}\label{FATOU-F}
 h_x(z)\,=\,\alpha\,z\,+ \,\beta\,+\int_{-\infty}^\infty \,\left(\, \frac{1}{t-z}-\frac{t}{1+t^2}\right) \,d\mu(t),
 \end{equation}
  with constants $\alpha=\alpha_x,\, \beta=\beta_x \in \mathbb{R}$,  $\alpha_x\geq 0$,
  and with
 $d\mu(t)=d\mu_x(t)$ a non-negative Borel measure with $$\int_{-\infty}^\infty\frac{d\mu(t)}{1+t^2}\,<\,\infty\,.$$

 The  measure $d\mu(t)$ is determined by
  a weak limit,
  $$\int_{a}^bd\mu(t)\,=\,\lim_{y\mapsto 0+} \frac{1}{\pi}\,\int_{a}^b\,\,\Im h_{x}(t\,+\,\sqrt{-1}y)\,dt\,.$$
Thus, in our particular case, $h_x(z)$ is
  meromorphic is $z$, with  poles that are all real and simple, namely the zeros of $f_{x,\,1}(t)$, with negative residues\,.

  In addition, from (\ref{EQ-hx-PERIOD}), we have
 $\alpha=0$ in (\ref{FATOU-F}), and moreover, for some  integer
 \begin{equation}\label{EQ-N}
 N\,=\,N(x, T)\,\geq 1,
 \end{equation} the polar set, $
f_{x,1}^{-1}(0)$,  is of the form
$$
\{a_1+k\,T,\,\cdots, \,a_N+k\,T\,, \,k\in \mathbb{Z}\}
$$
where all the $a_i=a_i(x)$ are in $\mathbb{R}$ and satisfy
\begin{equation}\label{x-ai}
x-\frac{T}{2}<a_1<\cdots <a_N\leq x+\frac{T}{2}\,.
\end{equation}

The support of $d\mu(t)$ consists of a mass at each point in $\{a_i+\,k\,T$, $k\in \mathbb{Z}$, $i=1,\cdots,N\}$ with weight
$$r_i\,=\,r_i(x)\,=\,\lim_{t\mapsto a_i}(t-a_i)\,h_{x}(t)\;>\,0\,.$$
 Thus, using Herglotz formula,
 \begin{eqnarray}\label{P-H-TAN}
 h_{x}(z)\,&=&\,\sum_{i\,=\,1}^N\,\frac{r_i}{z-a_i}+
 r_i\,\sum_{k\,=\,1}^{\infty}\left(\frac{1}{z-a_i-k\,T}+
 \frac{1}{z-a_i+k\;T}\right)
 \\\notag
 &=&\,\frac{\pi}{T}\,
 \sum_{i\,=\,1}^N \,\frac{r_i}{
 \tan{\frac{\pi}{T}(z-a_i)}}\,.
 \end{eqnarray}

Now, observe  that
\begin{eqnarray}\label{P-H-COMP} F_{x,\,\lambda}(z):=
 h_{x}\left(x\,-\,\frac{\sqrt{-1}\lambda}{\pi}
 \ln{\left(1-\frac{\pi}{\sqrt{-1}\lambda}z\right)}
  \right)
  \end{eqnarray}
  is  rational in $z$ if we take $$\lambda\,=\,n\frac{T}{2}\,,$$ for arbitrary integer
  $n\neq 0$.
To check this explicitly, for $n=1$, consider the contributions to
(\ref{P-H-COMP}) of each of the terms in (\ref{P-H-TAN}). Use the identity $$\tan{(a-b)}\,=\,
\frac{\tan{a}\,-\,
\tan{b}}{ 1+\tan{a}\tan{b}}$$  with $a=x-a_i$,
and the identity
\begin{eqnarray}\label{TAN-COMP}
\tan{\left(\frac{\sqrt{-1}}{2}\,
\ln{(1-\frac{2\,\pi}{\sqrt{-1}T}z)}\right)}
&=&\,\frac{\sqrt{-1}z}{z-\sqrt{-1}\,\frac{T}{\pi}}\,,
\end{eqnarray}
to conclude that, for each $1\leq i\leq N$,
 \begin{equation}\label{P-H-COMP-1}
 \frac{1}{\tan{\left(\frac{\pi}{T}
 (x-a_i)\,-\,\frac{\sqrt{-1}}{2\,}
 \ln{(1-\frac{2\,\pi}{\sqrt{-1}T}z)}\right)\,}}
 \end{equation}
 simplifies to the Mo\"{e}bius transformation in $z$
 \begin{equation}\label{EQ-RED-MOEB}
\,\frac{\sqrt{-1}\,\,z +\,\frac{\,T}{\pi}A_i\,}
 {\,z-\,\sqrt{-1}\,(1-A_i)\,\frac{\,T}{\pi}\,},
\end{equation}
with
\begin{equation}\label{AI}
 A_i\,= \,A_i(x, T)\,=\,e^{\frac{\pi}{T}(x-a_i)\sqrt{-1}}\,\cos{ \frac{\pi}{T}(x-a_i)}\,.
 \end{equation}
Hence, $F_{x, \frac{T}{2}}(z)$ is rational in $z$, of degree depending on $T$, but not on $x$, since $F_{x_1, \frac{T}{2}}(z)$ is related to  $F_{x_2, \frac{T}{2}}(z)$ by a Mo\"{e}bius trasformation.
Moreover, by inspection of the denominators of (\ref{EQ-RED-MOEB}),
since by (\ref{x-ai}) we have that for all $1\leq i,j\leq N$
 $$A_i\neq A_j, \quad i\neq j\,,$$ it follows
 that the degree of the rational function $F_{x, \frac{T}{2}}(z)$ in $z$ is the number $N$, the possible
  non-constancy of which, indicated in (\ref{EQ-N}),  can be now fully described by $N=N(T)$\,.

So, we have
 \begin{eqnarray}\notag
 \fbox{$ \displaystyle
F_{x, \frac{T}{2}}(z)=
 \displaystyle \frac{\pi}{T}
 \sum_{i\,=\,1}^{N(T)}\, r_i(x)\,
 \,\frac{\sqrt{-1}\,\,z +\,\frac{T}{\pi}A_i(x,T)\,}
 {\,z-\,\sqrt{-1}\,(1-A_i(x,T))\,\frac{\,T}{\pi}\,}$}\,.
 \end{eqnarray}

Now,
the function $F_{x,\lambda}(z)$
 given by (\ref{P-H-COMP})
 is holomorphic in a neighborhood of $z=0\in \mathbb{C}$,
 and from the definitions we have
 $$F^{(n)}_{x,\lambda}(0)\,=\,
 \mathrm{S}_n^{\,\sigma, \,g_{\sqrt{-1}\lambda}}(\dot{\gamma}(x))\,,
 $$
 and for $\lambda=\frac{T}{2}$,
 \begin{equation}\notag
\mathrm{R}^{\,\sigma,\,g_{\sqrt{-1}\frac{T}{2}}}
({\dot{\gamma}(x)})=\,
\operatorname{Rank}
\left[
\frac{{\mathrm{S}}^{\,\sigma, \,g_{\sqrt{-1}\frac{T}{2}}}_{i+j-1}
(\dot{\gamma}(x))}{(i+j-1)!}
\right]_{\,i,\,j\,=\,1}^{\infty}\,
=\;\,
\operatorname{Rank}
\left[
\frac{F^{(i+j-1)}_{x,\frac{T}{2}}(0)}{(i+j-1)!}
\right]_{\,i,\,j\,=\,1}^{\infty}\,.
\end{equation}
 Thus, by
Proposition \ref{PROP-RAT-RANK} and Proposition
  \ref{PROP-RANK-GF}
  used exactly as in the previous section,
$$\mathrm{R}^{\,\sigma,\,g_{\sqrt{-1}\frac{L}{2}}}
({\dot{\gamma}(t)})
=\mathrm{R}^{\,\sigma,\,g_{\sqrt{-1}\frac{T}{2}}}
({\dot{\gamma}(x)})= N$$
for all $t\,\in \mathbb{R}$.
\end{proof}

\begin{Cor}\label{COR-RANK-BD-CLOSED}
Let $M$ be a two-dimensional closed real analytic Riemannian manifold with all geodesics closed of length $L$ and with  adapted complex structure
on the entire $TM$.  Then for any unit speed geodesic $\gamma$ and
for all $
\mathrm{t}\in \mathbb{R}$
\begin{equation}\label{EQ-RANK-BD}
\operatorname{Rank}
\left[
\frac{{\mathrm{S}}^{\,\sigma, \,g_{\sqrt{-1}\frac{L}{2}}}_{i+j-1}
(\dot{\gamma}(t))}{(i+j-1)!}
\right]_{\,i,\,j\,=\,1}^{\infty}\,\leq\,\frac{\,L}{\pi}\max_{x\in M}\sqrt{\sigma}\,.
\end{equation}
\end{Cor}
\begin{proof}
By the standard Sturm comparison Theorem, the distance $d$ between zeros of a normal Jacobi field along any unit speed geodesic $\gamma$
  satisfies $$d\geq\frac{\pi}{\sup_{t\in \mathbb{R}}\sqrt{\sigma(\gamma(t))}}\,.$$

In the present situation, the normal  Jacobi fields along any unit speed geodesic are all periodic, and so will be their quotients,  with period $L$. Thus, with the notation as in the proof of Proposition \ref{PROP-PJF} we must have $N\,\leq L/d$\,, and (\ref{EQ-RANK-BD}) follows.
\end{proof}

\begin{Rem}
Thus in a situation as in the Corollary \ref{COR-RANK-BD-CLOSED} we have the following canonical maps. Let $$N=\lfloor \frac{\,L}{\pi}\max_{x\in M}\sqrt{\sigma}\rfloor\,.$$
For any $k\geq 1$ the map
$ \displaystyle TM\to \mathbb{C}^{2(N+k)-1}$
$$\mathrm{v}\in TM\to
\left(Z_1(\mathrm{v}),\,\cdots,\,Z_{2(N+k)-1}(\mathrm{v})\right) \in
\mathbb{C}^{2(N+k)-1}$$
with
\begin{equation}\label{FLA-ZI}
Z_i(\mathrm{v}):=\frac{1}{k!}\mathrm{S}^{\,\sigma,\,g_{\sqrt{-1}
 \frac{L}{2}}}_{k}(\mathrm{v})\,,\quad i\,=\,1,\,\cdots,\,2(N+k)-1\,,
 \end{equation}
sends $TM$ to $$\{(Z_1,\, \cdots,\, Z_{2(N+k)-1})\in \mathbb{C}^{2(N+k)-1}|
\det \big{[}Z_{i+j-1}\big{]}_{\,i,\,j\,=\,1}^{N+k}=0\}\,.$$
\end{Rem}
The map is determined by its restriction to the unit tangent bundle $UM$ since by the homogeneity property of the Schwarzians as defined,
\begin{equation}\label{RESC}
\displaystyle Z_i(a\mathrm{v})=a^{i-1}Z_i(\mathrm{v})\,.
\end{equation}
\begin{Rem}
In case of $M$ is the  $2$-sphere with constant curvature $1$
the period $T$ in (\ref{EQ-hx-PERIOD}) can be taken to be
$\pi$ rather than $2\pi$, obviously for all unit-speed geodesics.

 It follows that (\ref{EQ-RANK-BD}) is also valid
  if we use instead of the value $L=2\pi$ we use  $L=\pi$.
  This of course is consistent with that (\ref{EQ-RANK-BD}) applies to projective space as well. With this value of $L$  the rank of the infinite matrix of the indicated Schwarzians is 1, corresponding to the fact that the entries, for $\|\mathrm{v}\|=1$, are the McLaurin coefficients of the right-hand side of (\ref{TAN-COMP}).

 In this simple case the functions
 $(\ref{FLA-ZI})$ are constant along the unit tangent bundle and satisfy (\ref{RESC}).
 Thus, since $Z_1\equiv 1$,  the image of $TM$ by the map
 determined by the $Z_i$ above, for any integer $k>1$, is just one point in a weighted projective space.
\end{Rem}

\section{Some computations}\label{SEC-Computations}
\subsection{A few Schwarzians ${\mathcal{S}}^{\,\vartheta}_n$ in $\mathbb{R}$ in terms of curvature $\vartheta$}\
\vskip.2cm

  We get the first few Schwarzians in $\mathbb{R}$,
  (besides ${\mathcal{S}}^{\,\vartheta}_1\,=\,1$, and $
{\mathcal{S}}^{\,\vartheta}_2\,=\,0$)
  using the method in Proposition \ref{Schwarzians-in-derivatives},
that we will need for some of the examples that follow.

We consider $x\,=\,0$ and  $f_1=f_1(t)$ and $f_2=f_2(t)$ as in (\ref{IC}).
Differentiate twice the identity  $f_1\,h\,=\,f_2$
to get, using (\ref{EQ}),
\begin{equation}\label{D}
0\,=\,2\,f_1^\prime \,h^\prime+f_1\, h^{\prime\prime}.
\end{equation}

Differentiate (\ref{D}), use (\ref{DEQA}) and  get
\begin{eqnarray}\label{D1}
0
&=&\left(-2\,\vartheta  h^\prime+ h^{\prime\prime\prime}\right)\,f_1+
3 \,f_1^\prime \,h^{\prime\prime}\,;
\end{eqnarray}
evaluation at $x=0$ and  (\ref{IC}) gives
$
h^{\prime\prime\prime}(0)\,=\,2\,\vartheta(0)
\Rightarrow
\fbox{$\displaystyle
     {\mathcal{S}}^{\,\vartheta}_3\,=\,2\,\vartheta
     $}\,.
$

Differentiate (\ref{D1}); by (\ref{EQ}),
\begin{eqnarray}\label{D2}
0
&=&\left(-2\,\vartheta^\prime \,h^\prime -5\,\vartheta \,h^{\prime\prime} + h^{(4)}
\right) \,f_1 +
\left(-2\,\vartheta \, h^\prime+4
 h^{\prime\prime\prime}\right)\, f_1^\prime;
\end{eqnarray}
set  $x\,=\,0$; by (\ref{IC})
$
h^{(4)}(0)\,=\,2 \,\vartheta^\prime(0)
\Rightarrow
\fbox{$\displaystyle
{\mathcal{S}}^{\,\vartheta}_4\,=\,2 \, \vartheta^\prime
$}\;.
$

Differentiate (\ref{D2}), use (\ref{EQ}) and set $x\,=\,0$; by  (\ref{IC}) and  $h^{\prime\prime\prime}(0)\,=\,2\,\vartheta(0)$ obtained earlier, get
$
h^{(5)}(0)\,=\,2\, \vartheta^{\prime\prime}(0)+16\, \vartheta^2(0)
\Rightarrow
\fbox{$\displaystyle
{\mathcal{S}}^{\,\vartheta}_{5}\,=\,2\,
\vartheta^{\prime\prime}+16 \, \vartheta^2$}\;.
$
\vskip.3cm
In a similar fashion
we get\,

\centerline{$\fbox{$\displaystyle
  {\mathcal{S}}^{\,\vartheta}_{6}\,=\,2\,\vartheta^{\prime\prime\prime}
+52\,\vartheta\, \vartheta^\prime
   $}$ }

   and

   \centerline{\,$\fbox{$\displaystyle
    {\mathcal{S}}^{\,\vartheta}_{7}\,= \,2\,\vartheta^{(4)}
         +76\, \vartheta^{\prime\prime} \,\vartheta
         +52\, \left(\vartheta^\prime\right)^2
       +272\, \vartheta^3
        $}$\,}.
\subsection{The Schwarzians
 $\mathcal{S}^{\,\vartheta,\,g_R}_{n}$
 for $g_R(z)\,=\,-\frac{R}{\pi}\ln{(1-\frac{\pi\,z}{R})}$}\
 \vskip.2cm
We know make a choice of $g_R$ and compute
$\mathcal{S}^{\,\vartheta,\,g_R}_{n}$
in terms of the $\mathcal{S}^{\,\vartheta}_{i}$\,.
\begin{Prop}\label{PROP-GR-LN-in-R-}
For and given $R\in (0,\,\infty]$ consider
\begin{equation}\label{GR-PICK-LN}
        \fbox{$\displaystyle
        g_R(z)\,=\,-\frac{R}{\pi}\ln{(1-\frac{\pi\,z}{R})}
          $},
\end{equation}
defined
in $\mathbb{C}\,\setminus{[\,\frac{R}{\pi}, \,\infty\,)}$
with $-\pi<\mathrm{arg} (1-\frac{\pi\,z}{R})<\pi$.
 For all
integer $n\,\geq 1$ the $g_R$-Schwarzians in $\mathbb{R}$ are
\begin{equation}\label{GNKNEWTON}\fbox{$\displaystyle
{\mathcal{S}}^{\,\vartheta, \,g_R}_{n}
\,=\,
\sum_{k=1}^n
 \; \left[\left[\begin{array}{c}n\\k\end{array}\right]\right]\,\,
 \left({\frac{\,\pi}{R}}\right)^{n-k} \;{\mathcal{S}}^{\,\vartheta}_k$}\,
\end{equation}
where\footnote{
$\left[\left[\begin{array}{c}n\\k\end{array}\right]\right]$
is
the absolute value of what is commonly known as the
\emph{$(n,k)$-Stirling number of the second kind}
\cite{ABRAMOWITZ}}
 $
\left[\left[\begin{array}{c}n\\k\end{array}\right]\right]\,=\,
\textrm{ coefficient of  } u^k \textrm{ in }
\prod_{s\,=\,0}^{n-1}(u+s)$. Moreover,

\centerline{\fbox{The expressions reduce to
${\mathcal{S}}^{\,\vartheta}_n$ for $R\,=\,\infty$}}\,.
\end{Prop}

\begin{proof}
That the expressions give ${\mathcal{S}}^{\,\sigma}_n$ for $R=\infty$
follows by construction, but can be verified directly since
$\left[\left[\begin{array}{c}n\\n\end{array}\right]\right]\,=\,1$.
Now, let $x \in \mathbb{R}$ arbitrary and fixed, and consider
\begin{equation}\notag
(T_x\circ{{g}_{R}})(t)\,=\,x-\frac{R}{\pi}\,\ln {(1-\frac{\pi\,t}{R})},
\end{equation}
which is interpreted $x+t$ for $R\,=\,\infty$\,.

Let $h_{x}$
 be the quotient of the pair of fundamental solutions associated to
 $x$ as in Proposition
 \ref{H-SCHWARZ}.
Then
\begin{eqnarray}\label{COMP-DER}
\;\;\;\;\;\;
{\mathcal{S}}^{\,\vartheta ,\,g_R}_n(x) \;\,
 &\overset{(i)}{=}&
\frac{\partial^{\,n}   }{\partial \,t^n}\,|_{t=0}\,
V \left( x,\, x+g_R(t)\right)\\\notag\\\notag
 &\overset{(ii)}{=}&
\left(h_{x}{\circ}\,T_{x}\,{\circ}\,{g}_{R}\right)^{(n)}(0)
\\\notag\\\notag
&\,=\,&
\sum_{k\,=\,1}^n
h_{x}^{(k)}(x) \; {g_R}_{n,k} \\\notag
&\overset{(iii)}{=}&
\sum_{k\,=\,1}^{n}
{\mathcal{S}}^{\,\vartheta}_k(x) \; {g_R}_{n,k}\,,
\end{eqnarray}
where (i) holds by
 (\ref{VH-SR}) and  (ii) by (\ref{VH});
the coefficients ${g_R}_{n,k}$
depend  on the derivatives  of order $1$ or higher
of ${g_R}$ at $0$ and are computed below; in (iii)
we use  Proposition \ref{H-SCHWARZ} by which
$h_{x}^{(n)}(x)\,
=\,\mathcal{S}^{\,\sigma}_{n}(x)$, for all integer $n\,\geq \,1$.

To find ${g_R}_{n,k}$, avoiding the Fa\`{a} di Bruno formula,
consider the function of $u$ and $v$,
$$r_{x}(u, v)\,=\,e^{\,u\frac{\pi}{R}(v-x)}\,,$$
and compute  the partial derivative with respect to $t$
 of order $n\,\geq\,1$ at $t=0$
 of
 $$r_{x}\left(u, x+{\,{g}_R} (t) \right)\,$$
 in two different ways.
 On the one hand that partial derivative
  is given in terms of the ${g_R}_{n,k}$ by
\begin{eqnarray}\label{D-C-1}
\sum_{k\,=\,1}^{n}
u^k\,\left(\frac{\pi}{R}\right)^k
\, {g_R}_{n,k}\,,
\end{eqnarray}
and
on the other hand, since
 $$r_{x}\left(u, \,x+{\,{g}_R} (t) \right)\,=\,
 \left(1-\frac{\pi\,t}{R}\right)^{-u},$$
that derivative computed directly gives
\begin{equation}\label{D-C-2}
\left(-\frac{\pi}{R}\right)^n \,\prod_{s\,=\,0}^{n-1}
(-u-s)
\,=\,\left(\frac{\pi}{R}\right)^n\,
\prod_{s\,=\,0}^{n-1}
(u+s)\,.
\end{equation}
Now the coefficient of $u^k$ in (\ref{D-C-1}) equals
the coefficient of $u^k$ in (\ref{D-C-2})
which
equals $\left(\frac{\pi}{R}\right)^{n}$ times the coefficient of $u^k$ in
$\prod_{s\,=\,0}^{n-1}
\left(u+s\right)$. Thus
$${g_R}_{n,k}\,=
\,\left[\left[\begin{array}{c}n\\k\end{array}\right]\right]\,
\left(\frac{\pi}{R}\right)^{n-k}\,.
$$
\end{proof}
\begin{Cor}\label{COR-GR-LN-COMP}
For the choice of $g_R$ given by (\ref{GR-PICK-LN}) the corresponding $g_R$-Schwarzian functions in $TM$ are given for $\mathrm{v}\in TM$
\begin{equation}\label{GNKNEWTON}\fbox{$\displaystyle
{\mathrm{S}}^{\,\sigma, \,g_R}_{n}(\mathrm{v})
\;=\;
\sum_{k=1}^n
 \; \left[\left[\begin{array}{c}n\\k\end{array}\right]\right]\,\,
 \left({\frac{\,\pi\,\|\mathrm{v}\|}{R}}\right)^{n-k}
 \;{\mathrm{S}}^{\,\sigma}_k(\mathrm{v})$}\,
\end{equation}
where ${\mathrm{S}}^{\,\sigma}_n(\mathrm{v})$ is given for $v\in TM$ by
(\ref{DEF-HOS-TM-PIN}).
\end{Cor}
\begin{proof}
Apply Definition \ref{DEF-S-R-in-TM}.
\end{proof}

\subsubsection{A few $g_R$-Schwarzians ${\mathcal{S}}^{\,\vartheta, \,g_R}_n$
for $g_R=-\frac{R}{\pi}\,\ln {(1-\frac{\pi\,t}{R})}
$ in terms of $\vartheta$ and $R$}

\begin{equation}\label{First-R-Schwarzians}{\mathcal{S}}^{\,\vartheta, \,g_R}_1
\,=\,1\,,\;\;\;
{{\mathcal{S}}^{\,\vartheta, \,g_R}_2}\,=\,
\frac{\pi}{R}\,,
\;\;\;{\mathcal{S}}^{\,\vartheta, \,g_R}_3\,=
\,{\mathcal{S}}^{\,\vartheta}_3\,+\, \frac{2\,\pi^2}{R^{2}}
\,=\,
2\,\vartheta\,+\, \frac{2\,\pi^2}{R^{2}}\,.
\end{equation}
Note that one recovers (\ref{LSZIR}) from
\begin{equation}\label{REM-D2-R}
          {\mathcal{D}}^{\,\vartheta, \,g_R}_2
\,=\,\det \left[  \begin{array}{cr}
          1&\frac{\pi}{2R}\\\\
           \frac{\pi}{2R}&\;\;
            \frac{\pi^2}{3\,R^2}
           +\frac{\,\vartheta }{3}
            \end{array}
\right]\,\geq\,0\,.
\end{equation}

In higher order,
\begin{eqnarray}\notag
{\mathcal{S}}^{\,\vartheta, \,g_R}_4\,&=&\,
{\mathcal{S}}^{\,\vartheta}_4+\,
{{6 \,{\mathcal{S}}^{\,\vartheta}_3\,}\frac{\pi}{R}}+\,6\,
\frac{\pi^{3}}{{R}^{3}}\\\notag
\,&=&\,2\,\vartheta^\prime+
12\,\vartheta\,\frac{\pi }{R} +\,6\,\frac{\pi ^{3}}{{R}^{3}},
\end{eqnarray}
and
\begin{eqnarray}{\mathcal{S}}^{\,\vartheta, \,g_R}_5\,&=&\,
{\mathcal{S}}^{\,\vartheta}_5+{10\,\mathcal{S}}^{\,\vartheta}_4\,\frac{\pi }{R}
+35\,{\mathcal{S}}^{\,\vartheta}_3\,\frac{\pi^2}{R^{2}}
+24\,\frac{\pi^4}{R^{4}}\\\notag
\,&=&\,
2\, \vartheta^{\prime\prime}+16 \, \vartheta^2
+20
\,\vartheta^\prime
\,\frac{\pi }{R}+70\,\vartheta\,\frac{\pi^2}{R^{2}}
+4\,\frac{\pi^4}{R^{4}}\,.
\end{eqnarray}
The corresponding determinantal condition in ${\mathcal{D}}^{\,\vartheta, \,g_R}_3\,\geq\,0$ on  $\mathbb{R}$, is explicitly computed,
\begin{equation}\label{REM-D3-R}
             \frac{\vartheta^3}{135}\,
             -\frac{\left(\vartheta^{\prime}\right)^2}{144}\,
             +\frac{\,\vartheta\,\vartheta^{\prime\prime}}{180}
             +\frac{(\vartheta^{\prime\prime}+
             8\,\vartheta^2)
              \,\pi^2}{720\,R^2}
             +\,\frac{\vartheta
             \,\pi^4}{240\,R^4}+\,
              \frac{\pi^6}{2160\,R^6}\,\geq\,0.
\end{equation}
We use this in the next result.

\begin{Prop}\label{PROP-SEC-OR-M}
 A  condition of second order in the Gauss curvature $\sigma$ necessary for the existence of the adapted complex structure up to
 radius $R$ is\footnote{ The weaker
diagonal condition
${\mathrm{S}}^{\,\sigma, \,g_R}_3\,\geq \,0$ is also valid and yields
$$
 \Delta\sigma +16 \, \sigma^2\,
+\,70\,\sigma\,\frac{\pi^2}{R^{2}}
+\,24\,\frac{\pi^4}{R^{4}}\,\geq\,0\,,$$ which in the limit $R\,=\,\infty$
reduces to  Sz\H{o}ke's inequality (\ref{SZI}).}
\begin{eqnarray}\label{INE-3-PROP}
  32\,\sigma^3
             +6\,\Delta{(\sigma^2)}
             +\left(3\,\Delta{\sigma} +
             48\,\sigma^2\right)
             \frac{\pi^2}{R^2}
             +18\,\sigma\,
             \frac{\pi^4}{R^4}+\,2\,
             \frac{\pi^6}{R^6}\,\geq\,27\,\|
             \operatorname{Grad}{\sigma}\|^2\,,&
             \end{eqnarray}
 \end{Prop}
 \begin{proof}
 In the expression on $\mathbb{R}$ of
 ${\mathcal{D}}^{\,\vartheta, \,g_R}_3$  given in (\ref{REM-D3-R}) use that
 $$\vartheta\,\vartheta^{\prime\prime}\,
   =\,\frac{1}{2}\,(\vartheta^2)^{\prime\prime}-\,{\vartheta^\prime}^2$$ to
  get
  \begin{eqnarray}\notag
             &&2160\,{\mathcal{D}}^{\,\vartheta, \,R}_3\,=\,
             16\,\vartheta^3
             -27\,{\vartheta^{\prime}}^2
             +6\,(\vartheta^2)^{\prime\prime}
             +\left( 3\,\vartheta^{\prime\prime}\,+
             24\,\vartheta^2\,\right)
             \frac{\pi^2}{R^2}
             +9\,\vartheta\,
             \frac{\pi^4}{R^4}+\,
             \frac{\pi^6}{R^6}\,.
             \end{eqnarray}

            The corresponding function in $TM$ restricted to $UM$
is obtained by replacing derivatives of $\vartheta$ by covariant derivatives
of $\sigma$. So, we have, in light of Theorem
\ref{TH-EXTEND-TUBE-FINITE}, for all $\mathrm{v}\in UM$, the unit tangent bundle,
 \begin{eqnarray}\label{EQ-INEQ-3-TM}
  16\,\sigma^3
             +6\,\nabla_\mathrm{v}^2(\sigma^2)
             +\left(3\,\nabla_\mathrm{v}^2\sigma +
             24\,\sigma^2\right)\,
              \frac{\pi^2}{R^2}
             +9\,\sigma\,
             \frac{\pi^4}{R^4}+\,
              \frac{\pi^6}{R^6} \geq\,27\,(\nabla_\mathrm{v}\sigma)^2&\,.
             \end{eqnarray}

Now, for each $x\in M$, average this inequality over any
 orthonormal basis $\{\mathrm{v_1}, \, \mathrm{v_2}\} $
of $T_xM$;
since  $\nabla_\mathrm{v_1}^2(\sigma^2)+
\nabla_\mathrm{v_2}^2(\sigma^2)\,=\,\Delta{(\sigma^2)}$
and $(\nabla_\mathrm{v_1}\sigma)^2+(\nabla_\mathrm{v_2}\sigma)^2
\,=\,
\|\operatorname{ Grad } \sigma\|^2$ then  (\ref{INE-3-PROP}) follows.
\end{proof}
\begin{Cor}
If  the adapted structure exists up of radius $R$  for  $M$ orientable and closed,
with $\mathrm{dA}$ the Riemannian area element, then
inequality (\ref{INTRO-INTEG}) holds.
\end{Cor}
             \begin{proof}
Since if $M$ is closed, by Stokes's Theorem the integrals of $\Delta\sigma$
and $\Delta{(\sigma^2)}$ vanish.
\end{proof}
\begin{Rem}
It is clear that in (\ref{INE-3-PROP} ) the equality  holds on $\mathbb{R}$ if $\sigma$ is constant
equal to $-\frac{\pi^2}{4\,R^2}$.
\end{Rem}

We display two more Schwarzians,
\begin{eqnarray}\notag
{\mathcal{S}}^{\,\vartheta,\,g_R}_{6}\,&=&\,
{\mathcal{S}}^{\,\vartheta}_{6}
+15\,{\mathcal{S}}^{\,\vartheta}_{5}\,\frac{\pi }{R}
+85\,{\mathcal{S}}^{\,\vartheta}_{4}\,\frac{\pi^2}{R^{2}}
+225\,{\mathcal{S}}^{\,\vartheta}_{3}\,\frac{\pi^3}{R^{3}}
 +120\,\frac{\pi^5}{R^{5}}\,\\\notag
\,&=&\,
      2\,\vartheta^{\prime\prime\prime}
      +52\,\vartheta\, \vartheta^\prime
    +(30\,\vartheta^{\prime\prime}+240\,\vartheta^2)\,\frac{\pi }{R}
     +{170\,\vartheta^\prime}\,\frac{\pi^2}{R^{2}}
   +450\,\vartheta\,\frac{\pi^3}{R^{3}}+120\,\frac{\pi^5}{R^{5}}\,,\notag
\end{eqnarray}
and
\begin{eqnarray}\notag
{\mathcal{S}}^{\,\vartheta, \,g_R}_7\,&=&
{\mathcal{S}}^{\,\vartheta}_{7}\,
+\,21\, {\mathcal{S}}^{\,\vartheta}_{6}\,\frac{\pi }{R}\,+\,175\, {\mathcal{S}}^{\,\vartheta}_{5}
\frac{\pi^2 }{R^{2}}
\,+\,735\,{\mathcal{S}}^{\,\vartheta}_{4}\,\frac{\pi^3 }{R^{3}}
+\,1624\,{\mathcal{S}}^{\,\vartheta}_{3}\,\frac{\pi^4 }{R^{4}}\,+\,720
\,\frac{\pi^6 }{R^{6}}\,,
\\\notag
&=&
2\,\vartheta^{(4)}  +76\,\vartheta\,\vartheta^{(2)}
 +52\, ( \vartheta^{\prime} )^{2}
 +272\,{\vartheta}^{3}+( 42\,\vartheta^{(3)} \,+1092\,\,\vartheta\,
\vartheta^{\prime} )\,\frac{\pi  }{R}\\\notag
&&\qquad\quad +(350\,\vartheta^{(2)} +2800\,{\vartheta}^{2})
\,\frac{\pi^2}{R^{2}} +1470\,\vartheta^{\prime}\,
\,\frac{\pi^3}{R^{3}}
+3248\,\vartheta\,\frac{\pi^4}{R^{4}}+
720\,\frac{\pi^6 }{R^{6}}\,.
\end{eqnarray}

\subsection{A formula for
$\,\mathcal{D}^{\,\vartheta,\,g_R}_{n}$ circumventing the Schwarzians}\
\vskip.2cm
 There is a classical determinantal formula for quadratic forms, that appears in \cite{AKHIEZER-KREIN}, which we will use  to find the functions
$\,\mathcal{D}^{\,\vartheta,\,g_R}_{n}$ involving directly the solutions of the diffential equation rather than the individual Schwarzians.
In fact
for real $x_1$ and $x_2$ with $|x_1|$ and $|x_2|$ small enough from a computation with power series, that uses
$(x_1^n-x_2^n)=(x_1-x_2)\sum_{k=0}^nx_1^{n-k}x_2^k$, and Proposition  \ref{H-SCHWARZ},
  \begin{eqnarray}\label{FLA-H-R-BEZOUTIAN}
   \frac{\left({h_{x}{\circ}T_{x}{\circ }\, g_R}\right)(x_1)\,-\,
  \left({h_{x}{\circ}T_{x}{\circ }\, g_R}\right)(x_2)}{x_1-x_2}
  \,=\,
   \sum_{p,\, q\,=\,0}^\infty\,
   \mathcal{S}^{\,\vartheta,\,g_R}_{i+j+1}(x)\,x_1^p\,x_2^q
   \,,\end{eqnarray}
   and hence we have the following.
\begin{Def}\label{BlockT}
For any smooth $f\,=\,f(t)$ and integers $n\,\geq\,1$ and $m\,\geq \,1$
let
$$\mathbf{T}(f,\,  n)=\mathbf{T}(f,\,  n)\,(t)$$ be the  $n\times 2\,n$ matrix
whose entries are the functions in the variable $t$ given by
\begin{equation}\label{DEFTIJ}
\mathbf{T}(f,\,   n)_{\,i,\,j}\,(t)\,=\,\begin{cases}\;\;
{\displaystyle 0 \; \textrm{ for } i\,<\,j;}\\\\
{\displaystyle \frac{f^{(j-i)}(t)}{(j-i)!}\; \textrm{ for } j\,\geq\,i\,.}
\end{cases}
\end{equation}
\end{Def}
\begin{Prop}\label{DET-HANKEL-MOEBIUS}
 Let $y_1$ and $y_2$ be any pair of independent
solutions of (\ref{DEQA})
(so the constant $W(y_1, y_2)=y_1\,y_2^\prime-y_1^\prime\,y_2\,\neq \,0$). Then
for all integer $n\,\geq\,1$ and all $x\in \mathbb{R}$
 \begin{equation}\label{Det-fla}\fbox{$\displaystyle
 \mathcal{D}^{\,{\vartheta,\,g_R}}_{n}(x)
\,=\,\frac{(-1)^{\frac{n(n+1)}{2}}}{W^n(y_1,\,y_2)}\,
\det{\left[
\begin{array}{cc}
\mathbf{T}(y_{1}{\circ}T_{x}{\circ }\, g_R, \, n)\,(0) \\\\
\mathbf{T}(y_{2}{\circ}T_{x}{\circ }\, g_R, \,  n)\,(0)
\end{array}
\right]}$}\,,
\end{equation}where
the $2n\times2n$ matrix whose
determinant is indicated above is  given as a pair of $n\times 2n$ blocks
each one as in Definition \ref{BlockT}.
\end{Prop}

\begin{proof}
On the one hand, if $f=f(t)$ and $g=g(t)$ are smooth, and
 $a,b,c,d$ any constants,
\begin{equation}\label{Det-abcd}
\det
\left[
\begin{array}{cc}
\mathbf{T}(a\,f+b\,g, \, n)(t) \\\\
\mathbf{T}(c\,f+d\,g, \,  n,)(t)
\end{array}
\right]=(a\,d-b\,c)^n\;\det
\left[
\begin{array}{cc}
\mathbf{T}(f,\, n)(t) \\\\
\mathbf{T}(g, \,  n)(t)
\end{array}
\right]
\,;\end{equation}
for,
by Laplace formula, we can expand $\det
\left[
\begin{array}{cc}
\mathbf{T}(f,\, n)(t) \\\\
\mathbf{T}(g, \,  n)(t)
\end{array}
\right]$
as a sum of  terms each of which  consists of a product of $n$
minors of order $2$:  one minor formed   with elements from rows $1$ and $n$,
another one with elements from  rows $2$ and $n+1$, and so on, up  to one minor formed
from rows $n$ and $2\,n$; but all these minors are of the form
$$  W_{i,\,j}(f,\,g)\,
  =\,\frac{1}{i!\,j!}\;\det{\left[
       \begin{array}{cc}
                f^{(i)}&f^{(j)}
               \\g^{(i)}&g^{(j)}
       \end{array}\right]}\,,$$
and for them it holds that \,$
W_{i,\,j}(a\,f + b \, g , \,
c\,f + d \,g)
\,=\,(a\,d-b\,c)\, W_{i,\,j}(f,\,g)
$\,.
On the other hand, if $y$ is any solution of (\ref{DEQA})
and $x$ any point in $\mathbb{R}$,
$$y(t)=y(x)\,f_{x,1}(t)+y^{\prime}(x)\,f_{x,2}(t)\,.$$
Thus,
in light of (\ref{Det-abcd}), to prove (\ref{Det-fla}) it
is enough to show that given any $x\in \mathbb{R}$,
for all integer $n\,\geq\,1$, with notation as in Definition \ref{Def-FPS},
\begin{equation}\label{Det-KAPPA-x0}
\det
\left[
\begin{array}{cc}
\mathbf{T}(\,f_{x,1}{\circ}T_{x}{\circ}\,g_R, \, n)(0) \\\\
\mathbf{T}(\,f_{x,2}{\circ}T_{x}{\circ}\,g_R,\,  n)(0)
\end{array}
\right]
\,=\,
(-1)^{\frac{n(n+1)}{2}}\,\mathcal{D}^{\,{\vartheta},\,g_R}_{n}(x)
.\end{equation}
But (\ref{Det-KAPPA-x0}) is shown by
a classical computation with power series and (\ref{FLA-H-R-SCHWARZ}) as follows.

If  $f=f(t)$ and $g=g(t)$ are smooth functions, $w=w(t)$ is given by
$f\,w = g$ and $\equiv1$ represents  the constant function with value $1$,
then using $g^{(n)}(t)\,=\,\sum_{k\,=\,0}^n
\binom{n}{k} f^{(n-k)}(t)\,w^{(k)}(t)$ one verifies that
\begin{equation}\label{MW}
\left[
\begin{array}{cc}
\mathbf{T}(f, \, n)(0) \\\\
\mathbf{T}(g, \,  n)(0)
\end{array}
\right]
=\underset{(*)}{\underbrace{
\left[
\begin{array}{cc}
\mathbf{T}(\equiv1, \, n)(0) \\\\
\mathbf{T}(w,\,  n)(0)
\end{array}
\right]}}
\, \left[
\begin{array}{cc}
\mathbf{T}(f, \, n)(0) \\\\
\mathbf{T}(x^n\,f, \,  n)(0)
\end{array}
\right]\,.
\end{equation}
For example, for $n\,=\,2$,
$$\left[
\begin{array}{cccc}
f(0)&f^{\prime}(0)
&\frac{f^{\prime\prime}(0)}{2!}
&\frac{f^{\prime\prime\prime}(0)}{3!}
\\\\
0&f(0)
&f^{\prime}(0)
&\frac{f^{\prime\prime}(0)}{2!}
\\\\
g(0)&g^{\prime}(0)
&\frac{g^{\prime\prime}(0)}{2!}
&\frac{g^{\prime\prime\prime}(0)}{3!}
\\\\
0&g(0)
&g^{\prime}(0)
&\frac{g^{\prime\prime}(0)}{2!}
\end{array}
\right]=
\left[
\begin{array}{cccc}
1&0
&0
&0
\\\\
0&1
&0
&0
\\\\
w(0)&w^{\prime}(0)
&\frac{w^{\prime\prime}(0)}{2!}
&\frac{w^{\prime\prime\prime}(0)}{3!}
\\\\
0&w(0)
&w^{\prime}(0)
&\frac{w^{\prime\prime}(0)}{2!}
\end{array}
\right]\,.\left[
\begin{array}{cccc}
f(0)&f^{\prime}(0)
&\frac{f^{\prime\prime}(0)}{2!}
&\frac{f^{\prime\prime\prime}(0)}{3!}
\\\\
0&f(0)
&f^{\prime}(0)
&\frac{f^{\prime\prime}(0)}{2!}
\\\\
0&0
&f(0)
&f^\prime(0)
\\\\
0&0
&0
&f(0)
\end{array}
\right]\,.
$$
Now, note that
the lower right $n\times n$ block of
(*) in (\ref{MW}), that is
the right  $n\times n$ block in
$$\\
\mathbf{T}(w,\,  n)(0)\,=\,
\left[\begin{array} {ccc}
w(0)
&\,&
\frac{w^{(n-1)}(0)}{(n-1)!}
\\
\, &\ddots &\,\\
 0 &\,&w(0)
\end{array} \,
      \big{|}
      \begin{array} {ccc}
   \frac{w^{(n)}(0)}{n!}
   &\,&
  \frac{w^{(2\,n-1)}(0)}{(2\,n-1)!}
   \\
   \, &\ddots &\,\\
   w^{\prime}(0)&\,&\frac{w^{(n)}(0)}{n!}
 \end{array} \right] \,,$$
is up to row permutation the $n\times\,n$ the Hankel matrix of
$w=w(t)$ at $t\,=\,0$, the
permutation needed being the one that interchanges
rows $n-k$ and $k$  of that block for every $k\,=\,1,\ldots, n-1$
which has a parity with sign  $(-1)^{\frac{n(n+1)}{2}}$\,.
Thus, taking determinants
in (\ref{MW})  with $$f\,=\,f_{x,1}{\circ}T_{x}{\circ}\,g_R\,,\;\;\;\;\;
g\,=\,f_{x,2}{\circ}T_{x}{\circ}\,g_R\,,\;\;\;\;\;
w\,=\,h_{x}{\circ}T_{x}{\circ}\,g_R$$
we see, by \ref{FLA-H-R-SCHWARZ}, that the resulting right-hand
side agrees with right-handide of equation (\ref{Det-KAPPA-x0}),
and by construction, so do
the left-hand sides.
\end{proof}


\subsection{Schwarzians in  $\mathbb{R}$ and in $TM$
for constant curvature.}\label{REM-K}\
\vskip.2cm
When the Gauss curvature $\sigma$ is constant equalt to $K$ on $M$, for each integer $n\geq 1$  the
function
$
\mathrm{v}\in TM \rightarrow\,
\displaystyle \frac{\mathrm{S}^{\,\sigma,\,g_R}_n(\mathrm{v})}{n!}
$
 is constant along level sets of $g(\mathrm{v}, \mathrm{v})$. Across level sets, due to Proposition \ref{HAT-SIGMA-HOMOGENEOUS}, it varies according to the rule
$\displaystyle
\lambda \mathrm{v}\in TM\,\longrightarrow  \,\lambda^{n-1} \frac{\mathrm{S}^{\,\sigma,\,g_R}_n(\mathrm{v})}{n!}
\;.$

To compute these Schwarzians for $R\,=\,\infty$
consider
 $f^{\prime\prime}+K F=0$ and
 let the function $s,t\mapsto V(x,\,x+t)$ be as in  Definition \ref{HOS}\,.
We have the independent solutions $f_1=\cos (\sqrt{K}t)$ and $f_2=\frac{1}{\sqrt{K}}\sin(\sqrt{K}t)$, hence
$h(t)=\frac{1}{\sqrt{K}}\tan (\sqrt{K}t)$,  interpreted henceforth as
$h(t)=t$ for $K=0$, and
 $h(t)=\frac{1}{\sqrt{|K|}}\tanh(\sqrt{|K|}t)$ for $K<0$.
Since for
any $a$ and $b$ in $\mathbb{C}$,
$\displaystyle \tan{(a+b)}\,=\,
\frac{\tan{a}\,+\,\tan{b}}{ 1-\tan{a}\tan{b}}$,
and by (\ref{Schwarzian-shift}), we have
\begin{equation}\label{VCC}
\sum_{n\,=\,0}^\infty\mathcal{S}^{\,K}_{n}(x)\frac{t^n}{n!}\,
=\,V(x,\,x+t)\,
=\,
\frac{1}{\sqrt{K}}\tan(\sqrt{K}t)\,,
\end{equation}
independent of $x$ as expected.
 Thus, the Schwarzians in $\mathbb{R}$ are the constants,
 for integer $n\,\geq\,1$,
\begin{eqnarray}\label{SN-CONSTANT-K}
\mathcal{S}^{\,K}_{2n-1}(x)\,&=&\,\mathcal{S}^{\,K}_{2n-1}=
\frac{1}{n}\,\binom{4^{n}}{2}\; |{B_{2\,n}}| \,K^{n-1}\;,\;\;\;\;
\\\notag {\mathcal{S}}_{2n}(x)\,&=&\,0\,,
\end{eqnarray} with  $B_n$ \cite{ABRAMOWITZ}
 the Bernoulli numbers \footnote
{For instance, $B_0\,=\,1$, \;$B_1\,=\,-1/2$,\;
  $B_2\,=\,1/6$,\;
 $B_4\,=\,-1/30$, \;$B_6\,=\,1/42$, \; $B_8\,=\,-1/30$,\;
  $B_{10}\,=\,5/66$ \,and \,$B_{12}\,=\,-691/2730$\,. So,
  the first few non-vanishing Schwarzians are
 \begin{eqnarray}\notag
 \mathcal{S}^K_1\,=1,\;\;\;
\mathcal{S}^K_3\;=\;2\,K, \;\; \;\mathcal{S}^K_5\,=\,16 \,K^2, \;\;\;
\mathcal{S}^K_7 \,=\,272 \,K^3,
 \;\;\;\mathcal{S}^K_9\,=\,7936\, K^4, \; \;\; \mathcal{S}^K_{11}\,=\,353792\,{K}^{5}\,.
 \end{eqnarray}
} defined  by $\frac {x}{e^x-1}\,=\,\sum_{n=0}^\infty \,B_n\,\frac {x^n}{n!}$.

  According to Definition \ref{DEF-HOS-TM-PIN},  the  Schwarzians in  $TM$
  are obtained via the substitutions
in (\ref{SUBSTITUTE}) which when $M$ has constant curvature $\sigma\,=\,K$
reduce to only  $K\,\mapsto \, \|\mathrm{v}\|^2 \,K$; thus for all $
\mathrm{v}\in TM$ and integer $n\,\geq \,1$,
\begin{eqnarray}\label{SN-CONSTANT-K-TM}
\mathrm{S}^{\,{K}}_{2n-1}(\mathrm{v})\,&=&\,
\frac{1}{n}\,\binom{4^{n}}{2} \,K^{n-1}\,|B_{2\,n}|
\,\|\mathrm{v}\|^{2n-2}\;,\\\notag
 \mathrm{S}^{\,{K}}_{2n}(\mathrm{v})\,&=&\,0\,.
\end{eqnarray}
 \begin{Rem}Formulas (\ref{SN-CONSTANT-K-TM}) in one direction and $\mathrm{S}^{\,\sigma}_{4}(\mathrm{v})\,=\,2\,\nabla_\mathrm{v}\sigma$ for all $v\in UM$ in the other, show that
 $2n$-th Schwarzians $\mathrm{S}^{\,\sigma}_{2n}(\mathrm{v})\,=\,0$ for all $\mathrm{v}\in UM$ and all integer $n\,\geq\,1$
if and only if $\sigma$ is constant.
\end{Rem}

Continuing with $\sigma\,=\,K$ a constant, and $R=\infty$ , so $g_R$
is the identity map, we
compute the function
\begin{equation}\label{DKN}
\mathrm{D}^{\,\sigma,\,g_R}_{n}(\mathrm{v}) \,=\,\mathrm{D}^{\,K}_n(\mathrm{v})\,.
\end{equation}

For $K=0$ the value of (\ref{DKN}) is $1$ for $n=1$ and $0$ for $n\geq 2$.

On the other hand, for $\|\mathrm{v}\|\,=\,1$, the function $\mathrm{D}^{\,K}_n(\mathrm{v})$ is the determinant of the Hankel matrix
associated
to the MacLaurin coefficients of  $\frac{1}{\sqrt{K}}\tan{\sqrt{K}\,x}$ for $K>0$ and to those of $\frac{1}{\sqrt{|K|}}\tanh(\sqrt{|K|}x)$ for $K<0$. The values of these determinants are surely known, but unfortunately  we were not able to locate this information
in the literature.
However the following is available to us (Formula (2.1) \cite{LAVOIE}),
\begin{equation}\label{KNOWN-DET}
\det{\left[ \frac{1}{(i+j-1)!}\right]_{i,\,j\,=\,1}^n}\;
=\;(-1)^{n(n-1)/2}\prod_{k=1}^{n}\frac{(n-k)!}{(2\,n-k)!}\,.
\end{equation}
We use this formula and Proposition \ref{DET-HANKEL-MOEBIUS} to compute our determinants. We get the following.

\begin{Prop}
\begin{eqnarray}\label{DET-HANKEL-TANGENT}\fbox{$\displaystyle
\mathrm{D}^{\,K}_{n}(\mathrm{v})\,=\,
\left(2\,\|\mathrm{v}\|\sqrt{K}\right)^{n(n-1)}\;
\prod_{k\,=\,1}^{n}\frac{(n-k)!}{(2\,n-k)!}$}\,.
\end{eqnarray}
\end{Prop}
\begin{proof}
Due to (\ref{VCC}) and our definitions,
\begin{eqnarray}
{\mathrm{S}}^{\,K}_{k}(\mathrm{v})\,&=&\,
\frac{1}{\|\mathrm{v}\|\,\sqrt{K}}\,
\tan^{(k)}{\left(\|\mathrm{v}\|\sqrt{K}\,t\right)}|_{t=0}
\\\notag\\\notag
&\overset{(*)}{=}&2\,\left(2\,\sqrt{-1}\,\|\mathrm{v}\|\sqrt{K}\right)^{k-1}\;
\tanh^{(k)}{\left(\frac{t}{2}\right)}|_{t=0}
\\\notag \\\notag
&=&2\,\left(2\,\sqrt{-1}\,\|\mathrm{v}\|\sqrt{K}\right)^{k-1}\;
\left(M\left(e^t-1\right)\right)^{(k)}|_{t=0}\notag
\end{eqnarray}
where in (*) we perform  a linear change of variable involving $\sqrt{-1}$
(hence the chance to hyperbolic
tangent) and  $M$ is the Mo\"{e}bius transformation defined by
$$M(z)\,=\,\frac{z}{z+2}\,.$$
Then, in light of  Proposition \ref{DET-HANKEL-MOEBIUS}, we compute
\begin{eqnarray}
{\mathrm{D}}^{\,K}_{n}(\mathrm{v})\,&=&\,
\det\left[\frac{{\mathrm{S}}^{\,K}_{i+i-1}(\mathrm{v})}{(i+j-1)!}
\right]_{i,\,j\,=\,1}^n
\\\notag \\\notag
\,&=&\,
\frac{2^n\,\left(2\,\sqrt{-1}\,
\|\mathrm{v}\|\sqrt{K}\right)^{n\,(n-1)}
}{\left(\det{M}\right)^n}\,
\det{\left[ \frac{1}{(i+j-1)!}\right]_{i,\,j\,=\,1}^n}\,.
\end{eqnarray}
Now (\ref{KNOWN-DET}) together with
$\det {M}=2$ and $\left(\sqrt{-1}\,\right)^{n(n-1)}\,=\,(-1)^{n(n-1)/2}$
yields (\ref{DET-HANKEL-TANGENT}).
\end{proof}
\begin{Rem}\label{REM-RANK-CC}
 Thus for  $K\,\neq\,0$, on $\mathrm{v}\in TM\setminus{M}$, we have
 $\mathrm{D}^{\,K}_{n}(\mathrm{v})\neq 0$ for all integer $n\geq 1$,
and hence
$$
 \mathrm{Rank}
 \left[\frac{{\mathcal{S}}^{K}_{i+j-1}(\mathrm{v})}{(i+j-1)!}
 \right]_{i,\,j\,=1}^{\infty}\,=\,\begin{cases}
 \infty \textrm{ for } \mathrm{v}\in TM\setminus{M}\\\\
 1 \textrm{ for } \mathrm{v}\in M \end{cases}\,.
$$ Moreover, for $\mathrm{v}\in {TM}{\setminus} M $,
\begin{eqnarray}
\operatorname{sign}
\left(\mathrm{D}^{\,K}_{n}(\mathrm{v})\right)\,=\,
\left(\operatorname{sign}\,(K)\right)^{\frac{n(n-1)}{2}}\,,
\end{eqnarray}
and so for $K\,<\,0$
the infinite quadratic  defined by $\displaystyle \left[\frac{{\mathcal{S}}^{K}_{i+j-1}(\mathrm{v})}{(i+j-1)!}
 \right]_{i,\,j\,=1}^{\infty}\,
$
is non-degenerate but not positive-definite, preventing the adapted structure to be defined in the entire $TM$.
\end{Rem}
Now, let us consider briefly the case $R$ finite.
Due to (\ref{VCC}), taking  $g_R$ as in (\ref{GR-PICK-LN}) the Schwarzians
${\mathcal{S}}^{\,K, \,g_R}_{n}(x)$ at any $x\in \mathbb{R}$
are equal to ${\mathcal{S}}^{\,K, \,g_R}_{n}(0)$, the MacLaurin coefficients in $t$ of
\begin{equation}\notag
\,-\frac{1}{\sqrt{K}}
\tan\left(\,\frac{\sqrt{K}\,R}{\pi}\,\ln {\left(1-\frac{\pi\,t}{R}\right)}\right)\,,
\end{equation}
which can  now be written down explicitly from the formulas (\ref{SN-CONSTANT-K-TM}) and those in Corollary \ref{COR-GR-LN-COMP}.

This way we get that the $g_R$-Schwarzians in $TM$ for constant curvature $K$ on $M$ are given, for $\mathrm{v}\in TM$,
\begin{equation}
\fbox{$\displaystyle
{\mathrm{S}}^{\,K, \,g_R}_{n}(\mathrm{v})
\;=\;\|\mathrm{v}\|^{n-1}\;
\sum_{l\,=\,1}^{\lfloor(n+1)/2\rfloor}
 \; \binom{4^{l}}{2}\;
 \left[\left[\begin{array}{c}n\\2\,l-1\end{array}\right]\right]\,\,
 \left({\frac{\pi}{R}}\right)^{n-2l+1}
 \;
\frac{|B_{2\,l}|}{l}\, \,K^{l-1}$}\,,
\end{equation}
where $\lfloor x\rfloor$ indicates the largest integer smaller or equal to $x$\,.

For example, for $K=0$ we have, on $TM$, for all $n\geq 1$
$${\mathrm{S}}^{\,K\,\equiv\,0, \,g_R}_{n}(\mathrm{v})\,=\,(n-1)!\,
\left(\frac{\pi\,\|\mathrm{v}\|}{R}\right)^{n-1}\,,$$
and hence we compute,
\begin{eqnarray}
{\mathrm{D}}^{\,K\,\equiv\,0,\,g_R}_{n}(\mathrm{v})
&=&\det\left[\frac{ {\mathrm{S}}^{\,K,\,g_R}_{(i+j-1)}(\mathrm{v})}{(i+j-1)!}
\right]_{i,\,j\,=\,1}^n\\\notag \\\notag
&=&
\,\det\left[\frac{1}{i+j-1}
\left(\frac{\pi\,\|\mathrm{v}\|}{R}\right)^{i+j-2}\right]_{i,\,j\,=\,1}^n
\\\notag\\\notag
&=&\,\left(\frac{\pi\,\|\mathrm{v}\|}{R}\right)^{n(n-1)}
\,\det\left[\frac{1}{i+j-1}
\right]_{i,\,j\,=\,1}^n
\\\notag\\\notag
&=&\,\left(\frac{\pi\,\|\mathrm{v}\|}{R}\right)^{n(n-1)}
\,\frac{\left[1!\,2!\cdots(n-1)!\right]^3}{n!\,(n+1)!\cdots (2n-1)!}\,.
\end{eqnarray}
Here the value of the last determinant is classical (Part VII, Problem $4$ \cite{POLYA-SZEGO}), and the indicated matrix is known as the Hilbert matrix .

\section{Schwarzians in $TM$ as moments\,}\label{SEC-SCHW-AS-MOMENTS}
There are some straight-forward restatements  of Theorem \ref{TH-EXTEND-TUBE-FINITE} in terms of classical \emph{moment sequences} all based on viewing the $g_R$-Schwarzians on $TM$ as a
family of sequences parametrized by $\mathrm{v}\in TM$   which gives
the assignment
\begin{equation}\label{MAP-H-SEQ}
\mathrm{v}\in TM
\rightarrow\,
\left\{\frac{\mathrm{S}^{\,\sigma,\,g_R}_n(\mathrm{v})}{n!}
\right\}_{n=1}^\infty\,,
\end{equation}
together with the fact that the determinantal conditions
(\ref{INTRO-DET-R-INEQ}) characterize \emph{non-negative sequences in $(-\infty, \,\infty)$}, which are precisely those
 sequences  $\{c_1,\,c_2,\, \cdots \}$ expressible as  power moments
of a non-negative  Borel  measure \cite{AKHIEZER-CTM}, \cite{SHOHAT-TAMARKIN}, \cite{WIDDER-LT}.
$$c_n\,=\,\int_{-\infty}^{\infty} t^{n-1} d\mu(t)\,.$$

  Thus, for each $\mathrm{v}\in TM$ we have a classical moment problem concerning the sequence in (\ref{MAP-H-SEQ}) which is  solvable for all $\mathrm{v}$ in $TM$ precisely when the adapted complex structure is defined in $T^RM$, and this happens whenever
the
 Schwarzian  functions are expressible as
 the integrals
\begin{equation}\label{INTRO-MOMENT-1}
\fbox{$\displaystyle
 \,{\mathrm{S}}^{\,\sigma, \,g_R}_{n}(\mathrm{v})\,=\,n!\,
 \int_{-1/\rho_{\mathrm{v}} }
 ^{1/\rho_{\mathrm{v}}}\,t^{n-1}\,
 d\mu^{\sigma,\,g_R}_{\mathrm{v}}(t)$}\,
 \end{equation}
  with a non-negative Borel measure $d\mu^{\sigma,\,g_R}_{\mathrm{v}}(t)$
and
$\rho_{\mathrm{v}}\,=\,\rho^{\sigma,\,g_R}_{\mathrm{v}}\,\in (0,\,\infty]$ both depending on $\mathrm{v}$, and of course the choice of $g_R$\,.

 Note that the Gauss curvature  $\sigma$  is determined by the integral
above
with $n=3$ computed point-wise by choosing  for each $x\in M$ any non-zero $\mathrm{v}\in T_xM$\,, and that the  measure, at any $\mathrm{v}\in TM$  integrates to the value $1$ since
${\mathrm{S}}^{\,\sigma, \,g_R}_{1}(\mathrm{v})\equiv 1$.

 Similarly, the existence of the adapted structure on $T^RM$ is equivalent to
the positivity of a classical functional defined by sequences \cite{AKHIEZER-CTM}.

In fact,  introduce
the sets of functions
\begin{eqnarray}\notag
\mathcal{P}\,&=&\,\{\;\mathrm{p}\colon TM\times
\mathbb{R} \,\to \,\mathbb{R}\,, \textrm{ where } \,\,\mathrm{p}(\mathrm{v},t)\textrm{ is polynomial in } t\}\,;
\\\notag
\mathcal{P}_+\,&=&\,\{\;\mathrm{p}\,\in\,\mathcal{P} \,,\,\mathrm{p}(\mathrm{v},t)\,\geq \,0 \,,\,\,\textrm{ for all } (\mathrm{v},t)\in TM\times\mathbb{R}\}\,.
\end{eqnarray}
Thus $\mathrm{p}$ is in $\mathcal{P}$ if and only if
$$\mathrm{p}(\mathrm{v},t)\,=\,\sum_{k\,=\,0}^n\,a_k(\mathrm{v})\,t^k$$ for some non-negative integer $n$
where $a_k\colon TM\to \mathrm{R}$ have no particular regularity assumption imposed on them.
Also let
\begin{eqnarray}\notag
\mathcal{F}\,&=&\,
\{\mathrm{f}\colon TM\,\to\, \mathbb{R} \}\,, \\\notag
\mathcal{F}_+\,&=&\,
\{\mathrm{f}\,\in\,\mathcal{F}\,,\,
\mathrm{f}(\mathrm{v})\,\geq\,0\,,\,\,\mathrm{v}\in TM  \}\,,
\end{eqnarray}
where no regularity is required here either.

Consider the map
\begin{equation}\label{FUNCT-S}{\mathfrak{S}}^{\,\sigma,\, g_R}\colon \mathcal{P}\to\mathcal{F}
\end{equation}
obtained
by extending  $\mathcal{F}$-linearly the  assignments
\begin{equation}\label{FUNCT-S-0}{\mathfrak{S}}^{\,\sigma,\, g_R}[t^{n-1}](\mathrm{v})\,
:=\,\frac{{\mathrm{S}}^{\,\sigma,\,g_R}_n(\mathrm{v})}{n!}\,, \,\,\, n\,=\,1,\,2,\,\cdots\,\,.
\end{equation}
Since for all $\mathrm{v}\in TM$ by definition ${\mathfrak{S}}^{\,\sigma,\, g_R}[t^0]\,=\, {\mathrm{S}}^{\,\sigma,\, g_R}_1(\mathrm{v})\,\equiv\,1$, we have always
$$ {\mathfrak{S}}^{\,\sigma,\,g_R}
\left[\mathcal{P}\right]\,=\,\mathcal{F}\,.$$
On the other hand, the following holds, by Theorem \ref{TH-EXTEND-TUBE-FINITE}.
\begin{Cor}\label{COR-LIN-POS}
 The adapted complex structure is defined on $T^RM$
if and only if
$${\mathfrak{S}}^{\,\sigma,\,g_R}\left[\mathcal{P}_+\right]
\,=\,\mathcal{F}_+\,.$$
\end{Cor}
In such case, for any $p\in \mathcal{P}$ and for all $\mathrm{v} \in TM$,
$$\mathfrak{S}^{\,\sigma, \,g_R}[p](\mathrm{v})\,=\,
\int_{-1/\rho^{\sigma,\,g_R}_{\mathrm{v}} }
 ^{1/\rho^{\sigma,\,g_R}_{\mathrm{v}}}\;p(\mathrm{v},\,t)\;
 d\mu^{\sigma,\,g_R}_{\mathrm{v}}(t)
 $$

\subsection{Estimate for the zeros of ${\mathrm{P}}^{\,\sigma, \,g_R}_n(\mathrm{v}, t)$.}\
\vskip.2cm
One final comment on this transplantation of the moment problem to $TM$ concerns the  expressions
\begin{equation}\label{ORTH-POLY-R}
{\mathrm{P}}^{\,\sigma, \,g_R}_n(\mathrm{v}, t)\,:=\,\det\left[\frac{{\mathrm{S}}^{\,\sigma, \,g_R}_{i+j}
(\mathrm{v})}{(i+j)!}\,-\,\frac{{\mathrm{S}}^{\,\sigma, \,g_R}_{i+j-1}
(\mathrm{v})}{(i+j-1)!}\,t\right]_{i,\,j \,=\,1}^{n} \, \in \; \mathcal{P},
\end{equation}for integer $n\,\geq \,1$.
When the support of measure at $\mathrm{v}\in TM$ does not consist of a finite number of points, that is when $\mathrm{R}^{\sigma,\,g_R}(\mathrm{v})\,=\,\infty$ (see \ref{DEF-RANK-TM}),  this infinite set of polynomials
constitute along the set  $\{\phi_t\mathrm{v},\;\; t\in \mathbb{R}\}\subset TM$ a point-wise orthogonal system with respect to ${\mathfrak{S}}^{\,\sigma,\, g_R}$,
that is,
\begin{equation}\label{PRODUCT}
{\mathfrak{S}}^{\,\sigma,\, g_R}\left[\,
{\mathrm{P}}^{\,\sigma, \,g_R}_n{\mathrm{P}}^{\,\sigma, \,g_R}_m\,\right](\phi_t\mathrm{v})\,=\,
\begin{cases}
 \,\mathrm{D}^{\,\sigma, \,g_R}_{n-1}(\phi_t\mathrm{v})\,
 \mathrm{D}^{\,\sigma, \,g_R}_n(\phi_t\mathrm{v}) \;\;\textrm{ if } n\,=\,m\,,\;\\\\\;0 \;\;\textrm{ if } n\,\neq \,m\;,
 \end{cases}
 \end{equation}
 for all integers $n, m\geq 1$ and all $t\in \mathbb{R}$, or with respect to the measure above
$$
(\ref{PRODUCT})\,=\,
\int_{-1/\rho^{\sigma,\,g_R}_{\phi_t\mathrm{v}} }
 ^{1/\rho^{\sigma,\,g_R}_{\phi_t\mathrm{v}}}\,
  \mathrm{P}^{\,\sigma,\,g_R}_n(\phi_t\mathrm{v},\,t)\; \mathrm{P}_m(\phi_t\mathrm{v},\,t)
 \;
 d\mu^{\sigma,\,g_R}_{\phi_t\mathrm{v}}(t)\,.
 $$

Concerning the relation of the support of this  measures and  the Gauss curvature $\sigma$ on $M$ note that, by Remark \ref{onM}, the zero section $M\subset TM$  is mapped to the sequence $\{1,\,0,\cdots\}$,
 and correspondingly, for all $\mathrm{v}\in M\,\subset \,TM$,
 \begin{equation}\label{Z-S-DELTA}
 d\mu^{\sigma,\,g_R}_{\mathrm{v}}(t)\,=\,\delta(t),\;\;\;
 \rho^{\sigma,\,g_R}_{\mathrm{v}}\,=\,\infty\,,
  \end{equation}
Off $M$, note that
by Proposition \ref{HAT-SIGMA-HOMOGENEOUS},
 for any $\lambda$, $$\mathrm{\lambda\,v}\in TM \rightarrow\,\{\lambda^{n-1}\,
\mathrm{S}^{\,\sigma,\,g_R}_n(\mathrm{v})/n!\}_{n=1}^\infty\,,$$
 so, $(\ref{MAP-H-SEQ})$ and the measure is determined by its restriction to the unit tangent bundle $UM$. Along this parameter set the sequence and the measure vary, unless the Gauss curvature $\sigma$ is constant on $M$.

For $\mathrm{v}\in UM$,
and for
$\displaystyle g_R(z)\,=\,-\frac{R}{\pi}\,\ln \left(1-\frac{\pi\,z}{R}\right)$, from the computations on the real line $\mathbb{R}$ given in Proposition \ref{PROP-SUPP-EST}, together with the method in the proof of Theorem \ref{TH-EXTEND-TUBE-FINITE} by which we relate those
computations to a unit speed geodesic, we readily obtain estimates in terms of the curvature  $\sigma$ of  $M$ for the support of $d\mu^{\,\sigma,\,g_R}_{\mathrm{v}}(t)$,
 when $\sigma$ is bounded as in Proposition \ref{PROP-SUPP-EST}.

Now, since for each $\mathrm{v}$ the measure $d\mu^{\,\sigma,\,g_R}_{\mathrm{v}}(t)$
maybe constructed \cite{AKHIEZER-CTM} \cite{FREUND}, \cite{CHIHARA}  by quadratures as a limit of measures defined from the zeros of the sequence $\{\mathrm{P}^{\,\sigma,\,g_R}_n(\mathrm{v},\,t)\}_{n=\,1,\,\cdots}$.

\begin{Prop}[The zeros of ${\mathrm{P}}^{\,\sigma, \,g_R}_n(\mathrm{v}, t)$]\label{PROP-SUPP-M}
If the adapted structure is defined on $T^RM$,
for each $\mathrm{v}\in UM$, the unit tangent bundle, the $n$ zeros of the polynomial in $t$
$$
{\mathrm{P}}^{\,\sigma, \,g_R}_n(\mathrm{v}, t)$$
are simple and real, and  all lie in the interval
$\displaystyle
  \left[\,-\frac{\pi}{R}\;,\;\;\frac{\pi}{R}\,\right]$ as long as $\sup_{x\in M} \sigma\,\leq 0$, or,
  provided that
$\sup_{x\in M} \sigma\,>\,0$\,,
  in the interval
$\displaystyle
  \left[{ -\frac{\pi}{R-\,R\,e^{-\frac{\pi^2}{2\,R\sqrt{\sup \sigma}}}}\;,\;
  \;\frac{\pi}{\displaystyle R-\,R\,e^{-\frac{\pi^2}{2\,R\sqrt{\sup \sigma}}}}\,}\right]$
which if $R\,=\,\infty$ reduces to
the interval
 $\displaystyle \left[{\displaystyle  -\frac{2\,\sqrt{\sup{\sigma}}}{\pi}\;,\;
  \;\frac{2\,{\sqrt{\sup{\sigma}}}}{\pi}} \right]$,
 .
\end{Prop}
\begin{Rem}
As an illustration, for Gauss curvature $\sigma\equiv0$
and $g_R\,=\,-\frac{R}{\pi}\,\ln \left(1-\frac{\pi\,z}{R}\right)$ the assignment (\ref{FUNCT-S-0}) becomes
$$
{\mathfrak{S}}^{\,\sigma,\, g_R}[t^{n-1}](\mathrm{v})\,=\,
\frac{1}{n}\,
\left(\frac{\pi\,\|\mathrm{v}\|}{R}\right)^{n-1}\,,
$$ and thus (\ref{INTRO-MOMENT-1}) is satisfied with the measure given, for $\mathrm{v}\in TM\setminus M$, by
$$d\mu^{\,K\,\equiv\, 0, \,g_R}_{\mathrm{v}}(t)\,=\,
\begin{cases}
{\displaystyle \frac{R}{\pi\,\|\mathrm{v}\|} \,dt \;\;\textrm { if  }\;\;  0\,\leq \,t\,\leq\,
\frac{\pi\,\|\mathrm{v}\|}{R}}\,;\\\\
{\displaystyle\;\;
0 \;\;\textrm{ otherwise } }\,,
\end{cases}
$$
and by (\ref{Z-S-DELTA}) along $M\subset TM$.

Thus, along $\|\mathrm{v}\|=\frac{R}{\pi}$, the polynomials
 $\mathrm{P}^{\,K\,\equiv\, 0, \,g_R}_n(\mathrm{v}, t)$, orthogonal with respect to the measures above, agree up to constant factors with the shifted Legendre polynomials,
 $$P_n(x)=\frac{1}{2^nn!}((u^2-1)^n)^{(n)}|_{u=2x-1}\,.$$

For $\mathrm{v}\neq 0$, the  $\mathrm{P}^{\,K\,\equiv\, 0, \,g_R}_n(\mathrm{v}, t)$, have the variables rescaled  so that for $\mathrm{v}$ along $M\subset TM$ they degenerate to $\mathrm{P}^{\,K\,\equiv\, 0, \,g_R}_1(\mathrm{v}, t)\,=\,t$ and all the other polynomials equal to zero.
\end{Rem}

\section{Acknowledgements}

I introduced the main results in this paper at the KIAS/SNU International Conference in Symplectic and Complex Geometry in honor of Professor Dan Burns Jr. in his 65th Birthday,  that took place from February 28 to March 3,  2011, at the Gwanak Campus of the Seoul National University, Korea.
I am grateful to Professor Burns for introducing me to the topic  of geometric complexifications and for his continuous dedication to the subject.

It is a great pleasure to thank the organizers of that conference in Seoul, Professor Jeongseog Ryu, Professor
Chong-Kyu Han, and Professor
Jong Hae Keum, for the invitation and for their kind hospitality. I also take the opportunity to thank Dean Captain Brad Lima for facilitating my attendance.

\end{document}